\documentclass[12pt,reqno]{amsart}
\usepackage{amssymb}
\usepackage{amsmath}
\usepackage{amsthm}
\usepackage{eucal}
\usepackage{color}
\usepackage{amsfonts}
\usepackage[pdftex]{graphicx}
\usepackage[notcite,notref]{showkeys}
\usepackage[T1]{fontenc}

\newtheorem{theorem}{Theorem}[section]
\newtheorem*{theorem*}{Theorem}
\newtheorem{proposition}[theorem]{Proposition}
\newtheorem{remark}[theorem]{Remark}

\newtheorem{lemma}[theorem]{Lemma}
\newtheorem{corollary}[theorem]{Corollary}
\newtheorem{definition}[theorem]{Definition}

\numberwithin{equation}{section}
\newcommand{\be}{\begin{equation}}
\newcommand{\ee}{\end{equation}}
\newcommand{\bp}{\begin{proof}}
\newcommand{\ep}{\end{proof}}
\newcommand{\bel}{\begin{equation}\label}
\newcommand{\eeq}{\end{equation}}
\begin{document}
\title[Fractional Sch\"odinger equation]{On the  fractional Sch\"odinger equation with variable coefficients}
\author{C. E. Kenig}
\address[C. E. Kenig]{Department of Mathematics\\University of Chicago\\Chicago, Il. 60637 \\USA.}
\email{cek@math.uchicago.edu}
\author{D. Pilod}
\address[D. Pilod]{Department of Mathematics\\University of Berger\\Postbox 7800\\ 5020 Berger\\Norway.}
\email{Didier.Pilod@uib.no}

\author{G. Ponce}
\address[G. Ponce]{Department  of Mathematics\\
University of California\\
Santa Barbara, CA 93106\\
USA.}
\email{ponce@math.ucsb.edu}

\author{L. Vega}
\address[L. Vega]{UPV/EHU\\Dpto. de Matem\'aticas\\Apto. 644, 48080 Bilbao, Spain, and Basque Center for Applied Mathematics,
E-48009 Bilbao, Spain.}
\email{luis.vega@ehu.es}

\keywords{Nonlinear dispersive equation,  non-local operators }
\subjclass{Primary: 35Q55. Secondary: 35B05}

\begin{abstract} 
We study the initial value problem (IVP) associated to the semi-linear   fractional Sch\"odinger equation with variable coefficients.
We deduce several properties of  the anisotropic fractional elliptic operator modelling the dispersion relation  and use them to establish the local  well-posedness for the corresponding IVP. Also, we obtain  unique continuation results concerning the solutions of this problem. These are  consequences of uniqueness properties that we prove for the fractional elliptic operator with variable coefficients. \end{abstract}

\maketitle

\section{Introduction}

This work is mainly concerned with the  initial value problem (IVP) associated to  an  anisotropic fractional Schr\"odinger equation, i.e.
 \begin{equation}
 \label {fNLS}
\begin{cases}
 \begin{aligned}
 & i\partial_t u +(L(x))^{\alpha}u+P(u,\bar{u})=0,\;\;\;\;t\in\mathbb R,\;x\in \mathbb R^n,\\
 &u(x,0)=u_0(x)
 \end{aligned}
 \end{cases}
 \end{equation}
where $\alpha >0$,
\begin{equation}
\label{L}
L(x)v=Lv=-\partial_{x_k}(a_{jk}(x)\partial_{x_j}v)+c(x)v,
\end{equation}
is a non-negative elliptic operator,   and $P(z,\bar{z})$ is a polynomial or a regular enough function with $P(0,0)=\partial_zP(0,0)=\partial_{\overline {z}}P(0,0)=0$.

In the case where the operator modeling the dispersion relation  $L=L(x)$ has constant coefficients, its fractional powers \eqref{fNLS} have been considered  in several  physical contexts, for example :

- when $L=-\Delta$, $0< \alpha <1$, and $P=0$, \eqref{fNLS}  was used in \cite{La} to describe particles in a class of  Levi stochastic processes,

- when $L=-\Delta+1$, $\alpha=1/2$, and $P=0$, \eqref{fNLS} was derived as a generalized semi-relativistic (Schr\"odinger) equation, see \cite{Le} and references therein,

-  when $L=-\Delta$, $\alpha=1/2$ and $P(u,\bar{u})=\pm|u|^au,\,a>0,$ the model in \eqref{fNLS} is known as the half-wave equation, see \cite{BGLV}, \cite{GeLePoRa},

-  when $L=-\Delta$,  $\alpha=1/2$, and 
$$P(z,\bar z)=c_0|z|^2z+c_1z^3+c_3z{\bar z}^2+c_3{\bar z}^3,\;\;\,c_0\in\mathbb R,\;c_1,c_2, c_3\in \mathbb C, $$
 it appears in \cite{IoPu} on the analysis  of the long-time behavior of solutions to the water wave equation in $\mathbb R^2$, with $(-\partial_x^2)^{1/2}$ 
 modelling the dispersion relation of the linearized gravity water waves equations for one-dimensional interfaces,
 
 - when $L=-\Delta+1$ and $\alpha=1/2$, \eqref{fNLS} emerges  in the examination of  gravitational collapse, see \cite{ElSc}, \cite{FrLe} and references therein.
 \vskip.1in

 When the operator $L(x)=L  $ has constant coefficients, i.e. $L=-\Delta+c,\,c\geq 0$, the well-posedness of the IVP \eqref{fNLS} has been studied  in several publications, see for example \cite{CHHO}, \cite{HoSi}, \cite{KLR}, \cite{Le}, \cite{BoRi} and references therein.
 
Under appropriate assumptions on the ellipticity, regularity and the asymptotic behavior at infinity on the coefficients $a_{jk}(x)$'s and $c(x)$,  we shall prove that the IVP \eqref{fNLS} is locally well-posed in the classical Sobolev spaces $H^s(\mathbb R^n)=(1-\Delta)^{s/2}$ with $s$ depending on the assumptions on $a_{jk}(x)$'s,  $\alpha$ and $n$. More precisely, we shall assume :
\begin{equation}
\label{hyp1}
a_{jk}(x)=a_{kj}(x),\;\;\;\;j,k=1,...,n,\;\;x\in\mathbb R^n,
\end{equation}
$\,\exists\, \lambda>0\,$ such that $\forall x\in\mathbb R^n$
\begin{equation}
\label{hyp2}
a_{jk}(x)\xi_j\xi_k\geq \lambda \|\xi\|^2,\;\;\;\;\forall \xi\in\mathbb R^n,
\end{equation}
and
\begin{equation}
\label{hyp3}
c(x) \geq 0.
\end{equation}

Concerning the regularity of the coefficients $a_{jk}(\cdot)$'s and $c(\cdot)$ we define the hypothesis $J$  with $J\in \mathbb Z^+$
\begin{equation}
\label{J}
(J)\;\;\;
\begin{aligned}
\begin{cases}
&\;a_{jk}\in C^{J}(\mathbb R^n)\cap W^{J,\infty}(\mathbb R^n),\;\;\;j,k=1,...,n,\\
&\;c\in W^{J-1,\infty}(\mathbb R^n),
\end{cases}
\end{aligned}
\end{equation}
where
$$
W^{J,\infty}(\mathbb R^n)=\{f:\mathbb R^n\to\mathbb R\,:\,\partial_x^{\beta}f\in L^{\infty}(\mathbb R^n),\;\;|\beta|\leq J\}.
$$
Also, we need  the asymptotic assumption (flatness at infinity)
\begin{equation}
\label{asymp}
\lim_{R\to \infty}\;\sup_{\|x\|\geq R}\,\sum_{j,k=1}^n\;\| a_{jk}(x)-\delta_{jk}\|=0,
\end{equation}
where $\delta_{jk}$ is the Kronecker symbol.

 We shall assume that the non-linearity $P$ in \eqref{fNLS} satisfies: 
 \begin{equation}
 \label{hyp-P}
 P(z,\overline{z})=\sum_{N_1\leq |\beta|+|\gamma| \leq N_2} \,d_{\beta,\gamma}\,z^{\beta}\,\overline{z}^{\gamma},\;\;\;N_1\geq 2.
 \end{equation}
 
 \begin{theorem}\label{A1}

Let $\alpha \in (0,1)$. Assume \eqref{hyp1}-\eqref{hyp3}, \eqref{asymp}-\eqref{hyp-P}, and \eqref{J} where $J\in \mathbb N$,
$J\geq 2\kappa-1$, for some $\kappa\in \mathbb N$ with $2\kappa>n/2$. Then for any  $u_0\in H^s(\mathbb R^n)$, $s\in (n/2,2\kappa]$,  there exist $T=T(\|u_0\|_{s,2};n; N_1;N_2)>0$ and a unique strong solution $u=u(x,t)$ of the IVP \eqref{fNLS} such that
 $$
 u\in C([-T,T]:H^{s}(\mathbb R^n)\cap C^1([-T,T]:H^{s-2\alpha}(\mathbb R^n))\equiv X_{s}(T).
 $$
 Moreover the  (locally defined) map data $\to$ solution  is analytic from $H^{s}(\mathbb R^n)$ to $X_{s}(T)$.
 \end{theorem}
 
 \begin{remark}
 \vskip.1in
\begin{enumerate}


\item It will be clear from our proofs in Section 2 and Section 3 how to extend the results in Theorem \ref{A1} to all $\alpha>0$.

 \item  If the hypothesis of Theorem \ref{A1} are satisfied for  $s$ and $s'$ with  $s'>s$, then  the results hold in the same time interval $[-T,T]$ with $T=T(\|u_0\|_{s,2};n; N_1;N_2)>0$.
 
 \item In the case where the operator modeling the fractional dispersion  depends also on the time variable, i.e. $L=L(x,t)$, one needs to consider the family of propagators $\{U_{\alpha}(t',t'')\,:\,t',\,t^{''}\in\mathbb R\}$. These operators are defined as $U(t',t^{''})u_0=u_{u_0}(x,t^{''})$, where $u_{u_0}(x,t)=u(x,t)$ is the solution of the IVP
 \begin{equation}
 \label {time-dep}
 \begin{cases}
 \begin{aligned}
 & i\partial_t u +L(x,t)^{\alpha}u=0,\;\;\;\;t\in\mathbb R,\;x\in \mathbb R^n,\\
 &u(x,t')=u_0(x).
 \end{aligned}
 \end{cases}
 \end{equation}
 To construct the solution of the IVP \eqref{time-dep}, one needs to assume a strong stability version on the hypothesis \eqref{hyp1}-\eqref{asymp} on a time interval $[-T,T]$. 
 
 This paper is a first step in the study of the (local) well-posedness of the IVP associated to the non-linear equation
 $$
  i\partial_t u +L(x,t,u(x,t))^{\alpha}u+P(u,\bar{u})=0,\;\;\;\;t\in\mathbb R,\;x\in \mathbb R^n.
 $$
These problems will be considered in forthcoming works.

\item
With the estimates gotten here, one can also deduce the local well-posedness, in some classical Sobolev spaces $H^s(\mathbb R^n)$, of the IVP associated to the fractional dissipative equation

 \begin{equation}
 \label {par}
 \begin{cases}
 \begin{aligned}
 & \partial_t u =-L(x)^{\alpha}u+P((\partial^{\beta}_xu)_{{\beta}\leq [\alpha]}),\;\;\;\;t>0,\;x\in \mathbb R^n,\\
 &u(x,0)=u_0(x),
 \end{aligned}
 \end{cases}
 \end{equation}
where $u=u(x,t)$ is a real valued function, $[\cdot]$ denotes the greatest integer function, and $P(\cdot)$ is a polynomial on its variables.

 \end{enumerate}
 \end{remark}
 
 \vskip.1in
 Next, we shall consider the IVP 
 \begin{equation}
 \label {fNLSa}
\begin{cases}
 \begin{aligned}
 & \partial_t u =iL(x)^{\alpha}u+Q(u,\bar{u}, \nabla u,\nabla \bar{u}),\;\;\;\;t\in\mathbb R,\;x\in \mathbb R^n,\\
 &u(x,0)=u_0(x)
 \end{aligned}
 \end{cases}
 \end{equation}
where $\alpha \in (0,1)$ and $Q:\mathbb C^{2n+2}\to \mathbb C$ is a polynomial  with $Q(\vec 0)=\partial_{z_j}Q(\vec 0)=0$, for $j=1,2,..,2n+2$, i.e. 
for some $N_1,\,N_2\in\mathbb N,\;N_1\geq 2$
 \begin{equation}
\label{key-hyp-ener2}
Q(\vec z)=\sum_{N_1\leq |\beta|\leq N_2}\;a_{\beta}\,\vec z^{\,\beta},\;\;\;\;\vec z\in \mathbb C^{2n+2},\;\,\beta\in (\mathbb N\cup\{0\})^{2n+2}.
\end{equation}
In addition, we shall assume that 
\begin{equation}
\label{key-hyp-ener1}
\partial_{\partial_{x_j}v}Q(v,\bar v,\nabla v,\nabla \bar{v} )\;\;\;\;\;j=1,2,..,n,\;\;\;\;\text{is real}.
\end{equation}

\begin{theorem}\label{A1a}

Let $\alpha \in (0,1)$. Assume \eqref{hyp1}-\eqref{hyp3}, \eqref{asymp}-\eqref{hyp-P}, and \eqref{J} where $J\in \mathbb N$, $J\geq 2\kappa-1$, for some $\kappa\in \mathbb N$, with $2\kappa>n/2+6$. Then for  any  $u_0\in H^{s}(\mathbb R^n)$, $\,s$ an even integer such that $s\in (n/2+4, 2\kappa-2]$,  there exist $T=T(\|u_0\|_{s,2};n; N_1;N_2)>0$ and a unique strong solution $u=u(x,t)$ of the IVP \eqref{fNLSa} such that 
 \begin{equation}
 \label{loss}
 u\in C([-T,T]:H^{s}(\mathbb R^n)) \cap C^1([-T,T]:H^{s-2\alpha}(\mathbb R^n)).
  \end{equation}
 Moreover the (locally defined) map data $\to$ solution is continuous from $H^{s}(\mathbb R^n)$ to $C([-T,T]:H^{s}(\mathbb R^n))$.
 \end{theorem}
 
 \begin{remark}
 \vskip.1in
\begin{enumerate}

\item It will be clear from our proofs in section 2 and section 3 how to extend   the results in Theorem \ref{A1a}  to all $\alpha>0$.

\item The assumption that $s$ is an even integer appears when the operator $L^m, m\in \mathbb N$ is applied to the equation. More precisely, in Lemma \ref{2.9} we shall show that the operator $(L)^m=L..L$ $\,m$-times (with $(L)$ as in \eqref{L}) agrees with $L^m$ defined by the spectral calculus of the self-adjoint operator $L$.
This allows us to use the spectral properties of $L$ to commute $L^m,\,m\in\mathbb N,$ with the associated linear semi-groups appearing  in the proof as well as to compute $(L)^m$ of the non-linear term $Q(\cdot)$  in the energy estimate, computing their pointwise commutator and  performing integration by parts. The extension of these techniques to fractional powers of the operator $L$ will be considered in a future work. 

\item For $\alpha=1$ and $L=\Delta$, the local well-posedness of IVP \eqref{fNLSa}, under the hypothesis \eqref{key-hyp-ener1}, was considered in \cite{Kl}, \cite{Sh} and \cite{KlPo}. In addition, in these works, under appropriate assumptions on $n$ and $N_1$ (see \eqref{key-hyp-ener2}) the global existence of "small" solutions to the IVP \eqref{fNLSa}  was established.

\item Here, we are not trying to obtain the optimal values of $J$ and $s$ for which the result in Theorem \ref{A1a} holds.

\end{enumerate}

\end{remark}
The proofs of the  well-posedness results above rely on our study of the spectral properties of the operator $L=L(x)$.  Using that $L$ is an unbounded, positive, self-adjoint operator, we use the spectral theory to define its fractional powers. We prove that these powers have a natural domain in which they satisfy some key estimates. Among the crucial results one has the following:

\begin{theorem}
\label{2.1}
If the coefficients $a_{jk}$'s and $c$ satisfy the hypothesis \eqref{hyp1}-\eqref{hyp3}, \eqref{asymp} and \eqref{J}  with $\text{J}=1$, then for any $\alpha\in [0,1]$ it follows that
\begin{enumerate}
\item 
$ L^{\alpha}$ is self-adjoint and unbounded on $L^2(\mathbb R^n)$,
\vskip.05in
\item
the domain of $ L^{\alpha}$ is $H^{2\alpha}(\mathbb R^n)=(1-\Delta)^{-\alpha}L^2(\mathbb R^n)$,
\vskip.05in
\item
for any $\alpha\in [0,1]$ 
\begin{equation}
\label{r1}
\|f\|_{2\alpha,2}=\|(1-\Delta)^{\alpha}f\|_2\sim \|f\|_2+\|L^{\alpha}f\|_2,\;\;\;\forall f\in H^{2\alpha}(\mathbb R^n).
\end{equation}

\end{enumerate}

\end{theorem}

The extension of Theorem \ref{2.1} to $\alpha \in [1,\infty)$ will be addressed in Corollary \ref{abcd}. The proof of this extension will be given in full detail  for the values  $\alpha =\kappa,\,\kappa\in\mathbb N$. The general case will then  use complex interpolation between these values. As it was pointed out in the remarks to Theorem  \ref{A1a} we shall show that for any $\kappa\in\mathbb N$ the operator $L^{\kappa}$ defined by the functional calculus of $L$ is equal to $L...L$, i.e. the $\kappa$-composition of $L$ with itself.  Also,  we shall give an outline of a different proof of the estimate \eqref{r1} for $\alpha>1$. This is based on some commutator estimates and the results  in Theorem \ref{2.1}.

 \vskip.2in

 Next, we have the following results concerning unique continuation properties of solutions to the IVP associated to the equation \eqref{fNLS}. 
\begin{theorem}\label{A2}
Let $\,\alpha\in (0,1)$. Let $\,u_1,\,u_2$ be two solutions of the IVP associated to the equation  \eqref{fNLS}  provided by Theorem \ref{A1} with  $a_{jk}\in C^J(\mathbb R^n),\,j,k=1,..,n$, and $J\geq 2$ as in Theorem \ref{2.1}, such that
\begin{equation}
\label{class2}
u_1,\,u_2\in C([0,T]:H^{s}(\mathbb R^n))\cap C^1([0,T]:H^{s-2\alpha}(\mathbb R^n)).
\end{equation}
If there exist a non-empty open set $\Theta\subset \mathbb R^n$ and  $t^*\in [0,T)$ such that
\begin{equation}
\label{con2}
u_1(x,t)=u_2(x,t),\;\;\partial_tu_1(x,t)=\partial_tu_2(x,t),\;\;\;(x,t)\in \Theta\times\{t^*\},
\end{equation}
then $\,u_1(x,t)=u_2(x,t)$ for all $(x,t)\in\mathbb R^n\times [0,T]$.

\end{theorem}

 \begin{remark}
 \vskip.1in
\begin{enumerate}
\item 
It will be clear from our proof below that, under the appropriate hypothesis, the results in Theorem \ref{A2} can be extended to all values of $\alpha\in (0,\infty) \,\text{-} \,\mathbb N$.
\item
If $\alpha$ is an  integer, i.e. the operator modeling the dispersion is local, the result in Theorem \ref{A2} fails. In this case, $\alpha$ an  integer, the equation and the first hypothesis in \eqref{con2} imply the second one there.  The weak  hypothesis for unique continuation in non-local models was already established in \cite{KPV19} for the Benjamin-Ono  equation, in \cite{LiPo} for the Camassa-Holm equation (and related models), in \cite{HoPo} for the BBM equation, and in \cite{KPPV} for the case of fractional powers of the Laplacian and of $ (1-\Delta)$.

\item 
Theorem \ref{A2} also applies to solutions of the IVP \eqref{fNLSa} obtained in Theorem \ref{A1a}.

\item
The unique continuation argument for these non-local equations is based on stationary techniques, i.e. for a fixed time $t^*\in[0,T)$, different from  those used for the local ones. For the latter ones, the main tool are  Carleman estimates involving space and time, and where the hypothesis \eqref{con2}  is required in an open set in $\mathbb  R^n\times [0,T]$, see \cite{SaSc}, \cite{Iz}.
\end{enumerate}
\end{remark}

The proof of Theorem \ref{A2} is a direct consequence of a Global Unique  Continuation Property (GUCP) for the operators  $L^{\alpha}=(L(x))^{\alpha}$ with $\,\alpha\in(0,1)$:

\begin{theorem} {\label{A3}} 
 Let $\alpha\in (0,1)$ and $\,f\in H^{2\alpha}(\mathbb R^n)$. Assume that the coefficients of $L$, $a_{jk}$'s and $c$ satisfy the hypothesis \eqref{hyp1}-\eqref{hyp3}, \eqref{asymp} and \eqref{J}  with $\text{J}=1$, and in addition $a_{jk}\in C^2(\mathbb R^n),\,j,k=1,..,n$. If there exists a non-empty open  set $\Theta\subset \mathbb R^n$ such that 
\begin{equation}
\label{hyp7}
(L)^{\alpha}f (x)=f(x)=0,\;\;\;\;\text{ for all}\,\;\;x\in\Theta,
\end{equation}
 then $f\equiv 0$ in $\mathbb R^n$.

\end{theorem}

 \begin{remark}  
 
 \begin{enumerate}
 \item
 The regularity assumption on $f$ in Theorem  \ref{A3} could perhaps be lowered to  $s = \alpha$. In this case, one would need to define a weak solution of the extension problem \eqref{prop:U} with Dirichlet boundary $u \in H^{\alpha}(\mathbb R^n)$ and make sense of the identity \eqref{eq:U} in $H^{-\alpha}(\mathbb R^n)$, (see Remark \ref{weak:sol:Dirichlet} for more comments).

\item  
Theorem \ref{A3} is also true if $L$ is defined on a smooth compact manifold, (see Proposition 1.4 in \cite{RulandTAMS2016} for the corresponding Carleman estimate), see also \cite{FKU}.

\end{enumerate}
\end{remark}
In the constant coefficient homogeneous case, \textit{i.e.} $a(x) \equiv a_0$ and $c(x) \equiv 0$, Theorem \ref{A3} is a classical result of M. Riesz, for smooth enough $f$, (see \cite{Ri}). If $f \in H^s(\mathbb R^n)$, with $s \ge \alpha$, Theorem \ref{A3}  is a consequence of the Weak Unique Continuation Property (WUCP) result proved by R\"uland in Proposition 2.2 in \cite{Ru1} for the Caffarelli-Silvestre extension problem \cite{CaSi} (see also \cite{FaFe,Yu}, \cite{CaSir}). This was formalized by Gosh-Salo-Uhlmann \cite{GhSaUh} by using the regularity theory for the extension problem in \cite{FKS} and then extended to rougher $f\in H^r(\mathbb R^n)$, for $r \in \mathbb R$,  (see Theorem 1.2 in \cite{GhSaUh}). In Theorem 1.6 in \cite{KPPV}, the current authors showed how to deduce the case of $-\Delta+1$ from the case of $-\Delta$, by adding one extra spatial dimension.


The case of homogeneous variable coefficients was also considered by R\"uland in \cite{Ru1}. She established the relevant Carleman estimate there and sketched  briefly the proof of the unique continuation for the associated extension problem with $C^2$ coefficients (see page 113  in \cite{Ru1} just before Section 7.2). In \cite{RulandTAMS2016}, R\"uland revisited the homogeneous variable coefficients case for $C^\infty$ coefficients and gave a full proof of the unique continuation property for the extension problem associated to operators with coefficients of this regularity. Here, for completeness, we will give a detailed proof of the unique continuation result for the extension problem for elliptic operators with $C^2$ coefficients as it was stated in Theorem \ref{A3} (see Proposition \ref{WUCP:ext} below). The proof relies on the doubling property deduced from R\"uland's Carleman estimate and a blow-up argument which allows to come back to the WUCP for the constant coefficients case, a strategy suggested by R\"uland's sketch in  \cite{RulandTAMS2016}.

We use this property to prove Theorem \ref{A3} for the operator $L$ in the case of variable coefficients. For a function $f \in Dom({L}^{\alpha})$, we use the Stinga-Torrea extension \cite{StiTorCPDE2010}, which generalizes the Caffarelli-Silvestre one \cite{CaSi}, to the upper-half space. Recalling the property $Dom({L}^{\alpha})=H^{2\alpha}(\mathbb R^n)$ that was proved in the first part of the paper (see Theorem \ref{2.1}) by using spectral theory, we show that the extension $U$ associated to a function $u \in H^{2\alpha}(\mathbb R^n)$ is a weak solution of the extended problem in the class $\dot{H}^1(\mathbb R_+^{n+1},y^{1-2\alpha}dxdy) \cap L^2_{loc}(\mathbb R_+^{n+1},y^{1-2\alpha}dxdy)$, where $y>0$ is the extension variable. This allows us to use the WUCP for the extension problem and deduce that $U$ has to vanish in a small half-ball $B_r^+$, for some $r>0$. We conclude then from the  classical WUCP (see \cite{AronJMPA1957,Cordes1956}) that $U$ must vanish in $\mathbb R^n \times (\epsilon,+\infty)$, for any $\epsilon>0$, which concludes the proof of Theorem \ref{A3} by letting $\epsilon \to 0$. 

Finally, let us mention that the analog of the unique continuation property in Theorem \ref{A3} fails for the fractional discrete Laplacian (see \cite{FeBeRoRu}).


\vskip.1in

The rest of this manuscript is organized as follows. Section 2 contains a review of spectral theory of unbounded operators, and the deduction of several estimates concerning the fractional powers and other functional calculus of the operator $L=L(x)$. In particular, it contains the proof of  Theorem \ref{2.1} and its extension to higher values of $\alpha$. This will allow us to obtain all the ingredients  needed in the proof of Theorem \ref{A1}, which is contained in section 3. The proof of  Theorem \ref{A3} will be given in Section 4.  Theorem \ref{A2} is  a direct consequence of this, therefore its proof will be omitted.

\subsection{Notation}

To conclude this introduction, we introduce the following notations:
$$
A \lesssim B\;\; \;\;\text{if}\;\;\;\;\exists \,c>0 \;\;\;\text{s.t.}\;\;\;\; A\leq cB.
$$
Similarly, $A \gtrsim B$, and 
$$
A\sim B\; \; \;\;\text{if}\;\;\;\;A \lesssim B\;\; \;\;\text{and}\;\;\;\;A \gtrsim B.
$$

\vskip.1in

\section{Spectral Theory for $L=L(x)$}

In this section, we shall establish some results concerning properties of the operator 
\begin{equation}
\label{L1}
Lv=L(x)v=-\partial_{x_k}(a_{jk}(x)\partial_{x_j}v)+c(x)v=\mathcal {L}v+c(x)v,
\end{equation}
and its fractional powers $L^{\alpha}$. We shall assume  the hypothesis \eqref{hyp1}-\eqref{hyp3}, \eqref{J} with $J=1$, and \eqref{asymp}.

We shall start by proving  Theorem \ref{2.1}. This is partly inspired by \cite{stri}. First, we shall recall some definitions and results concerning spectral theory of unbounded operators. These can be found in  \cite{re-si-I},  \cite{re-si-II}, \cite{Yo}, \cite{rudin}, and \cite{davis}.

\subsection{Proof of Theorem \ref{2.1} in the case $\alpha=1$}


We consider the case $\alpha=1$. 
Let $\mathcal D\equiv C^{\infty}_0(\mathbb R^n)$. Then $L$ is a non-negative definite symmetric operator on $\mathcal D$.  Let $L_{min}$ be the $L^2$-closure of $L$ on $\mathcal D$.

The domain of $L_{min}$, $D_{min}(L)$, is defined as
$$
D_{min}(L)=\{f\in L^2\,: \,\exists\, (f_j)\subseteq \mathcal D\,\,\text{s.t.} \;f_j\to f\;\text{and}\;Lf_j\to g  \;\text{in}\;L^2\},
$$
and we write
$$
g=L f.
$$

\vskip.05in

\underline{Claim 1:} $L$ defined on $D_{min}(L)$, i.e. $L_{min}$, is a closed operator.
\vskip.1in
We shall prove : if $\{f_k\,:\,k\in \mathbb N\} \subset D_{min}(L)$ with 
$$
f_k\to f \;\;\text{in}\;L^2\;\;\;\text{and}\;\;\;Lf_k\to g \;\text{in} \;L^2,\text{then}\;\;f\in D_{min}(L) \;\;\text{and}\;\,g=Lf.
$$
For each $f_k$ there exists $\{f_{k,j}\,:\,j\in\mathbb N\}\subset \mathcal D$ such that 
$$
f_{k,j}\to f_k,\;\;L f_{k,j}\to Lf_k\;\;\text{in}\;L^2\;\;\text{as}\,\,j\uparrow \infty.
$$

Thus, for each $k\in\mathbb N$ we choose $j=j(k)\in\mathbb N$ such that 
$$
\|f_{k,j(k)}-f_k\|_2+\|L f_{k,j(k)}-L f_k\|_2<\epsilon_k\;\,\text{with}\;\;\epsilon_k\downarrow 0\;\;\text{as}\;\;k\uparrow \infty.
$$
Therefore,
$$
\|f_{k,j(k)}-f\|_2\leq \|f_{k,j(k)}-f_k\|_2+\|f_{k}-f\|_2\to 0\;\;\text{as}\;\;k\uparrow \infty,
$$
and
$$
\|Lf_{k,j(k)}-g\|_2\leq \|L f_{k,j(k)}-L f_k\|_2+\|Lf_{k}-g\|_2\to 0\,\; k\uparrow \infty.
$$

By definition of $D_{min}(L)$  one has $\;f\in D_{min}(L)$ and $\,g=Lf$. 

Clearly, $L f=g$ in the sense of distributions, so $g$ is unique. More precisely,  if $f_l,\,\tilde{f_l}\in \mathcal D,\,l\in\mathbb N$, with $f_l\to f$ and $\tilde{f_l}\to f$ in $L^2$ with  $L(f_l)\to g$ and $L(\tilde{f_l})\to \tilde{g}$ in $L^2$, then $g=\tilde{g}$, since $L(f)=g$ in $\mathcal D'$. To see this, let $\varphi\in \mathcal D$, then
$$
\aligned
&\int(g-\tilde g)\varphi=\lim_{l\to \infty} \int L(f_l-\tilde{f_l})\varphi\\
&=\lim_{l\to \infty}(\int a_{jk}\partial_{x_k}(f_l-\tilde{f_l})\partial_{x_k}\varphi+\int c(f-\tilde{f_l}))\\
&=\lim_{l\to \infty}(-\int \partial_{x_k}a_{jk}(x)(f_l-\tilde{f_l})\partial_{x_k}\varphi
-\int a_{jk}(f_l-\tilde{f_l})\partial^2_{x_kx_j}\varphi\\
&+\int c(f_l-\tilde{f_l}))=0,
\endaligned
$$
since $f_l-\tilde{f_l}\to 0$ in $L^2(\mathbb R^n)$.

\vskip.05in
Define $D_{max}(L)$ as
$$
\{f\in L^2 : \text{the distribution}\,L(f) \text{ can be identified with a }L^2\,\text{function}\},
$$
i.e.
$$
D_{max}(L)=\{f\in L^2(\mathbb R^n) : \,L(f) \in L^2(\mathbb R^n)\}.
$$
Since $\mathcal D$  is dense in $L^2(\mathbb R^n)$, $L^*$ is the  the operator whose domain consists of all $f\in L^2(\mathbb R^n)$ s.t. 
$$
\exists \,g \in L^2(\mathbb R^n)\; \,\text{s.t.}\; \,\forall\; h\in D_{min}(L)\;\;\;\;\langle Lh,f\rangle=\langle h, g\rangle,
$$

Thus, $D_{max}(L)=D(L^*)$ is the largest possible domain where one can consider the operator $L$. Our aim is to show that $D_{min}(L)=D_{max}(L)$, i.e. $L$ is self-adjoint. To show this, we shall recall  the following  result in \cite{re-si-II}, pages 136-137, see also \cite{stri} :

\begin{lemma}\label{reg}
Let $A$ be any closed, negative definite, symmetric, densely defined operator on a Hilbert space. Then $A$ is self-adjoint, $A=A^*$, if and only if there are no eigenvectors with positive eigenvalues in the domain of $A^*$.
\end{lemma}

We apply this to $-L$ with
$$
-Lu=\partial_{x_k}(a_{jk}(x)\partial_{x_j}u)-c(x)u,
$$
which is closed in $D_{min}(L)$, symmetric, densely defined in $L^2(\mathbb R^n)$, and negative definite. We need to show that $-L$ has no positive eigenvalues in $D((-L)^*)=D_{max}(L)$. This means, if $u\in D_{max}(L)$ such that 
\begin{equation}
\label{self-adj}
-L u=\lambda u,\;\,\lambda>0, \;u\in L^2(\mathbb R^n),\;\text{then}\;\;u\equiv 0.
\end{equation}

To prove this we need the following result:

\begin{lemma}\label{AA}
Assume that $u\in L^2(\mathbb R^n)$ and that 
$$
\partial_{x_j}(a_{jk}(x)\partial_{x_k}u)=f,\;\;f\in L^2(\mathbb R^n)\;\;\text{in the distributional  sense},
$$
then $u\in H^1_{loc}(\mathbb R^n)$.
\end{lemma}

To obtain Lemma \ref{AA}, we will use the result established in \cite{GW}, concerning properties of the Green function for the operator $- \mathcal L $ with
\begin{equation}
\label{def-Lambda}
- \mathcal L =\Lambda =\partial_{x_j}((a_{jk}(x)\partial_{x_k} \cdot).
\end{equation}
in a bounded domain $\Omega\subset \mathbb R^n, \,n\geq 3$, with the coefficients $a_{jk}$'s satisfying \eqref{hyp1}, \eqref{hyp2} and \eqref{J} with $J=1$. The case $n=2$ requires minor modifications. This can be done using \cite{KN}.

\begin{theorem}[\cite{GW}]\label{green}\hskip15pt Let $\Omega \subset \mathbb R^n$ with $n\geq 3$ be a bounded domain. There exists a unique function
$$
G:\Omega\times \Omega \to \mathbb R \cup\{\infty\},
$$
such that : $G\geq 0$, for any $y\in\Omega$ and any $r>0$, $$G(\cdot,y)\in W^{2,1}(\Omega\,\text{-}\,B_r(y))\cap L^{1,1}(\Omega),
$$ 
where $L^{1,1}(\Omega)=\{f\in L^1(\Omega)\,:\,\partial_{x_j}f\in L^1(\Omega),\,j=1,..,n\}$, 
and for any $\varphi\in C^{\infty}_0(\Omega)$ 
$$
-\int_{\Omega}a_{jk}(x)\partial_{x_j}G(x,y)\partial_{x_k}\varphi(x)dx=\varphi(y),
$$
and for any $f\in H^{-1}(\Omega)$ supported in $\Omega$ 
$$
v(y)=\int_{\Omega}G(x,y)f(x)dx,
$$
$v\in H^1(\mathbb R^n)$ solves $\Lambda v=f$, (see  \eqref{def-Lambda}) in the weak sense (in fact, in the $H^1$-weak sense, see \eqref{defH1w}).

Moreover, there exists $c>0$, independent of $\Omega$, such that
$$
|G(x,y)|\leq \frac{c}{|x-y|^{n-2}},\;\;\;\text{and}\;\;\;|\nabla_xG(x,y)|\leq \frac{c}{|x-y|^{n-1}}.
$$
\end{theorem}

\begin{proof}[Proof of Lemma \ref{AA}]

Fix $x_0\in\mathbb R^n$, $|x_0|=R>0$. Choose $\theta\in C^{\infty}_0(\mathbb R^n)$ such that $\theta(x)=1$ for $|x|<2R$ and $supp(\theta)\subset\{x\in\mathbb R^n\,:\,|x|<4R\}$. Let $\eta\in C^{\infty}_0(\mathbb R^n)$,$\,\eta\geq 0$, $supp(\eta)\subset\{x\in\mathbb R^n:|x|\leq 1\}$, and $\int\eta(x)dx=1$. Define for $\epsilon>0$
$$\eta_{\epsilon}(x)=\frac{1}{\epsilon^n}\eta\big(\frac{x-x_0}{\epsilon}\big).$$

Let $G(x,y)$ be the Green function on $B_{10R}(0)$, and
$$
\tilde G(x,y)=\theta(x)G(x,y).
$$
Hence, for $\epsilon$ sufficiently small, one has
\begin{equation}
\label{def.varphi}
\begin{aligned}
\varphi_{\epsilon}(x)&=\int\tilde G(x,y)\eta_{\epsilon}(y)dy\\
&=\theta(x)\int G(x,y)\eta_{\epsilon}(y)dy=\theta(x)\tilde \eta_{\epsilon}(x).
\end{aligned}
\end{equation}

We observe  that $\varphi_{\epsilon}$ has compact support inside $\{x\in\mathbb R^n\,:\,|x|<4R\}$, and that $\tilde \eta_{\epsilon}, \varphi_{\epsilon}\in H^2_{loc}(\mathbb R^n)$ with
$$
\Lambda \tilde \eta_{\epsilon}=\partial_{x_j}(a_{jk}(x)\partial_{x_k} \tilde \eta_{\epsilon})=\eta_{\epsilon},
$$
in the $H^1$-weak sense, see \eqref{defH1w}. Moreover,
$$
\begin{aligned}
&\Lambda (\varphi_{\epsilon})=\Lambda(\theta\tilde \eta_{\epsilon})\\
&=\tilde \eta_{\epsilon}\Lambda \theta+2a_{jk}\partial_{x_k}\theta \partial_{x_j}\tilde \eta_{\epsilon} +\theta\eta_{\epsilon}\in L^2(\mathbb R^n).
\end{aligned}
$$
Hence, $\varphi_{\epsilon}\in H^2(\mathbb R^n)$, (see \cite{evans}, Theorem 6.3.1).

Assuming  that $u\in L^2(\mathbb R^n)$ is a weak solution of $\Lambda u=f$ (in the distribution sense), with $f\in L^2(\mathbb R^n)$, it follows that  for any  $\psi\in H^2(\mathbb R^n)$ with compact support one has 
\begin{equation}
\label{H2}
\int f(x)\psi(x)dx=\int u(x)\Lambda \psi(x)dx.
\end{equation}

In our proof, we choose $\psi$ in \eqref{H2} as $\varphi_{\epsilon}$  in \eqref{def.varphi} to get
\begin{equation}
\label{z1}
A_{\epsilon}=\int f\varphi_{\epsilon}=\int u\Lambda \varphi_{\epsilon}=B_{\epsilon}.
\end{equation}

By construction, using Fubini's theorem we see that
$$
A_{\epsilon}=\int \eta_{\epsilon}(y)\big(\int G(x,y)f(x)\theta(x)dx\big) dy =\int \eta_{\epsilon}(y) v(y)dy,
$$ 
with $v\in H^1(\mathbb R^n)$ and $\Lambda v= - f\theta$ in the $H^1$-weak sense, i.e. 
\begin{equation}
\label{defH1w}
\forall\, \chi\in H^1(\mathbb R^n) \;\;supp(\chi)\subseteq B_{10R}(0))\;\;\int a_{jk}\partial_{x_k}v\partial_{x_j}\chi=\int f\theta\chi.
\end{equation}
Hence, $v\in H^2_{loc}(\mathbb R^n)$, from  \cite{evans}, (Theorem 6.3.1), and 
\begin{equation}
\label{z0}
\lim_{\epsilon\downarrow 0} A_{\epsilon}=v(x_0),\;\;\;\text{ for a.e.}\,\;x_0.
\end{equation}

Next, we write

\begin{equation}
\label{B1-B3}
\begin{aligned}
&B_{\epsilon}=\int u\Lambda \varphi_{\epsilon}=\int u\Lambda (\tilde \eta_{\epsilon}\theta)\\
&=\int u\Lambda (\tilde \eta_{\epsilon})\theta+\int u\tilde \eta_{\epsilon}\Lambda (\theta)+2\int u a_{jk}\partial_{x_j}\tilde \eta_{\epsilon}\partial_{x_k}\theta\\
&=B_{1,\epsilon}+B_{2,\epsilon}+B_{3,\epsilon}.
\end{aligned}
\end{equation}

Thus,
$$
B_{1,\epsilon}=\int u\theta\eta_{\epsilon}dx\to u(x_0)\;\;\text{as}\;\;\epsilon\to 0,\;\;\;\text{ for a.e.}\,\;x_0.
$$
next by an argument similar to that in \eqref{z0}
$$
\aligned
&B_{2,\epsilon}=\int u(x)(\int G(x,y)\eta_{\epsilon}(y)dy)\Lambda \theta(x)dx\\
&=\int\eta_{\epsilon}(\int G(x,y)\Lambda (\theta)u(x) dx)dy\\
&=\int \eta_{\epsilon}(y)w(y)dy \to w(x_0),\;\;\text{as}\;\;\epsilon\to 0,
\endaligned
$$
with $w\in H^1(\mathbb R^n)$ by definition and $w\in H^2_{loc}$ since $\Lambda w=\Lambda (\theta) u\in L^2(\mathbb R^n)$. Finally
using a similar argument
$$
\aligned
&B_{3,\epsilon}=\int u(x)a_{jk}(x)\partial_{x_k}\theta(x)(\int\partial_{x_j}G(x,y)\eta_{\epsilon}(y)dy)dx\\
&=\int \eta_{\epsilon}(y)(\int a_{jk}(x)\partial_{x_j}G(x,y)\partial_{x_k}\theta(x)u(x)dx)dy\\
&=\int \eta_{\epsilon}(y)z(y)dy\to z(x_0),\;\;\text{as}\;\;\epsilon\to 0,\;\;\;\text{ for a.e.}\,\;x_0,
\endaligned
$$
with $z\in L^2(\mathbb R^n)$ vanishing in $\partial \Omega$ with $\Lambda z=-\partial_{x_k}(a_{jk}(x)\partial_{x_j}\theta u)$, therefore $\nabla z\in L^2_{loc}(\mathbb R^n)$. 

Collecting the above information we see that $u(x_0)$ is the sum of terms all of them with gradient in $L^2_{loc}(\mathbb R^n)$. This completes the proof of Lemma \ref{AA}.

\end{proof}

As a direct consequence of Lemma \ref{AA} one has the following result :

 \begin{corollary}\label{AA1}
Assume that $u, \,f\in L^2(\mathbb R^n)$ and 
$$
\partial_{x_j}(a_{jk}(x)\partial_{x_k}u)=f,\;\;\text{in the distributional  sense},
$$
then for any $h\in H^1(\mathbb R^n)$v with compact support one has
$$
-\int a_{jk}(x)\partial_{x_j}u(x)\partial_{x_j}h(x)dx= \int f(x)h(x)dx
$$
Hence,
$u\in H^2_{loc}(\mathbb R^n)$, by \cite{evans}, Theorem 6.3.1, and verifes $\Lambda u=f$ in the $H^1$-sense. 
\end{corollary}

\begin{corollary}\label{AA2}
Assume that $u, \,f\in L^2(\mathbb R^n)$ and 
$$
-Lu=\partial_{x_j}((a_{jk}(x)\partial_{x_k}u)-c(x)u=\lambda u,\;\;\;\;\text{in the distribution  sense},
$$
with $\lambda>0$, then $u\equiv 0$.
\end{corollary}

\begin{proof}[Proof of Corollary \ref{AA2}]

Applying Lemma \ref{AA} and Corrollary \ref{AA1} to
$$
-\mathcal Lu=(c(x)+\lambda)u=f,
$$ 
one has that $u\in H^2_{loc}(\mathbb R^n)$ and $u$ is a weak solution, in the $H^1$-weak sense, of $-Lu=\lambda u$. Then for  
$\theta\in C^{\infty}_0(\mathbb R^n)$ one has
$$
\aligned
&\lambda \langle \theta^2u,u\rangle=-\langle \partial_{x_k}(\theta^2u), a_{jk}\partial_{x_j}u\rangle -\langle \theta^2u,cu\rangle\\
&=-2\langle  \theta \partial_{x_k}\theta u, a_{jk}\partial_{x_j}u\rangle -\langle \theta^2\partial_{x_k} u, a_{jk}\partial_{x_j}u\rangle-\langle \theta^2u,cu\rangle \geq 0.
\endaligned
$$
But then,
$$
\aligned
&\langle \theta^2\partial_{x_k} u, a_{jk}\partial_{x_j}u\rangle \leq 2 |\langle  \theta \partial_{x_k}\theta u, a_{jk}\partial_{x_j}u\rangle|\\
&\lesssim  \|\nabla \theta\|_{\infty}\|u\|_2 \| \theta a_{jk}(x)\partial_{x_j}u\|_2\\
&\lesssim  \|\nabla \theta\|_{\infty}\|u\|_2 (\langle\theta^2 \partial_{x_k}u,a_{jk}(x)\partial_{x_j}u\rangle)^{1/2},
\endaligned
$$
i.e.
\begin{equation}
\label{z2}
(\langle\theta^2 \partial_{x_k}u,a_{jk}(x)\partial_{x_j}u\rangle)^{1/2}\lesssim  \|\nabla \theta\|_{\infty}\|u\|_2.
\end{equation}

Now, we pick $\theta=\theta_R$,  with $\theta_R\equiv 1$ on $B_R(\vec 0)$, $\theta_R\equiv 0$ on $(B_{2R}(\vec 0))^c$ and $\|\nabla\theta_R\|_{\infty}\leq c/R$.  Hence, from \eqref{z2} and the ellipticity hypothesis one has 
$$
\int_{B_{R}(\vec 0)}|\nabla u(x)|^2dx\leq \frac{c}{R^2}\,\|u\|^2_2.
$$
Letting $R\uparrow \infty$ it follows that $\nabla u\equiv 0$, so $u$ is constant, but $u\in L^2(\mathbb R^n)$. Hence $u\equiv 0$.

\end{proof}
Thus, from Lemma \ref{reg}, $L=L(x)$ is self-adjoint, $D_{min}(L)=D_{max}(L)$, since $u\in L^2(\mathbb R^n)$ satisfies that $u\in D(L^*)$ if there exists $\varphi\in L^2(\mathbb R^n)$ such that
$$
\langle L \psi,u\rangle =\langle \psi,\varphi\rangle\;\;\;\;\;\forall \,\psi\in \mathcal D,\;\;\text{and}\;\;\;L^*u=\varphi,
$$
i.e. $D(L^*)=D_{max}(L)$.
\vskip.1in

\begin{corollary}\label{AA3}
Assume that $u,\,f\in L^2(\mathbb R^n)$ and 
$$
Lu= - \partial_{x_j}(a_{jk}(x)\partial_{x_k}u)+c(x)u=f,\;\;\;\;u\in D(L).
$$
Then $\nabla u\in L^2(\mathbb R^n)$ and there exists $\{u_l:l\in\mathbb N\}\subset C^{\infty}_0(\mathbb R^n)$ such that
$$
u_l\to u,\;\;Lu_l\to Lu,\;\;\nabla u_l\to \nabla u \;\;\text{in}\;L^2(\mathbb R^n)\;\;\text{as}\;\;l\uparrow \infty.
$$
\end{corollary}

\begin{proof}[Proof of Corollary \ref{AA3}]
 If $u\in C^{\infty}_0(\mathbb R^n)$
 $$
 \int a_{jk}(x) \partial_{x_j}u\partial_{x_k}u=\int u Lu-\int c(x)u u.
 $$ 
 Therefore,
 \begin{equation}
 \label{10L}
 \| \nabla u\|^2_2\leq c(\|u\|_2\|Lu\|_2+\|u\|_2^2).
 \end{equation}

By continuity, the same estimate holds for $u\in D(L)=D_{max}(L)=D_{min}(L)$. We can pick up $\{u_l:l\in\mathbb N\}\subset C^{\infty}_0(\mathbb R^n)$
such that 
$$
u_l\to u,\;\;Lu_l\to Lu,\;\;\text{in}\;L^2(\mathbb R^n)\;\;\text{as}\;\;l\uparrow \infty.
$$
Then
$$
 \| \nabla (u_l-u_{l'})\|^2_2\leq c(\|u_l-u_{l'}\|_2\|Lu_l-Lu_{l'}\|_2+\|u_l-u_{l'}\|_2^2).
 $$
and the corollary follows.
\end{proof}

We have a unique self-adjoint realization of $L=L(x),$ so using the spectral theory  (see \cite{re-si-I}, Theorem VIII.5) one can define functions of $L$, like the heat kernel $e^{-tL}$, the unitary group $\{e^{itL}\,:\,t\in\mathbb R\}$, its fractional powers $L^{\alpha},\,\alpha>0$, etc.

Next, we shall prove (1)-(3) in Theorem \ref{2.1} for $\alpha=1$.

\begin{proof}[Proof of Theorem \ref{2.1}  : \underline{Case $\alpha=1$}]

(1) was already proved. So we consider (2) and (3).
First, assuming that  $f\in \mathcal D=C^{\infty}_0(\mathbb R^n)$ we shall prove that
 \begin{equation}
 \label{key}
 \| f\|_{2,2}\equiv \|(1-\Delta)f\|_2\sim \|Lf\|_2+\|f\|_2\sim \|\mathcal Lf\|_2 +\|f\|_2,
\end{equation}
 with $\mathcal {L}$ as above, i.e. 
 $\mathcal {L}=-\partial_{x_j}(a_{jk}(x)\partial_{x_k}\cdot).$
 
 For $f\in \mathcal D$, one has, from our hypothesis, that  
 $$
 \|L f\|_2\leq c\|f\|_2+ \|\partial_{x_j}(a_{jk}(x)\partial_{x_k}f)\|_2\leq c\|f\|_{2,2}.
 $$
 Next, we shall see that
 \begin{equation}
 \label{step1}
 \|f\|_2+\|\nabla f\|_2+\|\Delta f\|_2\leq c  \|\mathcal {L}f\|_2+c\|f\|_2.
 \end{equation}
 By Corollary  \ref{AA3}   it suffices to see that
 $$
 \|\Delta f\|_2\leq c  \|Lf\|_2+c\|f\|_2\sim \|\mathcal Lf\|_2+\|f\|_2.
 $$
 Pick $\theta\in C^{\infty}_0(\mathbb R^n)$ such that $\theta\equiv 1,\|x\|\leq 1/2$ , $\theta\equiv 0,\,\|x\|>1$ and for $R>0$ define $\theta_R(x)=\theta(x/R)$. Let $A(x)=(a_{jk}(x))$. We write
$$
A(x)=I+A_1(x)+A_2(x),
$$ where
$$
A_1(x)=(1-\theta_R(x))(A(x)-I),\;\;\;\;A_2(x)=\theta_R(x)(A(x)-I).
$$ 
  Note that
  $$
  \| A_1(x)\|_{\infty}\leq \sup_{\|x\|\geq R/2}\;\|A(x)-I\|_{\infty},
  $$
  and $\text{support} \,A_2(x)\subset \{x\,:\,\|x\|\leq R\}$. Thus,
  $$
  -\Delta f=\mathcal Lf+{\rm div}(A_1\nabla f)+{\rm div}(A_2\nabla f), 
  $$
  and 
  $$
 \| \Delta f\|_2\leq c(\|\mathcal L f\|_2+\|div(A_1\nabla f)\|_2+\|{\rm div}(A_2\nabla f)\|_2).
 $$
 Since
 $$
 {\rm div}(A_j\nabla f)=\nabla A_j\cdot \nabla f+A_j \nabla^2f,\;\;\;\;\;j=1,2.
 $$
 and
 $$
 \|\nabla A_1\|_{\infty}\leq c\, (\|\nabla a_{jk}\|_{\infty}+\|a_{jk}\|_{\infty}+1),
 $$
so by using hypothesis \eqref{asymp}
$$
\begin{aligned} 
&\| \rm{div}(A_1\nabla f)\|_2\leq  c\, (\|\nabla a_{jk}\|_{\infty}+\|a_{jk}\|_{\infty}+1)\|\nabla f\|_2\\
&\;\;\;\;\;\;\;\;\;\;\;\;\;\;\;\;\;\;\;\;\;\;+\sup_{\|x\|\geq R/2}\|A(x)-I\|_{\infty}
\|\nabla^2f\|_2\\
&=O(\|\nabla a_{jk}\|_{\infty}+\|a_{jk}\|_{\infty}+1) \|\nabla f\|_2+c \,\epsilon_{R}\|\nabla^2f\|_2,
\end{aligned}
$$
with $\epsilon_R\to 0$ as $R\uparrow \infty$.  
Since $\text{support}\,A_2\subset \{\|x\|\leq R\}$, one has
 $$
 \begin{aligned}
 &\| {\rm div}(A_2\nabla f)\|_2=\big(\int_{\|x\|<R} |{\rm div}(A_2(x)\nabla f(x))|^2dx\big)^{1/2}\\
 &\leq c  (\sup_{\|x\|<R}\|\nabla A_2\|+1)\|\,(\int_{\|x\|<R}|\nabla f(x)|^2dx)^{1/2})\\
 &+c(\sup_{\|x\|<R}\| A_2\|+1)\|\,(\int_{\|x\|<R}|\partial^2 f(x)|^2dx)^{1/2}.
 \end{aligned}
 $$
 
By Schauder estimates for $\mathcal L$, (see \cite{evans} section 6.3.1 Theorem 1) since the coefficients $a_{jk}(\cdot)$ are $C^1$ 
 $$
 \sum_{|\alpha|\leq 2}\,\int_{\|x\|<R}|\partial_x^{\alpha}f|^2dx\leq c_R\,\int_{\|x\|\leq 2R}(|\mathcal Lf|^2+|f|^2)dx.
 $$
 Thus,
 $$
 \sum_{|\alpha|\leq 2}\,\int_{\|x\|<R}|\partial_x^{\alpha}f|^2dx\leq c_R (\|\mathcal Lf\|^2 +\|f\|_2).
 $$
 
 Gathering the above information we conclude that
 $$
\| \Delta f\|_2\leq c \,\epsilon_R \|\Delta f\|_2 +c_R\|\mathcal Lf\|_2+c_R\|\nabla f\|_2+c_R\|f\|_2.
$$
 Choosing $R$ so large such that $c\,\epsilon_R<1/2$ we get that
 \begin{equation}
 \label{res1}
 \| \Delta f\|_2\leq c \|\mathcal Lf\|_2+c_R\|\nabla f\|_2+c_R\|f\|_2\;\;\;\;\text{for}\;\;\;f\in\mathcal D.
\end{equation}
 which combined with \eqref{10L} gives
 \begin{equation}
 \label{final11}
 \|\Delta f\|_2\leq c \|\mathcal L f\|_2+c\|f\|_2\;\;\;\;\;\;\text{for}\;\;\;\;f\in\mathcal D.
 \end{equation}
 Hence, \eqref{key} holds for $f\in C^{\infty}_0(\mathbb R^n)$.

 Next, let $f\in D_{min}(L)=D_{max}(L)$. Choose $\{f_j\in \mathcal D\,:\,j\in\mathbb N\}$ such that 
 $$
 f_j\to f,\;\;\;\;\mathcal {L}f_j\to \mathcal {L}f\;\;\;\;\:\text{in}\;\;\;\;L^2\;\;\;\;\text{as}\;\;\;j\to \infty.
 $$
 Then, by \eqref{final11}
$$
\|f_j-f_k\|_{2,2} \leq c \|\mathcal {L}(f_j-f_k)\|_2+ \|f_j-f_k\|_2\to 0\;\;\;\text{as}\;\;\;j,k \to \infty.
 $$
Hence, the $f_j$'s converges in $H^2(\mathbb R^n)$ and $f\in H^2(\mathbb R^n)$ and \eqref{final11} holds for $f\in H^2(\mathbb R^n)$.

 This finishes the proof of Theorem \ref{2.1} for $\alpha=1$. We now turn to the proof of the case $\alpha\in(0,1)$
 \end{proof}
 
 \subsection{Proof of Theorem \ref{2.1} in the case $0<\alpha<1$}
 
 \begin{proof} 
 Since 
 $$
\langle Lf,f\rangle\geq 0, \;\;\;\; \forall f \in \mathcal D,
$$ and $L$ is self-adjoint with $D(L)=H^2(\mathbb R^n)$, the spectrum of $L$,  $\,\sigma(L)\subset [0,\infty)$, (see \cite{davis}, Theorem 4.3.1).

We shall use the following version of the spectral theorem given in \cite{davis} (Theorem 2.5.1):
\begin{theorem}
\label{spec-repr}
(Spectral Representation  {\cite{davis})} There   exist a finite measure $\mu$ in $S\times \mathbb N$  and a unitary operator 
$$
U:L^2(\mathbb R^n) \to L^2(S\times \mathbb N,d\mu),\;\;\;\;\;S\equiv \sigma(L)\subseteq[0,\infty),
$$ 
such that if $h:S\times \mathbb N\to \mathbb R$ is $h(s,n)=s$, then for $\xi\in L^2(\mathbb R^n)$
$$
\xi\in D(L) \Leftrightarrow  h\cdot U(\xi)\in L^2(S\times \mathbb N,d\mu).
$$
Moreover, $$ULU^{-1}\varphi=h\varphi,\;\;\;\;\;\forall\; \varphi \in U(\{D(L)\})$$
and 
$$g(h)\psi=Ug(L)U^{-1}\psi\,\;\;\;\; \forall g\in C_0(\mathbb R),\;\;\;\forall \psi \in L^2(S\times \mathbb N,d\mu).$$ 
\end{theorem}

 
 Thus, to define $L^{\alpha},\,\alpha\in(0,1)$, we proceed as in Theorem 4.3.3. of \cite{davis}, and the paragraph that precedes it
 $$
 D(L^{\alpha})=\{\xi\in L^2(\mathbb R^n)\,:\,h^{\alpha}U(\xi)\in  L^2(S\times \mathbb N,d\mu)\},
 $$ 
 with 
 $$
 h^{\alpha}\psi=UL^{\alpha}U^{-1}\psi,\;\;\;\forall\:\psi\in U(D(L^{\alpha})).
 $$
 We claim that for $\alpha\in(0,1)$ one has that $D(L)\subseteq D(L^{\alpha})$. This is because, $s^{\alpha}\leq s$ for $s\geq 1$ and both $s^{\alpha}, \,s$ are bounded for $s\in [0,1)$ 
 (see also \cite{re-si-I} last chapter, and \cite{fitz}).
 
 We now turn first to the special case where $c(x)=0$, i.e.
 $$
 L=\mathcal L=-\partial_{x_j}((a_{jk}(x)\partial_{x_k} \cdot).
 $$
 
 Define
 $$
 \mathcal L_{\alpha}=\mathcal L^{\alpha}+I,\;\;\;\alpha\in[0,1].
 $$
 Thus, if $m_{\alpha}(s)=s^{\alpha}+1,\;s\geq 0$, we have $\mathcal L_{\alpha}=m_{\alpha}(\mathcal L)$, so $\mathcal L_{\alpha}$ is in the spectral calculus of $\mathcal L$.
 \vskip.05in
 
 \underline{Claim} : $D(\mathcal L_{\alpha})=D(\mathcal L^{\alpha})$.

From  the above it follows : given  $f\in L^2(\mathbb R^n)$ one has that $$f \in D(\mathcal L^{\alpha}) \iff s^{\alpha}U(f)\in L^2(S\times \mathbb N,d\mu),$$
and
$$
f\in D(\mathcal L_{\alpha}) \iff m_{\alpha}(s)U(f)\in L^2(S\times \mathbb N,d\mu).
$$

We recall that $1\in L^2(S\times \mathbb N,d\mu)$, since it gives the identity. Thus, it is easy to see that
$$
s^{\alpha}U(f)\in L^2(S\times \mathbb N,d\mu) \iff (1+s^{\alpha})U(f)\in L^2(S\times \mathbb N,d\mu).
$$

We turn to our attention to $D(\mathcal L_{\alpha})$. For $\alpha=0$, $D(\mathcal L_{0})=L^2(\mathbb R^n)$ and for $\alpha=1$, $D(\mathcal L_{1})=D(L)=H^2(\mathbb R^n)$.

Let
$$
L_{2\alpha}^{\mathcal L}=\{f\in L^2(\mathbb R^n) : f\in D(\mathcal L_{\alpha})\}=\{f\in L^2(\mathbb R^n) :\|\mathcal L_{\alpha}f\|_2<\infty\}.
$$

Since
$$
\aligned
&\|\mathcal L_{\alpha}f\|_2^2=\langle (\mathcal L^{\alpha}+I)f,(\mathcal L^{\alpha}+I)f\rangle\\
&=\langle \mathcal L^{\alpha}f,\mathcal L^{\alpha} f\rangle+\langle f,f\rangle+2\langle \mathcal L^{\alpha}f,f\rangle
\endaligned
$$
and $\mathcal L^{\alpha}$ is a non-negative operator, since $\mathcal L$ is one to one, it follows that 
$$
\|f\|_2\leq \|\mathcal L_{\alpha}f\|_2.
$$

Consider next
$$
L_{2\alpha}^{\mathcal L,\sim}=\{f\in L^2(\mathbb R^n) :f=(\mathcal L_{\alpha})^{-1}v,\;v\in L^2(\mathbb R^n)\}.
$$
 \underline{Claim} : $L_{2\alpha}^{\mathcal L}=L_{2\alpha}^{\mathcal L,\sim}$.
 
 Assume that $g=\mathcal L_{\alpha}f\in L^2(\mathbb R^n)$, with $f\in L^2(\mathbb R^n)$. Then $(\mathcal L_{\alpha})^{-1}g=f$, since
 $$
 \aligned
 & U((\mathcal L_{\alpha})^{-1}g)=U(\mathcal L_{\alpha})^{-1}U^{-1}Ug=\frac{1}{h^{\alpha}+1}Ug=\frac{1}{h^{\alpha}+1}U\mathcal L_{\alpha}f\\
 &=\frac{1}{h^{\alpha}+1}U\mathcal L_{\alpha}U^{-1}Uf=\frac{(h^{\alpha}+1)}{(h^{\alpha}+1)}Uf=Uf.
 \endaligned
 $$
 Thus, $(\mathcal L_{\alpha})^{-1}g=f$. Hence, $L_{2\alpha}^{\mathcal L}\subseteq L_{2\alpha}^{\mathcal L,\sim}$.
 
 Assume now that $f\in L^2(\mathbb R^n)$ $f=(\mathcal L_{\alpha})^{-1}v,\,v\in L^2(\mathbb R^n)$. Then
  $$
 \aligned
 &(h^{\alpha}+1)Uf=(h^{\alpha}+1)U(\mathcal L_{\alpha})^{-1}v\\
 &=(h^{\alpha}+1)U(\mathcal L_{\alpha})^{-1}U^{-1}Uv=\frac{(h^{\alpha}+1)}{(h^{\alpha}+1)}Uv=Uv.
 \endaligned
 $$
 Thus, $\mathcal L_{\alpha}f\in L^2(\mathbb R^n)$. Hence, $L_{2\alpha}^{\mathcal L,\sim}\subseteq L_{2\alpha}^{\mathcal L}$, which completes the proof of the claim.
 \vskip.05in
 
 \underline{Remark}: Let $P_j(\mathcal L)f=f_j$, where $f\in L^2(\mathbb R^n)$ and
 $$
 UP_j(\mathcal L)U^{-1} \Psi=\chi_{[0,j]}(h)\Psi,\;\;\;\;\Psi \in L^2(S\times \mathbb N,d\mu).
 $$
 Since $\sigma(\mathcal L)\subseteq [0,\infty)$, one has that $f_j\to f$ in $L^2(\mathbb R^n)$ and
 $$
 \mathcal L^{\alpha}f_j \to \mathcal L^{\alpha}f\;\;\text{and} \;\;\mathcal L_{\alpha}f_j \to \mathcal L_{\alpha}f_j \;\;\text{in}\;\;L^2(\mathbb R^n).
 $$
 Since $D(\mathcal L)=H^2(\mathbb R^n)$,
 $$
\mathcal L P_j(\mathcal L)f\to \mathcal L f \;\;\text{in}\;\;L^2(\mathbb R^n), \;\;\text{for}\;\;\mathcal Lf\in L^2(\mathbb R^n).
$$
Moreover, the $H^2$-norm of $P_j(\mathcal L)f \big(\to f\big)$ is uniformly bounded by  the $H^2$-norm of $f$.

Define
$$
\mathcal D_{\mathcal L}=\{P_j(\mathcal L)f\,:\,f\in L^2(\mathbb R^n)\}.
$$
This is dense in $L^{\mathcal L}_{2\alpha}$ for $\alpha\in [0,1]$. Also, if $u\in \mathcal D_{\mathcal L}$, then $\mathcal Lu,\,u\in L^2(\mathbb R^n)$ and so $u\in H^2(\mathbb R^n)$.

We recall that $(\mathcal L)^{i\gamma},\,\gamma\in\mathbb R$  are uniformly bounded in $\gamma$ as operators in $L^2(\mathbb R^n)$, since $|(h)^{i\gamma}|\leq 1,\;h> 0$. Using this and complex interpolation of the weighted $L^2$ space $\int_{S\times \mathbb N}|\Psi|^2(h^{2\lambda}+1)d\mu$,  we have
$$
\Big[L^{\mathcal L}_{\alpha_1};L^{\mathcal L}_{\alpha_2}\Big]_{\theta}=L^{\mathcal L}_{\theta\alpha_1+(1-\theta)\alpha_2},
$$
where $[\,;\,]_{\theta}$ is the complex interpolation space.

Let us prove : 

If $\alpha\in [0,1]$, then $\,H^{2\alpha}(\mathbb R^n)\subseteq L^{\mathcal L}_{\alpha}$ with $\|f\|_{L^{\mathcal L}_{\alpha}}\leq c\|f\|_{2\alpha,2}$.

\vskip.05in

We know that $f\in H^{2\alpha}(\mathbb R^n)\iff f=(I-\Delta)^{-\alpha}g,\;g\in L^2(\mathbb R^n)$. Assume first that $g\in H^2(\mathbb R^n)$. Note that for $\alpha \in [0,1]$, $D(\mathcal L+I)\subseteq D(\mathcal  L_{\alpha})$, since for $h\geq 0$, $h^{\alpha}+1\leq c(h+1)$, and $D(\mathcal L+I)=H^{2}(\mathbb R^n)$.

Let  $\varphi \in \mathcal D=C^{\infty}_0(\mathbb R^n)$, $z\in \mathbb C$ with $0\leq Re(z)\leq 1$, and define
$$
F(z)=\int((\mathcal L)^z+I)(-\Delta+I)^{-z}(g)\,\overline{\varphi}\,dx.
$$

Note that $(-\Delta+I)^{-z}(g)\in H^2(\mathbb R^n)$, $0\leq Re(z)\leq 1$, since $g\in H^2(\mathbb R^n)$ so $
((\mathcal L)^z+I)(-\Delta+I)^{-z}(g)\in L^2(\mathbb R^n)$, because $D(\mathcal L+I)\subseteq D(\mathcal  L_{\alpha})$. Hence, $F(z)$ is analytic for $z\in \mathbb C$ with $0\leq Re(z)\leq 1$. When $z=0+i\gamma$
$$
\aligned
&|F(i\gamma)|\leq \|(\mathcal L)^{i\gamma}+I)(-\Delta+I)^{-i\gamma}(g)\|_2\|\varphi\|_2\\
&\leq c\|g\|_2\|\varphi\|_2=c\|f\|_{2\alpha,2}\|\varphi\|_2,
\endaligned
$$
and when $z=1+i\gamma$
$$
\aligned
&|F(1+i\gamma)|\leq \|(\mathcal L)^{1+i\gamma}+I)(-\Delta+I)^{-(1+i\gamma)}(g)\|_2\|\varphi\|_2\\
&\leq c\|(-\Delta+I)^{-(1+i\gamma)}(g)\|_{2,2}\|\varphi\|_2\leq c\|g\|_2\|\varphi\|_2=c\|f\|_{2\alpha,2}\|\varphi\|_2,
\endaligned
$$

Thus, by the three lines theorem 
$$
|F(\alpha)|\leq c\|f\|_{2\alpha,2}\|\varphi||_2.
$$
Taking the supremum over all $\varphi \in \mathcal D$ we get
$$
\|\mathcal L_{\alpha}f\|_2\leq c \|f\|_{2\alpha,2},\;\;\;\forall f=(-\Delta+I)^{-\alpha}(g),\;\;g\in H^2(\mathbb R^n),
$$
where $c$ above is a universal constant. For general $f\in H^{2\alpha}(\mathbb R^n)$, $f=(-\Delta+I)^{-\alpha}g,\,g\in L^2(\mathbb R^n)$, we choose $g_j\to g $ in $L^2(\mathbb R^n)$, $g_j\in H^2(\mathbb R^n)$, let $f_j=-\Delta+I)^{-\alpha}g_j$. Thus, for $j,j'$ we have
$$
\|\mathcal L_{\alpha}(f_j-f_{j'})\|_2\leq c\|f_j-f_{j'}\|_{2\alpha,2},
$$
and 
$$
\|\mathcal L_{\alpha}(f_j)\|_2\leq c\|f_j\|_{2\alpha,2}.
$$
Hence, $\mathcal L_{\alpha}(f_j)\to \tilde f$ in $L^{\mathcal L}_{2\alpha}$, with $\|\mathcal L_{\alpha}(\tilde f)\|_2\leq c\|f\|_{2\alpha,2}$. But, since $f_j\to f$ in $H^{2\alpha}(\mathbb R^n)$ and $\|h\|_2\leq c\|\mathcal L_{\alpha}h\|_2$, we have $\tilde f=f$, and we get the desired result $\,H^{2\alpha}(\mathbb R^n)\subseteq L^{\mathcal L}_{2 \alpha}$.

We now turn to the proof of  $L^{\mathcal L}_{2\alpha,2}\subseteq H^{2\alpha}(\mathbb R^n)$ for $\alpha\in [0,1]$, and $\|f\|_{2\alpha,2}\leq c\|f\|_{L^{\mathcal L}_{2\alpha}}$. Here, if we want to "repeat" the argument, we run into a difficulty. We would like to embed $(\mathcal L_{\alpha})^{-1}$ into an analytic family, the natural choice is $((\mathcal L)^z+I)^{-1}$. But if $z=a+ib$ with $a\in [0,1]$ and $b\in \mathbb R$, we could have $h^{a+ib}+1$ nearly $0$, since
$$
h^{a+ib}=h^a(\cos(b\log(h))+i\sin(b\log(h)),
$$ 
so if $\log(h)=\epsilon, \;b\epsilon=\pi,\,\cos((\pi)=-1,\,\sin(\pi)=0$, so we run into a problem. To solve it we recall that 
$$
L^{\mathcal L}_{2\alpha}=\{f\in L^2(\mathbb R^n)\,:\,\mathcal L^{\alpha}f\in L^2(\mathbb R^n)\}.
$$

\underline{Claim} : $L^{\mathcal L}_{2\alpha}=\{f\in L^2(\mathbb R^n)\,:\;(\mathcal L +1)^{\alpha}f\in L^2(\mathbb R^n)\}$, where $(\mathcal L +1)^{\alpha}=m_{\alpha}(\mathcal L)$, 
with $m_{\alpha}(h)=(h+1)^{\alpha}$, i.e. we define $(\mathcal L +1)^{\alpha}$ through the functional calculus of $\mathcal L$, \underline {not} through the functional calculus of 
$\mathcal L+I$. We have not shown that both coincide, so we proceed as described above.

To prove the claim, since $\mathcal L_{\alpha}$ and $m_{\alpha}(\mathcal L)$ are both defined through the functional calculus of $\mathcal L$, it suffices to show that, for $h>0$, one has $(1+h)^{\alpha}/(1+h^{\alpha})\sim 1$ which is obvious. Similarly, we define $(\mathcal L+1)^z$ by the functional calculus of $\mathcal L$, using $(h+1)^z, \,h>0$, which never vanishes.

In the rest of the proof, we are always using the functional calculus of $\mathcal L$.

Let $f\in L_{2\alpha}^{\mathcal L}$, so 
$$
\;\exists \,f\in L^2(\mathbb R^n) \;\;\text{s.t.}\;\;f=(\mathcal L+I)^{-\alpha}(g).
$$
Let  $g_j=\chi_{[0,j]}(\mathcal L)(g)$, so that $g_j\to g$ in $L^2(\mathbb R^n)$. Let $f_j=(\mathcal L+I)^{-\alpha}(g_j)$, so $f_j\to f$ in $L_{2\alpha}^{\mathcal L}$.  Fix $j$,  pick $\varphi \in 
 \mathcal D$, and consider
 $$
 F_j(z)=\int(-\Delta+I)^z((\mathcal L+I)^{-z}g_j)\overline{\varphi},\;\;\;\Re(z)\in [0,1].
 $$
Since $(\mathcal L+I)(\mathcal L+I)^{-z}g_j=(\mathcal L+I)^{1-z}g_j\in L^2(\mathbb R^n)$, one has that $(\mathcal L+I)^{-z}g_j\in H^2(\mathbb R^n)$. Hence, $F_j(z)$ is well defined.

If $z=i\gamma$, then
$$
|F_j(i\gamma)|\leq \|g_j\|_2 \|\varphi\|_2.
$$
 
 If $z=1+i\gamma$, then
$$
 \aligned
& |F_j(1+i\gamma)|\leq c\|(-\Delta+I)^{1+i\gamma}(\mathcal L+I)^{-(1+i\gamma)}(g_j)\|_2\|\varphi\|_2\\
&\leq c \|(-\Delta+I)^{1}(\mathcal L+I)^{-(1+i\gamma)}(g_j)\|_2\|\varphi\|_2\\
&=c \|(\mathcal L+I)^{-(1+i\gamma)}(g_j)\|_{2,2}\|\varphi\|_2\\
&\leq c \|(\mathcal L+I)(\mathcal L+I)^{-(1+i\gamma)}(g_j)\|_{2}\|\varphi\|_2\leq c\|g_j\|_2\|\varphi\|_2.
\endaligned
$$
By the three lines theorem, taking supremum over all $\varphi\in \mathcal D$ with $\|\varphi\|_2\leq 1$, we get
$$
\|(-\Delta+I)^{\alpha}(\mathcal L+I)^{-\alpha}(g_j)\|_2\leq c\|g_j\|_2.
$$
We recall that $f_j=(\mathcal L+I)^{-\alpha}(g_j)$, which belongs to $H^2(\mathbb R^n)$ since $(\mathcal L+I)f_j\in L^2(\mathbb R^n)$. Hence,
$$
\|f_j\|_{2\alpha,2}\leq c \|g_j\|_2.
$$

Since $g_j \to g $ in $L^2(\mathbb R^n)$, using the above inequality for $g_j-g_{j'}$, we see that $f_j$ is uniformly bounded in $H^{2\alpha}(\mathbb R^n)$ and converges to $\tilde f$ in $H^{2\alpha}(\mathbb R^n)$. But $f_j=(\mathcal L+I)^{-\alpha}(g_j)$ so $f_j\to f$ in $L^2(\mathbb R^n)$. Then $f=\tilde f$, and the proof of Theorem \ref{2.1} completed, when $c\equiv 0$. The case $c=c(x)\geq 0$ follows in the same manner.

 \end{proof}

\subsection{Extension of   Theorem \ref{2.1} to the values $\alpha >1$}

 
 \begin{corollary}
\label{abcd}
If the coefficients $a_{jk}$'s and $c$ satisfy the hypothesis \eqref{hyp1}-\eqref{hyp3}, \eqref{asymp} and \eqref{J}  with $\text{J}\geq 2\kappa-1,\,\kappa\in \mathbb N$, then for any $\alpha\in [0,\kappa]$ it follows that
\begin{enumerate}
\item 
$ L^{\alpha}$ is self-adjoint and unbounded on $L^2(\mathbb R^n)$,
\vskip.05in
\item
the domain of $ L^{\alpha}$ is $H^{2\alpha}(\mathbb R^n)=(1-\Delta)^{-\alpha}L^2(\mathbb R^n)$,
\vskip.05in
\item
\begin{equation}
\label{r3}
\|f\|_{2\alpha,2}=\|(1-\Delta)^{\alpha}f\|_2\sim \|f\|_2+\|L^{\alpha}f\|_2,\;\;\;\forall f\in H^{2\alpha}(\mathbb R^n).
\end{equation}

\end{enumerate}

\end{corollary}
First, we shall prove Corollary \ref{abcd} in the case $\alpha=\kappa, \;\kappa\in\mathbb N$. Thus, we need to show that $L^{\kappa},$ is self-adjoint, i.e. $D_{min}(L^{\kappa})= D_{max}(L^{\kappa})$. From the spectral calculus of $L$ one has that $L^{\kappa}$ is self-adjoint. However, to prove \eqref{r3} we need to show that $(L)^{\kappa}=L.....L$ $\;\kappa$-times.

 Using this result,  we shall extend it to all values $\alpha\geq 0$. Latter, we shall also  give a different proof of the estimate \eqref{r3}.


\begin{definition}
\label{2.8}
For $\alpha>1$, let $\kappa$ be such that $\alpha\leq \kappa,\,\kappa\in \mathbb N$. Let $J\geq 2{\kappa}-1$ in \eqref{J}. For $m\in \mathbb N$ with $m\leq \kappa$,
$$
(L)^m=L...L\;\;\;\;m-\text{times},
$$ while for any $\alpha>1$, $L^{\alpha}$ will be given by the spectral calculus of $L$.
\end{definition}

\begin{lemma}
\label{2.9}
Let $\kappa\in \mathbb N$. Let $J>2\kappa-1$ in \eqref{J}. For $m\in \mathbb N$ with $m\leq \kappa$, $(L)^{m}$ is self-adjoint, with 
$$
L^m=(L)^m=L...L,
$$
and
$$
D_{min}((L)^{m})=D(L^m)=H^{2m}(\mathbb R^n).
$$
\end{lemma}

In order to prove Lemma \ref{2.9}, we need a preliminary result.

\begin{proposition}
\label{AB1}
For $m=0,1,2,....$ assume that the coefficients of $L$ satisfy the hypothesis \eqref{J} with $J\geq m+1$. If $u\in L^2(\mathbb R^n)$, with $Lu$ as a distribution in $\mathcal D'(\mathbb R^n)$ belongs to $H^m(\mathbb R^n)$, then $u\in H^{m+2}(\mathbb R^n)$ and
$$
\|u\|_{m+2,2}\leq c_m(\|Lu\|_{m,2}+\|u\|_2).
$$

\end{proposition}
\begin{proof} [Proof of Proposition \ref{AB1}] By induction in $m$. Case $m=0$. This follows by Theorem \ref{2.1} in the case $\alpha=1$, which we have already proved.

Assume now that if for some $m\geq 0$, $u\in L^2(\mathbb R^n),\,Lu\in H^m(\mathbb R^n)$, then $u\in H^{2+m}(\mathbb R^n)$ and
$$
\|u\|_{m+2,2}\leq c_m(\|Lu\|_{m,2}+\|u\|_2).
$$

We will show that this holds also for $m+1$, i.e.  if $u\in L^2(\mathbb R^n),\,Lu\in H^{m+1}(\mathbb R^n)$, then $u\in H^{3+m}(\mathbb R^n)$ and
\begin{equation}
\label{91}
\|u\|_{m+3,2}\leq c_m(\|Lu\|_{m+1,2}+\|u\|_2).
\end{equation}

Let $\theta_N$ be a regularizing operator of order zero, i.e. a smooth frequency truncation for $|\xi|\leq N$. Let $\partial_{x}$ be a first order differential operator. Then by induction hypothesis
$$
\aligned
&\|\partial_x\theta_Nu\|_{m+2,2}\leq c(\|L(\partial_x\theta_Nu)\|_{m,2}+\|\partial_x\theta_N u\|_2)\\
&\leq c(\| [L\partial_x\theta_N-\partial_x\theta_NL]u\|_{m,2}+\|\partial_x\theta_NLu\|_{m,2}+\|u\|_2),
\endaligned
$$
 where $\partial_x\theta_N\cdot=\partial_x(\theta_N\cdot)$. 
 But $L\partial_x\theta_N-\partial_x\theta_NL$ is an operator of order two, uniformly in $N$. Hence, uniformly in $N$, one has (since $J\geq m+2$)
 $$
 \|\partial_x\theta_Nu\|_{m+2,2}\leq c(\|u\|_{2+m,2}+\|Lu\|_{m+1,2}+\|u\|_2),
 $$
 and by the induction hypothesis 
 $$
 \|u\|_{m+2,2}\leq c(\|Lu\|_{m,2}+\|u\|_2)\leq c(\|Lu\|_{m+1,2}+\|u\|_2).
 $$
Therefore, we can conclude that $u\in H^{3+m}(\mathbb R^n)$ and \eqref{91} holds.

\end{proof}

\begin{proof} [Proof of Lemma \ref{2.9}]

We consider $(L)^m=L...L$, $\;m$-times, with domain
$$
D_{min}((L)^m)=\{f\in L^2(\mathbb R^n) : \exists\, f_j\in \mathcal D,\;f_j\to f \,\,\text{and}\,\,(L)^mf_j\to g \;\text{in}\;L^2\}.
$$

\underline{Claim 1}: $(L)^m$ with domain $D_{min}((L)^m)$ is a closed operator, i.e. if $\{f_k\}\subseteq D_{min}((L)^m)$ with
$$
f_k\to f\;\;\;\text{and}\;\;\;(L)^mf_k\to g\;\;\;\text{in}\;\;L^2(\mathbb R^n),
$$
then $f\in D_{min}((L)^m)$ and $(L)^m f=g$. Moreover, this $g$ is unique.

This proof is similar to that given for  claim 1 in the proof of Theorem \ref{2.1} for the case $\alpha=1$. Hence, it will be omitted.

\underline{Claim 2}: For $f\in \mathcal D$, we have
$$
\|f\|_{2m,2}\sim \|(L)^mf\|_2+\|f\|_2.
$$
 
 Since $J>2\kappa-1$ and $m\leq k$, one has that
 $$
\|(L)^mf\|_2+\|f\|_2\lesssim \|f\|_{2m,2}.
$$ 
The opposite inequality is already known for  the case $m=1$ (Theorem \ref{2.1}). Now using Proposition \ref{AB1} for $f\in \mathcal D$ one has
$$
\|f\|_{m+2,2}\lesssim \| Lf\|_{m,2}+\|f\|_2.
$$
Since  $f\in \mathcal D$ iterating the argument one has that 
$$
\aligned
\|f\|_{2m,2} &\lesssim \| Lf\|_{2m-2,2}+\|f\|_2\lesssim \|(L)^2f\|_{2m-4,2}+\|f\|_2\\
&\lesssim \| (L)^3f\|_{2m-6,2}+\|f\|_2\lesssim ...\lesssim \|(L)^mf\|_{2}+\|f\|_2,
\endaligned
$$
which yields Claim 2. 

We remark that a different proof of Claim 2 follows by the proof of \eqref{r3} given below.

\underline{Claim 3}: For $f\in D_{min}((L)^m)$, we have
$$
\|f\|_{2m,2}\sim \|(L)^mf\|_2+\|f\|_2.
$$

This clearly will imply that $D_{min}((L)^m)=H^{2m}(\mathbb R^n)$, see the comments at the end of the proof of Lemma \ref{2.9}.

Let $(f_j)\subseteq \mathcal D$ be such that 
$$
f_j\to f\;\;\;\text{and}\;\;\; (L)^mf_j\to (L)^mf\;\;\;\;\text{in}\;\;\;L^2(\mathbb R^n).
$$

We first show that $\{f_j\,:\,j\in\mathbb N\}$ is a Cauchy sequence in $H^{2m}(\mathbb R^n)$, since 
$$
\|f_j-f_k\|_{2m,2}\lesssim \|(L)^m(f_j-f_k)\|_2+\|(f_j-f_k)\|_2.
$$
Hence, $f_j\to \tilde{f}$ in $H^{2m}(\mathbb R^n)$. But  since $f_j\to f$ in $L^2(\mathbb R^n)$ it follows that $\tilde f=f$. Thus
$$
\aligned
\|f\|_{2m,2}&=\lim_{j\to \infty}\|f_j\|_{2m,2}\\
&\lesssim \lim_{j\to \infty}\big( \|(L)^mf_j\|_{2}+\|f_j\|_2\big)= \|(L)^mf\|_2+\|f\|_2.
\endaligned
$$
Conversely,
$$
\aligned
 \|(L)^mf\|_2+\|f\|_2&=\lim_{j\to \infty}\big( \|(L)^mf_j\|_2+\|f_j\|_2\big)\\
 &\lesssim \lim_{j\to \infty} \|f_j\|_{2m,2}= \|f\|_{2m,2},
 \endaligned
 $$
and the Claim 3 follows.

\underline{Claim 4}: Defining the composition domain of $(L)^m$, $D_{comp}((L)^m)$, as
$$
D_{comp}((L)^m)=\{x\in D_{comp}((L)^{m-1})\,:\,(L)^{m-1}x\in D(L)=H^2(\mathbb R^n)\}.
$$
we claim that
$$
D_{comp}((L)^m)=D_{min}((L)^m)=H^{2m}(\mathbb R^n).
$$
\underline{Proof} By induction. The case $m=2$ follows by Proposition \ref{AB1} 
$$
\{x\in D(L)=H^2(\mathbb R^n)\,:\,Lx\in D(L)\} =H^4(\mathbb R^n).
$$

Assume the result for $m-1$. Then
$$
\aligned
&\{x\in D_{comp}((L)^{m-1})\,:\,(L)^{m-1}x\in D(L)=H^2(\mathbb R^n)\}\\
&=\{x\in H^{2(m-1)}(\mathbb R^n)\,:\,(L)^{m-1}x\in H^2(\mathbb R^n)\}.
\endaligned
$$
Let $x\in D_{comp}(L)$. Since $(L)^{m-1}=L(L)^{m-2}$, by Proposition \ref{AB1}, $(L)^{m-2}x\in H^4(\mathbb R^n)$. Iterating, we obtain that $x\in H^{2m}(\mathbb R^n)$, and so $D_{comp}((L)^m)\subseteq H^{2m}(\mathbb R^n)$. The opposite inclusion is clear, since $J\geq m+1$.

To complete the proof of Lemma \ref{2.9} we need the following result in \cite{Yo} (Theorem 3 (iv), page 343) concerning the operational calculus:

\begin{theorem}
\label{oper-calc}
Let $H$ be a self-adjoint operator densely defined in a separable Hilbert space $X$. If $x\in D(f(H))$, then $f(H)x\in D(g(H))$ is equivalent to
 $x\in D((gf)(H))$ and
 $$
 g(H)f(H)x=((gf)(H))x.
 $$
 \end{theorem}

First, we consider the case $m=2$. To apply Theorem \ref{oper-calc} we  choose $H=L$, $X=L^2(\mathbb R^n)$, and $f(\lambda)=g(\lambda)=\lambda$ so $f(L)=L$. 

If $x\in D(L)=H^2(\mathbb R^n)$, and $Lx\in D(L)=H^2(\mathbb R^n)$, then by Proposition \ref{AB1} it follows that  $x\in H^4(\mathbb R^n)=D_{min}(LL)$, so that
$x\in D((gf)(L))=D(L^2)$ and $LLx=L^2x$. 

Thus, $H^4(\mathbb R^n)=D_{min}(LL)=D(L^2)$, and for $x\in H^4(\mathbb R^n)$ one has $L^2x=LLx$, so the case $LL=L^2$ is done.

Now we assume that 
$$
L...L=(L)^{m-1}=L^{m-1}\;\;\;\;\text{and}\;\;\;\;\;D_{min}((L)^{m-1})=D(L^{m-1})=H^{2(m-1)}(\mathbb R^n),
$$
and we want to prove the result for $(L)^m$. Let  $X=L^2(\mathbb R^n)$, $H=L$, $g(\lambda)=\lambda$, and $f(\lambda)=\lambda^{m-1}$. Assume $x\in D((L)^{m-1})=H^{2(m-1)}(\mathbb R^n)$, and $(L)^{m-1}x\in D(L)=H^2(\mathbb R^n)$. Then, by Proposition \ref{AB1}, $x\in H^{2m}(\mathbb R^n)$, and we know that 
$D_{min}((L)^m)=H^{2m}(\mathbb R^n)$, see Claim 3. Hence, $x\in D((gf)(L))=D(L^m)$, and 
$$
L^m(x)=L(L)^{m-1}x=(L)^m(x)\;\;\;\text{and}\;\;\;D(L^m)=H^{2m}(\mathbb R^n),
$$
as we want.

To complete the proof we shall show that $D_{min}((L)^m)=H^{2m}(\mathbb R^n)$ as it was stated in Claim 3. If $f \in D_{min}((L)^m)$, then there exists $\{f_j\,:\,j\in \mathbb N\}\subseteq \mathcal D$ such that 
$$
f_j\to f\;\;\;\text{and}\;\;\; (L)^mf_j\to (L)^mf\;\;\;\;\text{in}\;\;\;L^2(\mathbb R^n).
$$
But then by Claim 3, $f_j\to \tilde f$ in $H^{2m}(\mathbb R^n)$, with $\tilde f =f$.  Next, assume $f\in H^{2m}(\mathbb R^n)$, pick $\{f_j\,:\,j\in \mathbb N\}\subseteq \mathcal D$ such that $f_j\to f$ in $H^{2m}(\mathbb R^n)$. Then by Claim 3 $(L)^mf_j$ converges in $L^2(\mathbb R^n)$, so $f\in D_{min}((L)^m)$.
\end{proof}

Once we have Lemma \ref{2.9} using the same argument given in the proof of Theorem \ref{2.1}, Lemma \ref{2.9}, and spectral properties of the operators one has:

\begin{proposition}\label{2.14.b} 
Let $ \kappa\in \mathbb N$ and $J\geq 2\kappa-1$ in \eqref{J}. Then \begin{enumerate}
\item
for $\alpha\in [0, \kappa]$, $L^{\alpha}$  is a self-adjoint operator, with domain $H^{2\alpha}(\mathbb R^n)$ and for $f\in H^{2\alpha}(\mathbb R^n)$
$$
\|f\|_{2\alpha,2}\sim \| L^{\alpha}f\|_2+\|f\|_2.
$$
 with the comparability constants depending only on $L$, $J$ and $n$.
\item
 If $B$ is a bounded Borel function on $[0,\infty)$, $B(L^{\alpha})$ commutes with any operator of the form $f(L)$, where $f$ is a locally bounded Borel function, and $f(L)$ is the operator defined through the spectral calculus of $L$ as a self-adjoint operator.
\end{enumerate}
\end{proposition}

\subsection{Sketch of different  proof of the  estimate \eqref{r3}.}

One can give an alternative proof of the main estimate \eqref{r3} in Corollary \ref{abcd}. This is based on a commutator estimate, Proposition \ref{comm-BMO} below, and the result for $\alpha\in[0,1]$. The commutator estimate needed was established in \cite{LD} (Theorem 1.2, estimate 1.5):

 \begin{lemma} [\cite{LD}]
\label{abcde}
Let $\theta\in (0,1)$ and $1<p<\infty$. Then for any $f, g\in \mathcal S(\mathbb R^n)$
\begin{equation}
\label{r7}
\| [D^{\theta}; g]f\|_p\leq c\|D^{\theta} g\|_{BMO} \|f \|_p\leq c \| D^{\theta}g\|_{\infty}\|f\|_p.
\end{equation}
\end{lemma}

We recall for $f\in L^1_{loc}(\mathbb R^n)$ one says that $f\in BMO$ if
$$
\| f\|_{BMO}=\sup_{Q: cube}\frac{1}{|Q|}\int_Q|f(x)-f_Q|dx<\infty,
$$
where $\;f_Q=\int_Qf(x)dx$ and $\;|Q|=\text{measure of }\;Q$.

We shall also use the following result : 

\begin{proposition}
\label{comm-BMO}
If $g\in  W^{1,\infty}(\mathbb R^n)$, then $D^{\theta}g\in BMO(\mathbb R^n)$ for any $\theta\in(0,1)$, and
$$
\| D^{\theta}g\|_{BMO}\leq c \|g\|_{\infty}^{1-\theta}\|\nabla g\|_{\infty}^{\theta}\leq c (\|g\|_{\infty}+\| \nabla g\|_{\infty})=c\|g\|_{1\infty}.
$$
\end{proposition}

\begin{proof} [Proof of Proposition \ref{comm-BMO}]

\underline{Claim} : for  $\gamma\in \mathbb R$, $D^{1+i \gamma}g\in BMO$.

First, we shall see that $Dg\in BMO$. We write
$$
\widehat {Dg}(\xi)=|\xi|\widehat{g}(\xi)=\sum_{j=1}^n\frac{\xi^2_j}{|\xi|}\widehat{g}(\xi)=\sum_{j=1}^n\frac{\xi_j}{|\xi|}\xi_j \widehat{g}(\xi)
=\sum_{j=1}^n\widehat{R_j(\partial_{x_j}g)}(\xi).
$$
Since $\nabla g\in L^{\infty}(\mathbb R^n)$, and   $R_j $ the $j$-th Riesz transform maps $L^{\infty}$ to BMO, it follows that $Dg\in BMO$.

Next, we recall that convolution singular integral operators maps $BMO$ to $BMO$. Since $D^{i\gamma}$ is a convolution singular integral, it does that (for a different proof  see \cite{St} page 136, 5.21 (ii)). Hence
$$
D^{1+i\gamma}g=D^{i\gamma} Dg,
$$ 
and the claim follows.

Now, for $z=a+i\gamma$, $\gamma\in \mathbb R$, $0\leq a\leq 1$ we define $F(z)=D^zg$. Thus, $F$ is continuous on
$\{z=a+i\gamma\in \mathbb C\;:\,a\in[0,1],\;\;\gamma\in \mathbb R\}$ and analytic in its interior. 
Since $L^{\infty}(\mathbb R^n)\subset BMO(\mathbb R^n)$, one has that $F(0+i\gamma)\in BMO$. Also, by the above claim $F(1+i\gamma)\in BMO$.

Finally, using the three lines theorem one gets that $D^sg\in BMO$ for $s\in(0,1)$.

\end{proof}
 \begin{proof} [Sketch of a different proof of estimate \eqref{r3} in Corollary \ref{abcd}]
 

  First, we consider the case $\alpha\in (1, 3/2)$. Thus, $\alpha=1+\theta$ with $\theta \in(0,1/2)$ and we need to show :
\begin{equation}
\label{r3a}
\|f\|_{2\alpha,2}=\|(1-\Delta)^{\alpha}f\|_2\sim \|f\|_2+\|L^{1+\theta}f\|_2,\;\;\;\forall f\in H^{2+2\theta}(\mathbb R^n).
\end{equation}

A combination of  the results for $\alpha\in (0,1)$, the fact that $\theta\in(0,1/2)$, the inequality \eqref{r7} and\ Proposition \ref{comm-BMO} leads to
$$
\begin{aligned}
&\| L^{1+\theta}f\|_2=\|L^{\theta}Lf\|_2\lesssim \|J^{2\theta}Lf\|_2+\|Lf\|_2\\\
&\lesssim \|J^{2\theta}(a_{jk}\partial^2_{x_jx_k}f)\|_2 +\|J^{2\theta}(\partial_{x_j}a_{jk}\partial_{x_k}f+cf)\|_2+\|f\|_{2,2}\\
&\lesssim \|D^{2\theta}(a_{jk}\partial^2_{x_jx_k}f)\|_2 +\|a_{jk}\partial^2_{x_jx_k}f\|_2+\|f\|_{2,2}\\
&\lesssim \|D^{2\theta}(a_{jk}\partial^2_{x_jx_k}f)-a_{jk}D^{2\theta}\partial^2_{x_jx_k}f+a_{jk}D^{2\theta}\partial^2_{x_jx_k}f\|_2+\|f\|_{2,2}\\
&\lesssim \| [D^{2\theta};a_{jk}]\partial^2_{x_jx_k}f\|_2 +\|a_{jk}D^{2\theta}\partial^2_{x_jx_k}f\|_2+\|f\|_{2,2}\\
&\lesssim \|D^{2\theta}a_{jk}\|_{BMO}\|\partial^2_{x_jx_k}f\|_2+\|D^{2\theta}\partial^2_{x_jx_k}f\|_2+\|f\|_{2,2}\\
&\lesssim \|f\|_{2+2\theta,2}.
\end{aligned}
$$

To obtain the reverse estimate,  we write
$$
\begin{aligned}
&\| L^{1+\theta}f\|_2+\|f\|_{2,2}=\| L^{\theta}Lf\|_2+\|f\|_{2,2}\\
&\sim \|J^{2\theta}Lf\|_2+\|Lf\|_2+\|f\|_{2,2}\sim \|D^{2\theta}Lf\|_2+\|f\|_{2,2}\equiv V_2,
\end{aligned}
$$
and define
\begin{equation}
\label{gen22}
\begin{aligned}
Lf&=-a_{jk}(x)\partial_{x_jx_k}^2f-\partial_{x_j}a_{jk}(x)\partial_{x_k}f+c(x)f\\
&=L_{pr}f-\partial_{x_j}a_{jk}(x)\partial_{x_k}f+c(x)f.
\end{aligned}
\end{equation}

Now
$$
\begin{aligned}
&\|D^{2\theta}Lf\|_2=\|D^{2\theta}(L_{pr}f-\partial_{x_j}a_{jk}\partial_{x_k}f+c(x)f)\|_2\\
&\geq \|D^{2\theta}(L_{pr}f)\|_2-\|D^{2\theta}(\partial_{x_j}a_{jk}\partial_{x_k}f-c(x)f)\|_2\\
&\geq \|D^{2\theta}(a_{jk}\partial^2_{x_jx_k}f)\|_2-M\|f\|_{2,2}\\
&=\|a_{jk}D^{2\theta}\partial^2_{x_jx_k}f+[D^{2\theta};a_{jk}]\partial_{x_k}f\|_2-M\|f\|_{2,2}\\
&\geq \| L_{pr}D^{2\theta}f\|_2-\|a_{jk}\|_{BMO}\| \partial^2_{x_jx_k}f\|_2-M\|f\|_{2,2}\\
&\geq \| LD^{2\theta}f\|_2-\|(L- L_{pr})D^{2\theta}f\|_2-\|a_{jk}\|_{BMO}\| \partial^2_{x_jx_k}f\|_2-M\|f\|_{2,2}.
\end{aligned}
$$
Therefore,
$$
\begin{aligned}
&V_2=\|D^{2\theta}Lf\|_2+\|f\|_{2,2}\gtrsim \| LD^{2\theta}f\|_2+\|f\|_{2,2}\\
&\gtrsim \| J^2 D^{2\theta}f\|_2+\|f\|_{2,2}\sim \|f\|_{2+2\theta,2}
\end{aligned}
$$
Combining the above estimates, one gets the desired result for $\alpha\in (1, 3/2)$.

 For the case $\alpha\in [3/2, 2)$, i.e.  $2\alpha=3+\theta,\,\theta\in [0,1)$ one uses 
$$
\| J^{2\alpha}f\|_2\sim \| J^{2+\theta}\nabla f\|_2+ \| J^{2+\theta} f\|_2
$$
and reapplies the previous argument to the term $\| J^{2+\theta}\nabla f\|_2$. 

The case for larger values of $\alpha$ follows by an iterative argument. This combine the estimate \eqref{r7} and the following estimates:
if $2\alpha\in \mathbb N$, then
$$
\|L^{\alpha}f\|_2+\|f\|_2\sim \|f\|_{2\alpha,2} \iff \|L^{\alpha}f\|_2+\|f\|_{2\alpha-1,2}\sim \|f\|_{2\alpha,2}
$$ 
and if $2\alpha \in[1,\infty)-\mathbb N$, then
$$
\|L^{\alpha}f\|_2+\|f\|_2\sim \|f\|_{2\alpha,2} \iff \|L^{\alpha}f\|_2+\|f\|_{[2\alpha],2}\sim \|f\|_{2\alpha,2}
$$
where $[\cdot]$ denotes the integer part function.

\end{proof}

  \vskip.1in
  
  \subsection{Estimates for the non-linear results}
  
 To conclude this section, we deduce some results which will be needed for our non-linear results in Section 3,  
 
  First, we recall that 
 $$
L^{s/2} L^{\alpha} f =L^{\alpha} L^{s/2}f,\;\;\;\;\;\;f\in H^{s+2\alpha}(\mathbb R^n),
$$
where this follows from the spectral calculus for $L$ (\cite{fitz}, Theorem 4 (ii)).

\begin{corollary} \label{remark-future} For $\alpha\in(0,1)$ and $M>0$
 \begin{equation}
 \label{art-vis-est}
 \begin{aligned}
 &(i)\;\|e^{it L^{\alpha}}g\|_2\leq c\|g\|_2,\;\;\;\;\forall \,t\in\mathbb R,\\
 &(ii)\;\|e^{-\epsilon tL^{\alpha}}g\|_2\leq c\|g\|_2,\;\;\;\;\;\;\forall \,\epsilon, \, t>0,\\
 &(iii)\;\| Le^{-\epsilon tL^2+it L^{\alpha}}g\|_2\leq \frac{c}{\epsilon^{1/2}t^{1/2}}\|g\|_2,\;\;\forall \,\epsilon, t\in (0,M).
 \end{aligned}
 \end{equation}
 \end{corollary}
 
 The estimates $(i)-(ii)$  in \eqref{art-vis-est} follow from the spectral theorem for $L$, since $L$ is positive and self-adjoint, and 
the functional calculus associated to the bounded functions  $f_1(x)=e^{itx^{\alpha}}$, and $f_2(x)=e^{-\epsilon t x^{\alpha}}$.

 For $(iii)$,  $f_3(x)=\epsilon^{1/2}t^{1/2}x e^{-\epsilon t x^2}e^{i t x^{\alpha}}$ is a bounded function, so it follows from the functional calculus of $L$.

\vskip.1in
\section{Proof of Theorem \ref{A1} and Theorem \ref{A1a}}
\begin{proof} [Proof of Theorem \ref{A1}]
We shall consider the formally equivalent integral version of the IVP \eqref{fNLS}
\begin{equation}
\label{int-eq1}
u(t)=e^{itL^{\alpha}}u_0+i\int_0^{t} e^{i(t-t')L^{\alpha}}P(u,\overline{u})(t')dt'.
\end{equation}

We shall obtain the local well-possedness of the equation \eqref{int-eq1} via the contraction principle. We define the operator
\begin{equation}
\label{int-eq2}
\Psi_{u_0}(v)(t)=e^{itL^{\alpha}}u_0+i\int_0^{t} e^{i(t-t')L^{\alpha}}P(v,\overline{v})(t')dt',
\end{equation}
on the complete space
$$
\mathcal X_{s}^T=\{v:\mathbb R^n\times [-T,T]\to \mathbb C\:\; :\;v\in C([-T,T]:H^{s}(\mathbb R^n))\},
$$
 with norm
$$
|||v|||_{s,T}=\sup_{[-T,T]}\|v(t)\|_{s,2},
$$
$s$ as in the statement of Theorem \ref{A1} and $T>0$ to be fixed.

Using the results in Section  2, Corollary \ref{remark-future}, Corollary \ref{abcd}, and Proposition \ref{2.14.b} we have :
$$
\aligned
&\|e^{itL^{\alpha}} f\|_{s,2}=\|(1-\Delta)^{s/2} e^{itL^{\alpha}}f\|_2\sim \| L^{s/2}e^{itL^{\alpha}} f\|_2+\|e^{itL^{\alpha}}f\|_2\\
&=\|e^{itL^{\alpha}}L^{s/2}f\|_2+\|f\|_2\sim \|L^{s/2}f\|_2+\|f\|_2\\
&\leq c \|(1-\Delta)^{s/2}f\|_2=c\|f\|_{s,2}.
\endaligned
$$

Also, we recall an inequality found in \cite{Ka-Po} : for any $l>0$ 
\begin{equation}
\label{KP}
\| J^l(fg)||_2\leq c(\|f\|_{\infty}\|J^lg\|_2+\|g\|_{\infty}\|J^lf\|_2).
\end{equation}

Therefore, since $s>n/2$, for any $f\in H^{s}(\mathbb R^n)$, it follows from the form of $P(z,\overline{z})$ in \eqref{hyp-P} that
\begin{equation}
\label{ine-77a}
\begin{aligned}
\| P(f,\overline{f})\|_{s,2} &\leq c\,(\|f\|_{\infty}^{N_1-1}+\|f\|_{\infty}^{N_2-1})\|f\|_{s,2}\\
& \leq c\, (\|f\|_{s,2}^{N_1}+\|f\|_{s,2}^{N_2}),
\end{aligned}\end{equation}
and similarly
\begin{equation}
\label{ine-77b}
\begin{aligned}
&\| P(f,\overline{f})-P(g,\overline{g})\|_{s,2} \\
&\leq c\, (\|f\|_{s,2}^{N_1-1}+\|f\|_{s,2}^{N_2-1}+\|g\|_{s,2}^{N_1-1}+\|g\|_{s,2}^{N_2-1})\,\|f-g\|_{s,2}.
\end{aligned}
\end{equation}

Combining the above estimates, one has that  for any $v,w \in \mathcal X_{s}^T$ 
$$
||| \Psi_{u_0}(v)|||_{s,T}\leq c\|u_0\|_{s,2}+cT(|||v||_{s,T}^{N_1}+|||v|||_{s,T}^{N_2})
$$
and
$$
\aligned
&||| \Psi_{u_0}(v)-\Psi_{u_0}(w)|||_{s,T}\\
&\leq cT (|||v|||_{s,T}^{N_1-1}+|||v|||_{s,T}^{N_2-1}+|||w|||_{s,T}^{N_1-1}+|||w|||_{s,T}^{N_2-1})|||v-w|||_{s,T}
\endaligned
$$
\vskip.1in
Restricting the domain of $\Psi_{u_0}$ to the set 
$$
\overline{B_{s,T}(R)}=\{v\in\mathcal X_s^T\;:\;|||v|||_{s,T}\leq R\}\;\;\;\;\;R=8c\|u_0\|_{s,2},
$$
from the above estimates, one sees that if $ v, w\in \overline{B_{s,T}(R})$
$$
|||\Psi_{u_0}(v) |||_{s,T}\leq c\|u_0\|_{s,2}+cT(R^{N_1-1}+R^{N_2-1})R,
$$
and
$$
||| \Psi_{u_0}(v)-\Psi_{u_0}(w)|||_{s,T}\leq 2cT(R^{N_1-1}+R^{N_2-1})|||v-w|||_{s,T}.
$$

By fixing $T=T^*$ such that
\begin{equation}
\label{007}
2cT^*(R^{N_1-1}+R^{N_2-1})\leq 1/4,
\end{equation}
it follows that $\Psi_{u_0}$ defines a contraction map in $\overline{B_{s,T^*}(8c\|u_0\|_{s,2})}$, which yields the solution $u \in C([-T^*,T^*]:H^{s}(\mathbb R^n))\cap C([-T^*,T^*]:H^{{s}-2\alpha}(\mathbb R^n))$.

Next,  we consider the regularity of the map data $\to $ solution of the \eqref{int-eq1}, i.e. $u_0\to u(x,t)=u_{u_0}$ where  $u_{u_0}$ is the solution of \eqref{int-eq1} obtained by the fixed point argument above. Thus, we consider the operator
$$
\Lambda : H^{s}(\mathbb R^n)\times \mathcal X_{s}^T\to \mathcal X_{s}^T,
$$
defined as
$$
\Lambda(v_0,w)=w-\Psi_{v_0}(w),
$$
so that $\Lambda (u_0, u_{u_0})=0$. We observe that 
$$
D_{w}\Lambda(v_0,w)v=v-i\int_0^{t} e^{i(t-t')(L)^{\alpha}}(\partial_uP(w,\overline{w})v+
\partial_{\overline{u}}P(w,\overline{w})\overline{v})(t')dt'.
$$
From \eqref{007}, it follows that $D_{w}\Lambda(u_0,u_{u_0})$ is an invertible operator from $ \mathcal X_{s}^T$ to itself. Hence, 
 the implicit function theorem proves that the map data $\to$ solution, i.e.  $u_0\to u=u_{u_0}$ is locally defined and smooth (analytic) from $H^s(\mathbb R^n)$ into $C([-T^*,T^*]:H^{s}(\mathbb R^n))$.

Finally, we observe that  the  solution $$u \in C([-T^*,T^*]:H^{s}(\mathbb R^n))\cap C^1((-T^*,T^*):H^{s-2\alpha}(\mathbb R^n))$$ is regular enough (classical), hence it solves  the non-local differential equation in \eqref{fNLS}.

 \end{proof}

 \begin{proof} [Proof of Theorem \ref{A1a}]

The proof follows the artificial viscosity method. We shall sketch the main points, for details we refer to \cite{LiPo1}. 

\underline{Step 1 : The associated viscous problem :} For $\epsilon>0$ consider the IVP
\begin{equation}
 \label {fNLS-v}
\begin{cases}
 \begin{aligned}
 & \partial_t u =-\epsilon L^2 u+i L^{\alpha}u+Q(u,\bar{u}, \nabla u,\nabla \bar{u})=0,\;\;\;t\geq 0,\\
 &u(x,0)=u_0(x),
 \end{aligned}
 \end{cases}
 \end{equation}
and its equivalent integral form
\begin{equation}
\label{int-v}
u(t)=U(t)u_0+\int_0^t U(t-t')Q(u,\bar{u}, \nabla u,\nabla \bar{u})(t')dt',
\end{equation}
with the notation
\begin{equation}
\label{group-v}
U(t)f=U_{\epsilon,\alpha}(t)f=e^{t(-\epsilon L^2 +iL^{\alpha})}f.
\end{equation}

To obtain the solution $u^{\epsilon}(x,t)$ of the problem \eqref{int-v}, in a time interval depending on $\epsilon$, we define the operator
\begin{equation}
\label{ope-v}
\Phi (u)(t)=U(t)u_0+\int_0^t U(t-t')Q(u,\bar{u}, \nabla u,\nabla \bar{u})(t')dt',
\end{equation}
with domain 
$$
\mathcal X_{s}^T=\{v:\mathbb R^n\times [0,T]\to \mathbb C\:\; :\;v\in C([0,T]:H^{s}(\mathbb R^n))\},
$$
$s$ as in the statement of Theorem \ref{A1a}, and norm
$$
|||v|||_{s,T}=\sup_{[0,T]}\|v(t)\|_{s,2}.
$$

Applying the operator $L^{s/2}$ to  the equation in \eqref{ope-v} and recalling from Corollary \ref{abcd} that
$$
\|L^{s/2}f\|_{2}+\|f\|_2\sim \|f\|_{s,2},\;\;\;\;\,f\in H^{s}(\mathbb R^n),
$$
we get using Proposition \ref{2.14.b} that
$$
\aligned
&L^{s/2}\Phi (u)(t)\\
&=L^{s/2}U(t)u_0+\int_0^t L^{s/2}U(t-t')Q(u,...)(t')dt'\\
&=U(t)L^{s/2}u_0+\int_0^t LU(t-t')L^{s/2-1}Q(u,...)(t')dt'
\endaligned
$$

Hence, taking the $L^2$-norm above, a combination of  the estimates in Corollary \ref{remark-future}, and its comments, and an inequality consequence of \eqref{KP}
$$
\|Q(u,\bar{u},\nabla u,\nabla \bar{u})\|_{s-1,2}\leq c ( |||u|||_{s,2}^{N_1}+|||u|||_{s,2}^{N_2})
$$ 
which leads to the estimate
$$
|||\Phi (u)|||_{s,T}\leq c\|u_0\|_{s,2}+\frac{c\,T^{1/2}}{\epsilon^{1/2}}( |||u|||_{s,2}^{N_1}+|||u|||_{s,2}^{N_2}).
$$

A similar argument yields the inequality
$$
\aligned
&|||\Phi (u)-\Phi(v)|||_{s,T}\\
&\leq \frac{c\,T^{1/2}}{\epsilon^{1/2}}( |||u|||_{s,2}^{N_1-1}+|||u|||_{s,2}^{N_2-1}+|||v|||_{s,2}^{N_1-1}+|||v|||_{s,2}^{N_2-1})|||u-v|||_{s,T}.
\endaligned
$$

These estimates guarantee : $\,\exists \,T=T(\|u_0\|_{s,2};N_1;N_2;\epsilon)>0$ with $T(\|u_0\|_{s,2};N_1;N_2;\epsilon)\downarrow 0$ as $\epsilon\downarrow 0$, such that the operator $\Phi$ has a fixed point $u^{\epsilon}=u^{\epsilon}(x,t)$ in 
$$
\overline{B_{s,T}(R)}=\{v\in\mathcal X_{s}^T\;:\;|||v|||_{s,T}\leq 8c\|u_0\|_{s,2}\}
$$
which is the solution of the IVP \eqref{fNLS-v}.  Due to our hypothesis on $J$ and the fact that
$$
\int_0^T \|LL^{s/2}u^{\epsilon}(t)\|_2dt<\infty,
$$
one has that
\begin{equation}
\label{extra-reg}
u^{\epsilon}\in C([0,T]:H^{s}(\mathbb R^n)\cap C((0,T]:H^{s+2}(\mathbb R^n)).
\end{equation}

Moreover, we can reapply this argument to extend the solution from the time interval $[0,\tilde T]$ to $[0,\tilde T+T]$ as long as the solution 
$u^{\epsilon}=u^{\epsilon}(x,t)$ remains in the set $\overline{B_{s,\tilde T}(R)}$ satisfying \eqref{extra-reg}..
 
 \underline{Step 2 : A priori estimate independent of $\epsilon>0$.}  Now, we shall deduce an a priori estimate for the life span of solution $u=u^{\epsilon}=u^{\epsilon}(x,t)$ of the IVP \eqref{fNLS-v} independent of the parameter $\epsilon>0$. To obtain it, we shall use energy estimates on solutions of the  IVP \eqref{fNLS-v}.  Thus, we apply the operator $L^{s/2}$ to the equation in  \eqref{fNLS-v}, multiply the result  by $\overline{L^{s/2}u}$, integrate the result in $\mathbb R^n$ and take the real part.  Here, we use Lemma \ref{2.9}, i.e. $(L)^{s/2}=L^{s/2}$ with $s$ an even integrar. We notice that the arguments at the end of the previous section show that
 $$
 Re\Big (-\epsilon\int L^{s/2}L^{2}u\overline{L^{s/2}u} dx+i\int L^{s/2}L^{\alpha}u\overline{L^{s/2}u}dx\Big)\leq 0,
 $$
 with $u=u^{\epsilon}$. Here we shall use the regularity of $u=u^{\epsilon}$ deduced in \eqref{extra-reg}. 
 To estimate the non-linear terms,  we shall only be concerned with those terms involving only derivatives strictly larger than $s$. Thus, we write
 $$
 \aligned
 &LQ(u,\bar{u}, \nabla u,\nabla \bar{u})=\partial_{x_j}(a_{jk}(x)\partial_{x_k}Q(u,..))\\
 &=\partial_{x_j}(a_{jk}(x) (\partial_{\partial_{x_l}u}Q(u,..)\partial^2_{x_lx_k}u +\partial_{\partial_{x_l}\overline{u}}Q(u,..)\partial^2_{x_lx_k}\overline{u}))+E_2\\
& =\partial_{\partial_{x_l}u}Q(u,..)\partial_{x_j}(a_{jk}(x)\partial^2_{x_lx_k}u) +\partial_{\partial_{x_l}\overline{u}}Q(u,..)\partial_{x_j}(a_{jk}(x)\partial^2_{x_lx_k}\overline{u})+ E_2 \\
 &=\partial_{\partial_{x_l}u}Q(u,..)L \partial_{x_l}u +\partial_{\partial_{x_l}\overline{u}}Q(u,..)L \partial_{x_l}\overline{u}+E_2\\
 & = A_1+A_2+E_{2},\endaligned
 $$
with $u=u^{\epsilon}$ and  where $E_{2}$ represents  terms involving derivatives of at most order $2$ of $u$.
 
 Therefore, using again Lemma \ref{2.9},
 $$
 \aligned
 &L^{s/2}Q(u,\bar{u}, \nabla u,\nabla \bar{u})\\
 &=\partial_{\partial_{x_l}u}Q(u,..)L^{s/2} \partial_{x_l}u +\partial_{\partial_{x_l}\overline{u}}Q(u,..)L^{s/2} \partial_{x_l}\overline{u}+E_{s},
 \endaligned
 $$
 with $u=u^{\epsilon}$ and where $E_{s}$ represent  terms involving at most derivatives of order $s$ which satisfies that
 $$
 \| E_{s}\|_2\leq  c (\|u\|_{s,2}^{N_1}+\|u\|_{s,2}^{N_2}).
 $$
 By integration by parts and the hypothesis \eqref{key-hyp-ener1} one has that
 $$
\aligned
&|\text{Re} \int \partial_{\partial_{x_l}u}Q(u,..)L^{s/2} \partial_{x_l}u \,L^{s/2} \overline{u}\,dx|\\
&=|-\frac{1}{2} 
  \int \partial_{x_l}( \partial_{\partial_{x_l}u}Q(u,..))|L^{s/2} u |^2dx| +E_s\\
 & \leq c (\|u\|_{s,2}^{N_1}+\|u\|_{s,2}^{N_2})
\endaligned
  $$
with $u=u^{\epsilon}$ and  where $N_1, N_2$ were defined in  \eqref{key-hyp-ener2}. Also, one has that
 $$
\aligned
&| \int \partial_{\partial_{x_l}\overline u}Q(u,..)L^{s/2} \partial_{x_l}\overline{u} \,L^{s/2} \overline{u}\,dx|\\
&=|-\frac{1}{2} 
  \int \partial_{x_l}( \partial_{\partial_{x_l}\overline{u}}Q(u,..))(L^{s/2} \overline{u})^2dx|\\
 & \leq c (\|u\|_{s,2}^{N_1}+\|u\|_{s,2}^{N_2}).
\endaligned
  $$
  
  Collecting the above information we obtain the energy estimate
  $$
  \frac{d}{dt} \|L^{s/2}u^{\epsilon}(t)\|_2 \leq c(\|u^{\epsilon}(t)\|_{s,2}^2+\|u^{\epsilon}(t)\|_{s,2}^{N_2}),
  $$
 which combined with \eqref{r3} shows : $\,\exists \,T^*=T^*(\|u_0\|_{s,2};N_1;N_2;n)>0$ such that
 \begin{equation}
 \label{abc}
 \sup_{[0,T^*]}\|u^{\epsilon}(t)\|_{s,2} \leq 8c\|u_0\|_{s,2}.
 \end{equation}

 This allows us to reapply step 1 and thus extend the local solutions $u^{\epsilon}(x,t)$ to the whole time interval $[0,T^*]$ (independent of $\epsilon$) in the class 
 $C([0,T^*]:H^{s}(\mathbb R^n))$ satisfying the uniform estimate  in  \eqref{abc} for $\epsilon\in (0,1)$.
 
  \underline{Step 3 : Convergence of the $u^{\epsilon}$ as $\epsilon  \downarrow 0$.} 
  
  For $\epsilon>\epsilon'>0$ we define 
  $$
  w(x,t)=w^{\epsilon,\epsilon'}(x,t)=(u^{\epsilon}-u^{\epsilon'})(x,t),
  $$
  which satisfies 
  $$
  \partial_tw-\epsilon' L^2w -i L^{\alpha} w =Q(u^{\epsilon},..)-Q(u^{\epsilon'},..)+(\epsilon-\epsilon')L^2u^{\epsilon},
  $$
  with  $w(x,0)=0$. Using the energy estimate as above it follows that 
  $$
  \frac{d}{dt}\|w(t)\|_{2,2}\leq c(1+\|u^{\epsilon}\|_{s,2}+\|u^{\epsilon'}\|_{s,2})^{N_2} \|w(t)\|_{2,2}+(\epsilon-\epsilon')\|u^{\epsilon}(t)\|_{6,2}.
  $$
  Hence, using \eqref{abc} with $s=6$ one gets that
 $$
 \sup_{[0,T]}\|w(t)\|_{2,2}\leq c(\epsilon-\epsilon')\|u_0\|_{6,2} T^* e^{cT^*(1+\|u_0\|_{s,2})^{N_2}}\leq K(\epsilon-\epsilon').
 $$
 
 Hence, the $u^{\epsilon}$ converges to $u(x,t)$  in $C([0,T^*]:H^2(\mathbb R^n))$ as  $\epsilon  \downarrow 0$. This combined with the a priori estimate \eqref{abc} 
 shows that that for any $\delta>0$
 $$
 u\in C([0,T^*]:H^{s-\delta}(\mathbb R^n))\cap L^{\infty}([0,T^*]:H^{s}(\mathbb R^n)),
 $$ 
 and that $u(x,t)$ solves the IVP \eqref{fNLSa}.
 
 Finally, as we have previously  remarked,  the proof of the persistence property, i.e.  $u\in C([0,T^*]:H^{s}(\mathbb R^n))$, and the (strong) continuous dependence of the solution $u(x,t)$ upon the data $u_0$ can be established by following the Bona-Smith argument  in \cite{BoSm}, for details we refer to \cite{LiPo1}.
  \end{proof}
  
\section{Proof of Theorems \ref{A2}}
\subsection{Extension to a local problem in the upper-half space.} In this subsection, we recall the Stinga-Torrea extension to the upper-half space $\mathbb R^{n+1}_+=\mathbb R^n\times(0,\infty)$ in \cite{StiTorCPDE2010}. This extension generalizes the Caffarelli-Silvestre extension for the fractional Laplacian \cite{CaSi} to more general self-adjoint operators $L$ like the ones considered in Theorem \ref{A3}.

Let $\alpha\in (0,1)$. Let $L$ be a non-negative second order elliptic operator as in \eqref{L} satisfying the hypothesis of Theorem \ref{A3}. Then, as it was shown in section 2, $L: H^2(\mathbb R^n) \subseteq L^2(\mathbb R^n) \to L^2(\mathbb R^n)$ is an unbounded, non-negative, densely defined, self-adjoint operator. Moreover,  for $u \in Dom(L^{\alpha})=H^{2\alpha}(\mathbb R^n)$ (see Theorem \ref{2.1} and  \cite{StiTorCPDE2010})), we define for $y\in[0,\infty)$,
\begin{equation}
\label{def:U}
\begin{aligned}
U(x,y)&=\frac{1}{\Gamma(\alpha)}\,\int_0^{\infty} e^{-tL}\,({L}^{\alpha}u)(x)\,e^{-y^2/4t}\,\frac{dt}{t^{1-\alpha}}\\
&=\frac{y^{2\alpha}}{4^{\alpha} \Gamma(\alpha)}\,\int_0^{\infty} e^{-t{L}}u(x)\,e^{-y^2/4t}\,\frac{dt}{t^{1+\alpha}}\\
\end{aligned}
\end{equation}

Thus, by Theorem 1.1 in  \cite{StiTorCPDE2010} one has 
\begin{equation}
\label{prop:U}
U \in C^{\infty}((0,\infty): H^{2\alpha}(\mathbb R^n)) \cap C^{0}([0,\infty):L^2(\mathbb R^n)),
\end{equation}
with
 \begin{equation}
 \label{eq:U}
 \begin{aligned}
 \begin{cases}
 & -{L}_xU+\dfrac{1-2\alpha}{y}\,\partial_yU+\partial_y^2U=0,\;\;\;\;\;\;(x,y)\in \mathbb R^{n+1}_+,\\
 & U(x,0)=u(x), \quad x \in \mathbb R^n ,
 \end{cases}
 \end{aligned}\quad\quad
 \end{equation}
 and 
\begin{equation}
\label{prop:U.2}
{L}^{\alpha}u(x)=  - c^*_{\alpha}\,\lim_{y\downarrow 0} \,y^{1-2\alpha} \,\partial_yU(x,y), \;\;\;\;\;\;\quad \text{in} \ L^2(\mathbb R^n),
\end{equation}
where $\,c^*_{\alpha}=4^{\alpha}\Gamma(\alpha)\left(2\alpha\Gamma(-\alpha)\right)^{-1}$.

 As it was remarked in \cite{StiTorCPDE2010} the identities above should be understood in $L^2(\mathbb R^n)$. Moreover,  $U(x,y)$  solves the singular boundary value problem \eqref{eq:U}, the first equality in \eqref{def:U} involves the heat semi-group $e^{-t{L}}$ acting on ${L}^{\alpha}u$ while the second does not involve fractional powers of ${L}$.  Thus,
 \eqref{prop:U.2} describes the fractional powers of ${L}$ in terms of the solution of the boundary value problem \eqref{eq:U} which is the original idea in \cite{CaSi}.

Notice that by using the special form of ${L}$, the equation in \eqref{eq:U} can be written in the form
 \begin{equation}
 \label{eq:U.Dir}
 \begin{aligned}
 \begin{cases}
 & \partial_y(y^{1-2\alpha}\partial_yU)+\nabla_{x} \cdot (y^{1-2\alpha}a(x)\nabla_{x}U)\\
 &\quad\quad\quad\quad\quad\quad\quad-y^{1-2\alpha}c(x)U=0 , \;\;\;\ (x,y)\in \mathbb R^{n+1}_+,\\
 & U(x,0)=u(x), \quad\quad\quad\quad\quad\quad\quad\quad\quad\quad\quad x \in \mathbb R^n ,
 \end{cases}
 \end{aligned}
 \end{equation}
 
 Finally, we claim a regularity property for the solution $U$ of the Dirichlet problem \eqref{eq:U.Dir}. 
 \begin{proposition} \label{propo:reg:U}
 Let $u \in H^{2\alpha}(\mathbb R^n)$. Then, the solution $U$ of \eqref{eq:U.Dir} given by \eqref{def:U} belongs to  $\dot{H}^1(\mathbb R_+^{n+1},y^{1-2\alpha}dxdy) \cap L^2_{loc}(\mathbb R_+^{n+1},y^{1-2\alpha}dxdy)$. Moreover,  
 \begin{equation} \label{propo:reg:U.1} 
\int_{\mathbb R^{n+1}_+} y^{1-2\alpha} \left((\partial_yU)^2+\left| \nabla_xU\right|^2 \right)dxdy \lesssim \|u\|_{L^2}^2+\|{L}^{\alpha}u\|_{L^2}^2 \lesssim \|u\|_{H^{2\alpha}}^2 .
 \end{equation}
 \end{proposition}
 
 \begin{remark} \label{weak:sol:Dirichlet}
One could try to extend the definition of ${L}^{\alpha}u$ to functions $u$ in $H^{\alpha}(\mathbb R^n)$ as follows. For $u\in H^{\alpha}(\mathbb R^n)$, define a weak solution $U \in \dot{H}^1(\mathbb R_+^{n+1},y^{1-2\alpha}dxdy) \cap L^2_{loc}(\mathbb R_+^{n+1},y^{1-2\alpha}dxdy) $ of the Dirichlet problem \eqref{eq:U}. Then, prove that the limit $\lim_{y\downarrow 0} \,y^{1-2\alpha} \,\partial_yU(x,y)$ exists in $H^{-\alpha}(\mathbb R^n)$. Finally, define ${L}^{\alpha}u \in H^{-\alpha}(\mathbb R^n)$ by \eqref{prop:U.2}.
\end{remark}

Before proving Proposition \ref{propo:reg:U}, we recall the following regularity results on $U$ obtained in \cite{StiTorCPDE2010}. 
\begin{lemma} \label{lemm:reg:U}
Let $u \in H^{2\alpha}(\mathbb R^n)$. Then, 
\begin{align}
\sup_{y>0}\|U(\cdot,y)\|_{L^2} &\le \|u\|_{L^2} \label{lemm:reg:U.1}, \\ 
\sup_{y>0}\|y^{1-2\alpha}\partial_yU(\cdot,y)\|_{L^2} &\lesssim \|\mathcal{L}^{\alpha}u\|_{L^2} .\label{lemm:reg:U.2}
\end{align}
\end{lemma}
\begin{proof}
Estimate \eqref{lemm:reg:U.1} is proved in Step 1 of the proof of Theorem 1.1 in \cite{StiTorCPDE2010} on page 2097.

To prove \eqref{lemm:reg:U.2}, we start with formula (2.5) on page 2098 in \cite{StiTorCPDE2010} (which holds in $L^2(\mathbb R^n)$) and use the change of variable $r=y^2/(4t)$ to deduce that, for all $x \in \mathbb R^n$, $y>0$,
\begin{align*} 
y^{1-2\alpha}\partial_yU(x,y)&=\frac{-1}{\Gamma(\alpha)} \int_0^{+\infty} e^{-t{L}}(L^{\alpha}u)(x)y^{2-2\alpha} \frac{e^{-\frac{y^2}{4t}}}{2t^{2-\alpha}}dt \\ 
&=\frac{2}{4^{\alpha}\Gamma(\alpha)} \int_0^{+\infty} e^{-\frac{y^2}{4r}{L}}({L}^{\alpha}u)(x) \frac{e^{-r}}{r^{\alpha}}dr .
\end{align*}
Thus, it follows from the Minkowski inequality and the fact that $e^{-t{L}}$ is a semi-group of contractions ($\|e^{-t{L}}f\|_{L^2} \le \|f\|_{L^2}$, for any $f \in L^2(\mathbb R^n)$), that 
\begin{align*} 
\left\|y^{1-2\alpha}\partial_yU(\cdot,y)\right\|_{L^2}& \le \frac{2}{4^{\alpha}\Gamma(\alpha)} \int_0^{+\infty} \|e^{-\frac{y^2}{4r}L}({L}^{\alpha}u)\|_{L^2} \frac{e^{-r}}{r^{\alpha}}dr\\
& \le  \frac{2\Gamma(1-\alpha)}{4^{\alpha}\Gamma(\alpha)}\|{L}^{\alpha}u\|_{L^2} ,
\end{align*}
for all $y>0$, which concludes the proof of \eqref{lemm:reg:U.2}.
\end{proof}

We recall the Cacciopoli estimate which follows from (2.3.2) in \cite{FKS} (we also refer to the proof of Lemma \ref{lemm:Cacciopoli} below). 
\begin{lemma}[Cacciopoli estimate] \label{lemm:Cacciopoli_bis}
Let $\alpha \in (0,1)$, ${\bf x}_0=(x_0,y_0) \in \mathbb R^{n+1}_+$ and $\epsilon>0$ such that  
\begin{equation*} 
B_{\frac{\epsilon}2}({\bf x}_0) = \left\{ (x,y) \in \mathbb R^{n+1} : |(x-x_0,y-y_0)| < \frac{\epsilon}2 \right\} \subset \mathbb R^{n+1}_+. 
\end{equation*} 
Assume that $U \in H^1(B_{\frac{\epsilon}2}({\bf x}_0) ,y^{1-2\alpha}dxdy)$ is a solution of \eqref{eq:U.Dir} on $B_{\frac{\epsilon}2}({\bf x}_0) $. Then, there exists $C>0$ independent of ${\bf x}_0 \in \mathbb R^{n+1}_+$ and $\epsilon>0$ such that 
\begin{equation} \label{Caccipolli:est.1_bis}
\left\|  y^{\frac{1-2\alpha}2}\nabla_{x,y}U\right\|_{L^2(B_{\frac{\epsilon}4}({\bf x}_0) )} \le C \epsilon^{-1 }\left\|  y^{\frac{1-2\alpha}2} U\right\|_{L^2(B_{\frac{\epsilon}2}({\bf x}_0) )} .
\end{equation}
\end{lemma} 

 We are now in position to prove Proposition \ref{propo:reg:U}.
 \begin{proof}[Proof of Proposition \ref{propo:reg:U}]
 The fact that $U \in L^2_{loc}(\mathbb R_+^{n+1},y^{1-2\alpha}dxdy)$ follows directly from \eqref{lemm:reg:U.1}, recalling $-1<1-2\alpha<1$. 
 
For $0<R_1<R_2$, $0<\epsilon<1<\delta$, we define 
\begin{equation*}
A_{R_1,R_2,\epsilon,\delta}=\left\{ (x,y) \in \mathbb R^{n+1}_+ : R_1<|x|<R_2 \ \text{and} \ \epsilon<y<\delta \right\} .
\end{equation*}
We claim that 
\begin{equation} \label{propo:reg:U.2a}
\int_{A_{R,R+1,\epsilon,\delta}} y^{1-2\alpha} \left|\nabla_x U \right|^2 dxdy \underset{R \to +\infty}{\longrightarrow} 0 .
\end{equation}
This is a consequence of the Cacciopoli estimate in Lemma \ref{lemm:Cacciopoli_bis} and a geometric argument. We cover $A_{R,R+1,\epsilon,\delta}$ by a finite number of balls of radius $\epsilon/4$ which remains $3\epsilon/4$ away from $\{y=0\}$. Observe that the cardinality of these balls only depends on $\epsilon$ and $\delta$ and that the resulting balls obtained after applying Lemma \ref{lemm:Cacciopoli_bis} stay $\epsilon/2$ away from the bottom. Then, it follows by applying \eqref{Caccipolli:est.1_bis}  and summing over all these balls that 
\begin{equation} \label{propo:reg:U.2}
\begin{aligned}
\int_{A_{R,R+1,\epsilon,\delta}} y^{1-2\alpha} &\left|\nabla_x U \right|^2 dxdy\\
& \le C_{\epsilon,\delta}\int_{A_{R-1,R+2,\frac{\epsilon}2,\delta+\frac{\epsilon}2}} y^{1-2\alpha}U^2 dxdy ,
\end{aligned}
\end{equation}
for a constant $C_{\epsilon,\delta}$ depending only on $\epsilon$ and $\delta$. The limit in \eqref{propo:reg:U.2a} follows then from the fact that $U \in L^{\infty}_yL^2_x$ and the dominated convergence theorem.

 Next, observe that
\begin{equation*} 
\begin{split}
 \partial_y(y^{1-2\alpha}\partial_y(U^2))+\nabla_{x} \cdot &(y^{1-2\alpha}a(x)\nabla_{x}(U^2))\\ & =2y^{1-2\alpha} \left( (\partial_yU)^2+a(x) \nabla_xU\cdot \nabla_xU +c(x)U^2\right) .
 \end{split}
\end{equation*}
Fix $0<R<r<R+1$ and $0<\epsilon<1<\delta$. We deduce from the divergence theorem that
\begin{align*} 
 \int_{\{|x| <r\}\times[\epsilon,\delta]}& y^{1-2\alpha}\left( (\partial_yU)^2+a(x) \nabla_xU\cdot \nabla_xU +c(x)U^2\right) dxdy \\ & =
 \int_{\partial(\{|x| <r\}\times[\epsilon,\delta])} y^{1-2\alpha} \left(U \partial_yU+Ua(x)\nabla_{x}U \right) \cdot \vec{n} d\sigma(x,y) ,
\end{align*}
where $\vec{n}$ denotes the outward pointing unit normal to  $\partial \left(\{|x| <r\}\times[\epsilon,\delta]\right)$ and $\sigma(x,y)$ its surface measure. Hence, by using the hypotheses on $a=a(x)$ and $c=c(x)$  and $\{|x|<R\} \subset \{ |x|<r \}$ (we assumed $R<r<R+1$), it follows that
\begin{align} \label{propo:reg:U.3}
 \lambda \int_{\{|x| <R\}\times[\epsilon,\delta]} y^{1-2\alpha}\left((\partial_yU)^2+ \left|\nabla_xU\right|^2 \right) dxdy \le I+II+III, 
\end{align}
where 
\begin{align*} 
I(r) &= \delta^{1-2\alpha} \left| \int_{\{|x|<r\}}   \left( U \partial_yU\right)(x,\delta) \, dx \right| , \\ 
II(r) & =\epsilon^{1-2\alpha} \left| \int_{\{|x|<r\}}   \left( U \partial_yU\right)(x,\epsilon) \, dx \right| , \\ 
III(r) & = \left| \int_{\epsilon}^{\delta} \int_{\{|x|=r\}}  y^{1-2\alpha}U a(x)\nabla_xU \cdot \vec{n}_x \, d\sigma_r(x)dy \right| ,
\end{align*}
where $\vec{n}_x$ denotes the outward pointing unit normal to  the sphere $\{|x|=r\}$ and $\sigma_r(x)$ its surface measure.

On the one hand, we observe by using \eqref{lemm:reg:U.1}-\eqref{lemm:reg:U.2} and the Cauchy-Schwarz inequality
\begin{align*}
I(r) & \le \|U(\cdot,\delta)\|_{L^2} \| \delta^{1-2\alpha}\partial_yU(\cdot,\delta)\|_{L^2} \lesssim \|u\|_{L^2}\|{L}^{\alpha}u\|_{L^2} ,\\ 
II(r) & \le \|U(\cdot,\epsilon)\|_{L^2} \| \epsilon^{1-2\alpha}\partial_yU(\cdot,\epsilon)\|_{L^2} \lesssim \|u\|_{L^2}\|{L}^{\alpha}u\|_{L^2}, 
\end{align*}
where the implicit constants are independent of $\epsilon$, $\delta$ and $R$. On the other hand, we have from the coarea formula and the Cauchy-Schwarz inequality 
\begin{align*} 
 \int_R^{R+1} III(r)  dr &  \le \|a\|_{L^{\infty}} \int_{A_{R,R+1,\epsilon,\delta}} y^{1-2\alpha} |U| \left| \nabla_xU \right| dxdy \\ 
 & \le \|a\|_{L^{\infty}} \left\|  y^{\frac{1-2\alpha}2} U\right\|_{L^2(A_{R,R+1,\epsilon,\delta})}\left\|  y^{\frac{1-2\alpha}2} \nabla_xU\right\|_{L^2(A_{R,R+1,\epsilon,\delta})} ,
\end{align*}
so that, by using \eqref{propo:reg:U.2}-\eqref{propo:reg:U.2a} and $U \in L^{\infty}_yL^2_x$, and hence $U \in L^2_{\epsilon<y<\delta}L^2_x$, for $\epsilon$, $\delta$ fixed,
\begin{align*} 
 \int_R^{R+1}  \left| III(r) \right|  dr \underset{R \to +\infty}{\longrightarrow} 0.
\end{align*}
Thus, it follows integrating \eqref{propo:reg:U.3} between $R$ and $R+1$, and then letting $R$ to $+\infty$  that, for any $0<\epsilon<\delta$, 
\begin{equation*}
\int_{\mathbb R^n \times[\epsilon,\delta]} y^{1-2\alpha}\left((\partial_yU)^2+ \left|\nabla_xU\right|^2 \right) dxdy \lesssim \|u\|_{L^2}\|\mathcal{L}^{\alpha}u\|_{L^2} .
\end{equation*}
Therefore, we conclude the proof of \eqref{propo:reg:U.1} by letting $\epsilon \to 0$, $\delta \to +\infty$ and using the fact that $\|{L}^{\alpha}u\|_{L^2} \sim \|u\|_{H^{2\alpha}}<\infty$ (see Theorem \ref{2.1}). 
\end{proof}

\subsection{Unique continuation for the extension problem} 
For $R>0$, we denote 
\begin{align} 
B_R^+&=\{(x,y) \in R^{n+1}_+ : |(x,y)| < R\} , \label{euc:ball} \\ 
\Gamma_R^0 &= \{(x,0) \in R^{n+1}_+ : |x| < R\} = B_R^+ \cap \{y=0\} , \label{def:Gamma0R}\\
\Gamma^+_R &=  \{(x,y) \in R^{n+1}_+ : |(x,y)| = R\}  . \label{def:Gamma+R}
\end{align}

\begin{definition} \label{def:weak_sol}
Let $0<\alpha<1$. Given $R>0$ and two functions $h \in L^1(\Gamma_R^0)$ and $f \in L^1(B_R^+)$, we say that a function $U$ is a weak solution of 
\begin{equation*}
 \begin{aligned}
 \begin{cases}
 \partial_y(y^{1-2\alpha}\partial_yU)+\nabla_{x} \cdot (y^{1-2\alpha}a(x)\nabla_{x}U)-y^{1-2\alpha}c(x)U={f} & \text{in}  \quad B_R^+,\\
- c^*_{\alpha}\lim_{y\downarrow 0} \,y^{1-2\alpha} \,\partial_yU=h & \text{on} \quad \Gamma_R^0 ,
 \end{cases}
 \end{aligned}
 \end{equation*}
if $U \in H^1(B_R^+,y^{1-2\alpha}dxdy)$ and 
\begin{align*}
\int_{B_R^+}y^{1-2\alpha} &\left( \partial_yU \partial_y\xi+a(x)\nabla_xU \cdot \nabla_x \xi+c(x)U\xi\right)dxdy\\ & =
-\int_{B_R^+} f \xi dxdy+
\frac{1}{c^*_{\alpha}}\int_{\Gamma_R^0} h \xi dx, 
\end{align*}
for all $\xi \in C^1(\overline{B_R^+})$ such that $\xi \equiv 0$ on $\Gamma_R^+$.
\end{definition}

The key point in the proof of Theorem \ref{A3} is the weak unique continuation property (WUCP) for the Neumann problem associated to the extension problem  \eqref{eq:U}.

\begin{proposition} \label{WUCP:ext}
Let $0<\alpha<1$  and let $a$, $c$ satisfy the hypotheses of Theorem \ref{A3}. Assume that for $R>0$, $U \in H^1(B_R^+,y^{1-2\alpha}dxdy)$ solves 
\begin{equation}
 \label{eq:U.2}
 \begin{aligned}
 \begin{cases}
 \partial_y(y^{1-2\alpha}\partial_yU)+\nabla_{x} \cdot (y^{1-2\alpha}a(x)\nabla_{x}U)&\\
\;\;\;\;\;\;\;\;\; \;\;\;\;\;\;\;\;\;\;\;\;\;-y^{1-2\alpha}c(x)U=0, & \text{in}  \quad B_R^+,\\
 - c^*_{\alpha} \lim_{y\downarrow 0} \,y^{1-2\alpha} \,\partial_yU=0 & \text{on} \quad \Gamma_R^0 ,
 \end{cases}
 \end{aligned}
 \end{equation}
 in the weak sense of Definition \ref{def:weak_sol}.
If $U(x,0)=0$ on $\Gamma_R^0$, then $U \equiv 0$ in $B_R^+$. The definition of $U(x,0)$ on $\Gamma_R^0$ is taken as in the definition after Corollary 2.1 in \cite{FKS}
\end{proposition}

\begin{remark} 
\begin{enumerate}
\item
Note that the condition $c \ge 0$ is not needed to prove Proposition \ref{WUCP:ext}.
\item
In the case of constant coefficients $a(x) \equiv a_0$, $c(x) \equiv 0$, Proposition \ref{WUCP:ext} was proved by R\"uland (see Proposition 2.2 in \cite{Ru1}).
\item
The regularity assumption on the coefficients, $\,a_{jk}\in C^2(\mathbb R^n),\,j,k=1,..,n$ in Theorem \ref{A3}, can probably be reduced to being Lipschitz, at the expense of extra technical complications based on the techniques developed in \cite{AKS} and \cite{KoTa2001} (see Remarks 10 in \cite{Ru1} and in \cite{RulandTAMS2016}).
 \end{enumerate}
 \end{remark}

We postpone the proof of Proposition \ref{WUCP:ext} to the next sub-section. 
\medskip

\begin{proof}[Proof of Theorem \ref{A3}]
From Proposition \ref{WUCP:ext}, we deduce that $U(x,y)=0$ in $B_R^+$ for some small $R>0$. 

Now observe that $U$ is a solution of an elliptic equation with $C^2$ coefficients in $\mathbb R^n \times (\epsilon,+\infty)$, for any $\epsilon>0$. Then, we deduce from the Weak Unique Continuation Property for such operators \cite{AronJMPA1957,Cordes1956} that $U \equiv 0$ in $\mathbb R^n \times (0,+\infty)$. Therefore we conclude from \eqref{prop:U.2} that $u \equiv 0$ in $\mathbb R^n$, which finishes the proof of Theorem \ref{A3}.  

\end{proof}

The remaining part of this section is dedicated to the proof of Proposition \ref{WUCP:ext}. The proof of this result relies on the Carleman estimates for the extension problem with variable coefficients proved by R\"uland in \cite{Ru1}, \cite{RulandTAMS2016}.

 
\medskip
\subsection{Carleman estimate.}  We introduce some useful notation.

\begin{itemize}
\item 
Recall that $a:\mathbb R^n \to \mathbb R^{n \times n}$ satisfies the assumptions of Theorem \ref{A3}. Let $a^{-1}=(a^{jk})_{1 \le j, k \le n}:\mathbb R^n \to \mathbb R^{n \times n}$ denote the pointwise inverse of $a(x)$ satisfying 
\begin{equation*}
a^{-1}(x)a(x)=a(x)a^{-1}(x)=1, \quad \forall \, x \in \mathbb R^n . 
\end{equation*} 
Observe that thanks our hypotheses on $a$, $(\mathbb R^n,a^{-1})$ has a Riemannian manifold structure. Then, we look at the upper-half space extension $\mathbb R^{n+1}_+=\mathbb R^n \times (0,+\infty)$ with the metric $1 \times a^{-1}$. 

\item If ${\bf x}=(x,y) \in \mathbb R^{n+1}_+=\mathbb R^n \times (0,+\infty)$, $1_{a^{-1}}(\bf x)$ denotes the geodesic distance of ${\bf x}$ to the origin on the Riemannian manifold $(\mathbb R^{n+1}_+,1 \times a^{-1})$. 

\item For $R>0$, we denote by $\mathcal{B}_R^+$ the geodesic half-ball centered at $0$ and of radius $R$ and by $\mathcal{B}_R^0$ its intersection with $\mathbb R^n$, \textit{i.e.} 
\begin{equation} \label{geo:ball}
\mathcal{B}_R^+=\left\{ {\bf x} \in \mathbb R^{n+1}_+ : 1_{a^{-1}}({\bf x}) <R \right\} 
\end{equation}
and 
\begin{equation} \label{int:geo:ball}
\mathcal{B}_R^0=\left\{ {\bf x} \in \mathbb R^{n+1}_+ : 1_{a^{-1}}({\bf x}) <R \right\}\cap\{y=0\} . 
\end{equation}

\item For $0<\delta<R$, we denote by $\mathcal{A}^+_{\delta,R}$ the geodesic half-annulus centered at $0$ and of radii $\delta$ and $R$, \textit{i.e.}
\begin{equation}
\mathcal{A}^+_{\delta,R}=\left\{ {\bf x} \in \mathbb R^{n+1}_+ : \delta<1_{a^{-1}}({\bf x}) <R \right\} .
\end{equation}

\item For a positive constant $C$, we will denote $C=C(a)$ to emphasize that the constant may depend on the ellipticity of $a$, its $L^{\infty}$-bound, and its $C^2$-smoothness 
($\sup_{\mathcal{B}^+_R} \left\{|a_{ij}|, |\nabla a_{ij}|, |\nabla^2 a_{ij}| \right\}$, if we work in the geodesic half-ball $\mathcal{B}^+_R$).
\end{itemize}

In this setting, we state a version of the variable coefficients Carleman estimate of R\"uland (see Proposition 7.1 in \cite{Ru1}).

\begin{proposition} [Variable coefficients Carleman estimate \cite{Ru1}] \label{prop:Carleman}Let $\alpha \in (0,1)$. Assume that the coefficients $a=(a_{jk})$ satisfy hypotheses of Theorem \ref{A3}. For $r=1_{a^{-1}}(\bf x)$, set
\begin{equation} \label{def:phi}
\phi(r)=-\ln(r)+\frac1{10} \left( \ln(r) \arctan(\ln(r)) -\frac12 \ln \left(1+\ln^2(r) \right)\right) .
\end{equation}
There exists $0<R_0=R_0(a) < 1$ such that for any $0<R <R_0$, the following is true.  Let $1<\kappa<\kappa_0$ and $\delta>0$ be such that $\kappa_0\delta<R$. Assume that $U \in H^1(\mathbb R^{n+1}_+,y^{1-2\alpha}d\bf x)$ with $\text{supp} \, U \subset \mathcal{A}^+_{\delta,R}$ satisfies 
\begin{equation}
 \label{eq:U.3}
 \begin{aligned}
 \begin{cases}
 \begin{aligned}
 \partial_y(y^{1-2\alpha}\partial_yU)+\nabla_{x} \cdot (y^{1-2\alpha}a(x)\nabla_{x}U)\\
 -c(x)y^{1-2\alpha}U=f, & \;\;\;\text{in}  \quad \mathcal{B}_R^+,\\
 \lim_{y\downarrow 0} \,y^{1-2\alpha} \,\partial_yU=0 & \;\;\;\text{on} \quad  \mathcal{B}_R^0.
 \end{aligned}
 \end{cases}
 \end{aligned}
 \end{equation}
 Then, there exists $\tau_0=\tau_0(a;\|c\|_{\infty})\ge 1$ such that for $\tau \ge \tau_0$, 
 \begin{equation} \label{Carleman.1} 
 \begin{split}
& \left\| e^{\tau \phi} r y^{\frac{2\alpha-1}2}f\right\|_{L^2(\mathcal{B}_R^+)} \\
&\gtrsim
 \tau^{\frac32} \left\| e^{\tau \phi} (1+\ln^2(r))^{-\frac12}r^{-1}y^{\frac{1-2\alpha}2}U\right\|_{L^2(\mathcal{B}_R^+)} \\ & \quad +  \tau^{\frac12}\left\| e^{\tau \phi} (1+\ln^2(r))^{-\frac12}y^{\frac{1-2\alpha}2}\nabla _{\bf x}U\right\|_{L^2(\mathcal{B}_R^+)}  \\ & \quad + \tau \delta^{-1} \left\| e^{\tau \phi} y^{\frac{1-2\alpha}2}U\right\|_{L^2(\mathcal{A}_{\delta,\kappa\delta}^+)} .
 \end{split}
 \end{equation}
\end{proposition}

\begin{remark}  
In \cite{Ru1,RulandTAMS2016}, it is shown that estimate \eqref{Carleman.1} holds in the more general case where 
\begin{equation*}
\lim_{y\downarrow 0} \,y^{1-2\alpha} \,\partial_yU=VU, \quad \text{on} \  \mathcal{B}_R^0,
\end{equation*} and $V$ is a potential satisfying homogeneous critical and subcritical assumptions under the restriction $\alpha \in [\frac14,1)$ (see \cite{Ru1}), or $V \in C^1(M)$ and $M$ is a compact Riemmanian manifold with a smooth metric in the full range $0<\alpha<1$ (see \cite{RulandTAMS2016}).
\end{remark}

The estimates for the first two terms on the right-hand side of \eqref{Carleman.1} of Proposition \ref{prop:Carleman} are proved by combining the proof of Proposition 7.1 in \cite{Ru1} with the arguments in pages 95-97 of \cite{Ru1} or following the proof of Proposition 1.4 in \cite{RulandTAMS2016}. The estimate for the third term on the right-hand side of \eqref{Carleman.1} follows by arguing as Remark 4 in \cite{Ru1}. For the convenience of the reader, we provide the details of the proof of Proposition \ref{prop:Carleman} below.

\begin{proof}[Proof of Proposition \ref{prop:Carleman}]
 For ${\bf x}=(x,y) \in \mathbb R^{n+1}_+$, we will call $x \in \mathbb R^n$ the tangential variable and $y>0$ the normal variable.
 
 \medskip
 \noindent \emph{Step $1$: Reformulation of the estimate in geodesic polar coordinates}.
Observe that the tangential part of the second order elliptic operator with non-constant coefficients $\mathcal{L}=-\nabla\cdot  \left( a(x) \nabla \right)$ corresponds up to first order operators to the Laplace-Beltrami operator $\Delta_{a^{-1}}$ associated to the metric $a^{-1}=(a^{jk})_{1\le  j, k \le n}$. More precisely, we have 
\begin{equation} \label{Laplace_Beltrami}
\nabla\cdot  \left( a(x) \nabla \right)=\Delta_{a^{-1}}-\frac12 v_{a^{-1}} \cdot \nabla_{a^{-1}}
\end{equation} 
where $v_{a^{-1}}$ is the vector whose $j-th$ component is given by $v_{a^{-1},j}=\text{tr}(a^{-1}\partial_j a)$. 
Then, $U$ is a weak solution of the problem 
\begin{equation}
 \label{eq:U.bis3}
 \begin{aligned}
 \begin{cases}
 \partial_y(y^{1-2\alpha}\partial_yU)+ y^{1-2\alpha}\Delta_{a^{-1}}U=\widetilde{f}, & \text{in}  \quad \mathcal{B}_R^+,\\
 \lim_{y\downarrow 0} \,y^{1-2\alpha} \,\partial_yU=0 & \text{on} \quad  \mathcal{B}_R^0  ,
 \end{cases}
 \end{aligned}
 \end{equation}
with 
\begin{equation*}
\widetilde{f}=f+y^{1-2\alpha}c(x)U+\frac12 y^{1-2\alpha} v_{a^{-1}} \cdot \nabla_{a^{-1}}U .
\end{equation*}

As explained above, we consider $\mathbb R^{n+1}_+$ as a Riemannian manifold with the product metric $1\times a^{-1}$ and we work with the geodesic polar coordinates 
\begin{equation*}
(r,\Theta(\theta_1,\cdots,\theta_n)) \in (0,R)\times \mathbb S^n_+
\end{equation*} 
defined through the exponential map\footnote{The $C^2$ assumption on the coefficients $a_{jk}$ is needed to obtain the uiqueness of the geodesics through the Cauchy-Lipschitz theorem, so that the exponential map is well defined.} by assuming that $R$ is chosen small enough. Here, by convention $\theta_{n}=\frac{y}{|{\bf x}|}$ and $\mathbb S^{n}_+=\{ \Theta(\theta_1,\cdots,\theta_n) \in \mathbb S^n : \theta_n>0 \}.$  It follows from Gauss' lemma (see for example \cite{Jost2011}) that the metric in the geodesic polar coordinates $g_{polar}$ and its inverse $g_{polar}^{-1}$ write 
\begin{equation} \label{metric:geo_polar}
g_{polar}= \begin{bmatrix} 1 & \begin{matrix} 0 & \cdots & 0 \end{matrix}  \\ 
\begin{matrix} 0 \\ \vdots \\[1ex] 0 \end{matrix} & r^2\beta
\end{bmatrix}
\quad \text{and} \quad 
(g_{polar})^{-1}= \begin{bmatrix} 1 & \begin{matrix} 0 & \cdots & 0 \end{matrix}  \\ 
\begin{matrix} 0 \\ \vdots \\[1ex] 0 \end{matrix} & r^{-2}\beta^{-1}
\end{bmatrix}
\end{equation}
where the Riemannian metric with respect to the angular variables $(\theta_1,\cdots,\theta_{d-1})$ on the geodesic half-sphere $S^+_r$ of radius $r$ is written as $r^2 \beta$ in analogy with the euclidean case. Note however that contrary to the euclidean case, the metric $\beta$ may also depend on $r$. The volume element on the geodesic half-sphere $S^+_r$ is given by 
\begin{equation} \label{volume element}
vol_{S_r^+}=r^n\sqrt{|\beta|} \quad \text{where} \quad |\beta|=|\det (\beta_{jk})| .
\end{equation}
Moreover, by using the theory of Jacobi fields (see \cite{DoFef1988} page 169-170)\footnote{The constant $C(a)$ depends on the sectional curvatures which are bounded by the $C^2$-norm of $a$ and on the ellipticity constant $\lambda$ in \eqref{hyp2}.}, there exists $r_0>0$, $c=c(a)$ and $C=C(a)>0$ such that, for all $0<r<r_0$ and $\Theta \in \mathbb S^{n}_+$,
\begin{align} 
&\max_{1 \le i,j \le n}\left| \partial_r\beta_{ij}(r\Theta) \right|+\left| \partial_r (|\beta(r\Theta)|) \right| \le C r , \label{bound_spherical_metric} \\ 
& c \le |\beta(r\Theta)| \le C . \label{bound_spherical_metric.2}
\end{align}

Then, by using \eqref{metric:geo_polar}, the problem \eqref{eq:U.bis3} rewrites in geodesic polar coordinates 
\begin{equation} \label{eq:U.bis3a}
 \begin{aligned}
 \begin{cases}
r^{-n}\theta_n^{1-2\alpha}\partial_r \left( r^{n+1-2\alpha}\partial_rU \right)
+r^{-1-2\alpha}\widetilde{\nabla}_{\mathbb S^n_+}\cdot \theta_n^{1-2\alpha}\widetilde{\nabla}_{\mathbb S^n_+}U&\\
\;\;\;\;\;\;\;\;\;\;\;\;\;\;\;\;\;\;\;\;\;\;\;\;\;\;\;\;\;\;\;\;\;\;\;\;\;\;=f^{\star}, & \text{in} \,  \mathcal{B}_R^+,\\
 \lim_{\theta_n\downarrow 0} \theta_n^{1-2\alpha} \partial_{\theta_n}U=0,  &\text{on}\,  \mathcal{B}_R^0  ,
 \end{cases}
 \end{aligned}
 \end{equation}
 where 
 \begin{equation} \label{def:f_star}
f^{\star}=\widetilde{f}-r^{1-2\alpha}\theta_n^{1-2\alpha} \partial_r\left( \ln \sqrt{|\beta|} \right) \partial_rU.
\end{equation}
and $\widetilde{\nabla}_{\mathbb S^n_+}$ denotes the gradient with respect to the metric $\beta$ on $\mathbb S^n_+$.  More specifically, we have, by denoting $\beta^{-1}=(\beta^{jk})$,
\begin{equation*}
\widetilde{\nabla}_{\mathbb S^n_+} \cdot \theta_n^{1-2\alpha} \widetilde{\nabla}_{\mathbb S^n_+} U=\frac1{\sqrt{|\beta|}} \partial_{\theta_i} \left( \theta_n^{1-2\alpha}\sqrt{|\beta|} \beta^{jk} \partial_{\theta_k} U \right) .
\end{equation*}

\medskip
We claim that it is enough to prove that if $U$ is a weak solution of \eqref{eq:U.bis3a}, then, for $\tau$ large enough,
 \begin{equation} \label{Carleman.1tilde} 
 \begin{split}
  \left\| f^{\star}\right\|_{L^2(\mathcal{B}_R^+)} &\gtrsim\tau^{\frac32} \left\| e^{\tau \phi} (1+\ln^2(r))^{-\frac12}r^{-1}y^{\frac{1-2\alpha}2}U\right\|_{L^2(\mathcal{B}_R^+)} \\ & \quad +  \tau^{\frac12}\left\| e^{\tau \phi} (1+\ln^2(r))^{-\frac12}y^{\frac{1-2\alpha}2}\nabla _{\bf x}U\right\|_{L^2(\mathcal{B}_R^+)}  \\ & \quad + \tau \delta^{-1} \left\| e^{\tau \phi} y^{\frac{1-2\alpha}2}U\right\|_{L^2(\mathcal{A}_{\delta,\kappa\delta}^+)}.
  \end{split}
 \end{equation}

Indeed, assume that this is the case. Notice from \eqref{bound_spherical_metric}-\eqref{bound_spherical_metric.2} that 
\begin{align*}
&\left\| e^{\tau \phi} r y^{\frac{2\alpha-1}2}y^{1-2\alpha}c(x)U\right\|_{L^2(\mathcal{B}_R^+)} \\
&+\lesssim \|c\|_{L^{\infty}}  \left\| e^{\tau \phi} (1+\ln^2(r))^{-\frac12}r^{-1}y^{\frac{1-2\alpha}2}U\right\|_{L^2(\mathcal{B}_R^+)} \\ 
& \le  \frac12 \tau^{\frac32}\left\| e^{\tau \phi} (1+\ln^2(r))^{-\frac12}r^{-1}y^{\frac{1-2\alpha}2}U\right\|_{L^2(\mathcal{B}_R^+)} , 
\end{align*}
and
\begin{align*}
&\left\| e^{\tau \phi} r y^{\frac{2\alpha-1}2}y^{1-2\alpha}v_{a^{-1}} \cdot \nabla_{a^{-1}}U\right\|_{L^2(\mathcal{B}_R^+)}\\
& + \left\| e^{\tau \phi} r y^{\frac{2\alpha-1}2}r^{1-2\alpha}\theta_n^{1-2\alpha} \partial_r\left( \ln \sqrt{|\beta|} \right) \partial_rU\right\|_{L^2(\mathcal{B}_R^+)} \\ &\le C(a)  \left\| e^{\tau \phi} (1+\ln^2(r))^{-\frac12}y^{\frac{1-2\alpha}2}\nabla_{{\bf x}} U\right\|_{L^2(\mathcal{B}_R^+)} \\ & \le \frac12 \tau^{\frac12}  \left\| e^{\tau \phi} (1+\ln^2(r))^{-\frac12}y^{\frac{1-2\alpha}2}\nabla_{{\bf x}} U\right\|_{L^2(\mathcal{B}_R^+)}
\end{align*}
by choosing $0<R<R_0$ small enough and $\tau$ large enough depending possibly on $\|c\|_{L^{\infty}}$ and on the $C^2$-smoothness of $a$. Thus, the extra contributions coming from $y^{1-2\alpha}c(x)U$, $ y^{1-2\alpha} v_{a^{-1}} \cdot \nabla_{a^{-1}}U$ and $r^{1-2\alpha}\theta_n^{1-2\alpha} \partial_r\left( \ln \sqrt{|\beta|} \right) \partial_rU$ can be absorbed by the right-hand side of \eqref{Carleman.1tilde}, which completes the proof of \eqref{Carleman.1}, assuming \eqref{Carleman.1tilde}.
 
 \medskip
 \noindent \emph{Step $2$: Reformulation of the estimate in geodesic conformal coordinates}.
In conformal coordinates $r=e^t$, so that $\partial_rU=e^{-t}\partial_tU$. In particular, the bounds \eqref{bound_spherical_metric}-\eqref{bound_spherical_metric.2} become
\begin{align} 
&\max_{1 \le i,j \le n}\left| \partial_t\left(\beta_{ij}\right)(r\Theta) \right|+\left| \partial_t (|\beta(r\Theta)|) \right| \le Ce^{2t} , \label{bound_conformal_metric.1} \\ 
& c \le  |\beta(r\Theta)| \le C , \label{bound_conformal_metric.2}
\end{align}
for some constants $c=c(a)$ and $C=C(a)$.

After multiplying the equation by $e^{(1+2\alpha)t}$ and conjugating by $$e^{\frac{n-2\alpha}2t} e^{\tau \phi(e^t)}\theta_n^{\frac{1-2\alpha}2},$$ the elliptic operator in \eqref{eq:U.bis3}-\eqref{eq:U.bis3a} writes 
\begin{equation} \label{def:L_phi}
L_{\phi}=e^{\tau \phi} \left(\left( \partial_t^2- \mu^2\right)+\theta_n^{\frac{2\alpha-1}2} \widetilde{\nabla}_{\mathbb S^n_+}  \cdot \theta_n^{1-2\alpha} \widetilde{\nabla}_{\mathbb S^n_+} \theta_n^{\frac{2\alpha-1}2}\right) e^{-\tau \phi}, 
\end{equation} 
with $\mu^2=\frac{(n-2\alpha)^2}4$. More precisely, by defining the new unknowns 
\begin{align*} 
V(t,\theta_1,\cdots,\theta_n)&=e^{\tau \phi} e^{\frac{n-2\alpha}2t} \theta_n^{\frac{1-2\alpha}2} U(e^t\Theta(\theta_1,\cdots,\theta_n)) \\ 
F(t,\theta_1,\cdots,\theta_n)&=e^{\tau \phi}e^{\frac{2\alpha+2+n}2t} \theta_n^{\frac{2\alpha-1}2} f^{\star}(e^t \Theta(\theta_1,\cdots,\theta_n)),
\end{align*}
the problem \eqref{eq:U.bis3} rewrites, after calculations as,
\begin{equation}
 \label{eq:U.4}
 \begin{aligned}
 \begin{cases}
L_{\phi}V= F, & \text{in}  \  \mathcal{B}_R^+,\\
 \lim_{\theta_n\downarrow 0} \,\theta_n^{1-2\alpha} \partial_{\theta_n} \left( \theta_n^{\frac{2\alpha-1}2}V\right)=0,  & \text{on} \  \mathcal{B}_R^0,
 \end{cases}
 \end{aligned}
 \end{equation}
where $L_{\phi}$ is defined in  \eqref{def:L_phi}. Then, the Carleman estimate \eqref{Carleman.1tilde} follows from
\begin{equation} \label{Carleman.3} 
 \begin{aligned}
 \left\| F\right\|_{L^2_{\beta}((-\infty,\ln R) \times \mathbb S^n_+)} &\gtrsim
 \tau^{\frac32} \left\| (\Phi'')^{\frac12}V\right\|_{L^2_{\beta}((-\infty,\ln R) \times \mathbb S^n_+)}\\ 
 & +\tau^{\frac12} \left\| (\Phi'')^{\frac12}\partial_tV\right\|_{L^2_{\beta}((-\infty,\ln R) \times \mathbb S^n_+)} \\ & +  \tau^{\frac12}\left\| (\Phi'')^{\frac12} \theta_n^{\frac{1-2\alpha}2}\left|\widetilde{\nabla}_{\mathbb S^n_+} \left(\theta_n^{\frac{2\alpha-1}2} V\right)\right|_{\beta}\right\|_{L^2_{\beta}((-\infty,\ln R) \times \mathbb S^n_+)} \\ 
 & +\tau \delta^{-1} \left\| e^{t} V\right\|_{L^2_{\beta}((\ln \delta,\ln (\kappa \delta))\times \mathbb S^+_r)}
 \end{aligned}
 \end{equation}
 where $\Phi(t)=\phi(e^t)$ satisfies 
 \begin{align}
 \Phi(t)&=-t+\frac1{10}\left(t \arctan  (t) -\frac12 \ln \left(1+t^2 \right)\right) , \\ 
 \Phi'(t)&=-1+\frac1{10}\arctan (t)<0 , \\
 \Phi''(t)&=\frac1{10}\frac1{1+t^2} >0.
 \end{align}
 In particular, there exists $C>0$ such that, for all $t \in \mathbb R$,
 \begin{align} 
 -2 \le  \Phi'(t)  & \le -\frac45  , \label{prop:Phi'} \\
 0< \Phi''(t) & \le \frac1{10} ,  \label{prop:Phi''} \\ 
 \left| \Phi'''(t) \right| +  \left| \Phi^{(4)}(t) \right| &\le C \Phi''(t) .  \label{prop:Phi3-4}
 \end{align}
 
 Moreover, the $L^2_{\beta}$-norm in the spherical variable $\Theta(\theta_1,\cdots,\theta_n) \in \mathbb S^n_+$ is to be taken with respect to the volume form $\sqrt{|\beta|} \, d\theta_1 \cdots d\theta_n$. More precisely, let $N_+=[0,2\pi) \times [-\frac{\pi}2,\frac{\pi}2) \times \cdots \times [0,\frac{\pi}2)$.
  Then, we  denote 
  $\|V\|_{L^2_\beta}=\sqrt{\left(V,V \right)_{L^2_{\beta}}}$, where $\left( \cdot , \cdot \right)_{L^2_{\beta}}$  is defined by
  \begin{align*}
  \left( V , W \right)_{L^2_{\beta}} = \int_{\mathbb R \times N_+} V(t,\theta_1,\cdots,\theta_n)W(t,\theta_1,\cdots,\theta_n) \sqrt{|\beta|} \, dt d\theta_1 \cdots d\theta_n
  \end{align*}
  To simplify the notation and when it is clear from the context, we will often denote $\|\cdot\|_{L^2_{\beta}}=\|\cdot\|_{L^2_{\beta}((-\infty,\ln R) \times \mathbb S^n_+))}$ and $\left( \cdot , \cdot \right)_{L^2_{\beta}}=\left( \cdot , \cdot \right)_{L^2_{\beta}((-\infty,\ln R) \times \mathbb S^n_+))}$ .
 
 \medskip
 \noindent \emph{Step $3$: Decomposition of $L_{\phi}$ into symmetric and antisymmetric operators}.
To prove estimate \eqref{Carleman.3}, we split the operator $L_{\phi}$ in \eqref{def:L_phi} into its "symmetric" part $\mathcal{S}$ and its "anti-symmetric" part $\mathcal{A}$:  
\begin{equation} \label{decomposition:Lphi}
L_{\phi}= \mathcal{S}+\mathcal{A},
\end{equation}
where
 \begin{align}
 \mathcal{S}&=\partial_t^2+\tau^2(\Phi')^2-\mu^2+\Lambda_{\mathbb S^n_+}, \label{def:S}\\ 
 \Lambda_{\mathbb S^n_+} &=\theta_n^{\frac{2\alpha-1}2} \widetilde{\nabla}_{\mathbb S^n_+} \cdot \theta_n^{1-2\alpha} \widetilde{\nabla}_{\mathbb S^n_+} \theta_n^{\frac{2\alpha-1}2}, \label{def:Lambda}\\
\mathcal{A} &= -2\tau\Phi' \partial_t-\tau\Phi'' . \label{def:A}
 \end{align}
Observe that the operators $\mathcal{S}$ and $\mathcal{A}$ are not actually symmetric and antisymmetric, due to the non-trivial geometry of the problem. 
Note however that the spherical operator $ \Lambda_{\mathbb S^n_+} $ is symmetric  for functions $V$ and $W$ satisfying the Neumann condition\footnote{For functions that do not satisfy the Neumann condition in \eqref{eq:U.4}, the operator $ \Lambda_{\mathbb S^n_+} $ is not symmetric due to the boundary term appearing after the integration by parts. A trace estimate has to be used to deal with these boundary terms. We refer to Lemma 3.1 in \cite{Ru1} for more details.} in \eqref{eq:U.4}. More precisely, it follows from the divergence theorem that 
\begin{equation*}
\left( \Lambda_{\mathbb S^n_+} V, W \right)_{L^2_{\beta}} = \left(  V, \Lambda_{\mathbb S^n_+}W \right)_{L^2_{\beta}} .
\end{equation*}
 
We have by using the decomposition \eqref{decomposition:Lphi} in the first equation of \eqref{eq:U.4} that 
\begin{equation} \label{Carleman:proof.1}
\left\| F \right\|_{L^2_{\beta}}^2=\left\| \mathcal{S}V \right\|_{L^2_{\beta}}^2+\left\| \mathcal{A}V \right\|_{L^2_{\beta}}^2
+2\left( \mathcal{S}V, \mathcal{A}V \right)_{L^2_{\beta}}
\end{equation}
 
  \medskip
 \noindent \emph{Step $4$: Control of the $L^2_{\beta}$-norm of $V$ and $\partial_tV$.}
 We use the contribution of the cross terms. First, we decompose the cross terms as
 \begin{align*}
2& \left( \mathcal{S}V, \mathcal{A}V \right)_{L^2_{\beta}} \\ & =  2 \tau^2\left( (\Phi')^2V, \mathcal{A}V \right)_{L^2_{\beta}} +2 \left( \partial_t^2V, \mathcal{A}V \right)_{L^2_{\beta}} -2\mu^2\left( V, \mathcal{A}V \right)_{L^2_{\beta}}+2\left( \Lambda_{\mathbb S^n_+} V, \mathcal{A}V \right)_{L^2_{\beta}} \\ 
&=:K_1+K_2+K_3+K_4 .
 \end{align*}
 We follow closely the computations on pages 2321-2322 of \cite{RulandTAMS2016} to compute each component $K_i$ by integrating by parts. We obtain that
 \begin{align*}
 K_1 &=4\tau^3 \int_{(-\infty,\ln R) \times N_+} V^2 \left( \Phi'\right)^2 \Phi'' \sqrt{|\beta|}+E_1, \\
 K_2&=4\tau \int_{(-\infty,\ln R) \times N_+} (\partial_tV)^2 \Phi''\sqrt{|\beta|}+E_2 ,\\
 K_3 & = E_3 ,\\ 
 K_4 &= E_4 ,
 \end{align*}
 where 
  \begin{align*}
 E_1 &=2\tau^3 \int_{(-\infty,\ln R) \times N_+} V^2 \left( \Phi'\right)^3 \partial_t\left( \sqrt{|\beta|}\right), \\
 E_2&=2\tau \int_{(-\infty,\ln R) \times N_+} (\partial_tV)^2 \Phi' \partial_t \left(\sqrt{|\beta|}\right)\\&+
 2\tau \int_{(-\infty,\ln R) \times N_+} (\partial_tV) V \Phi'' \partial_t \left(\sqrt{|\beta|}\right) \\
 & \quad -\tau \int_{(-\infty,\ln R) \times N_+} V^2 \Phi''' \partial_t \left(\sqrt{|\beta|}\right)-\tau \int_{(-\infty,\ln R) \times N_+} V^2 \Phi^{(4)} \sqrt{|\beta|}  \\
 &=:E_{2,1}+E_{2,2}+E_{2,3}+E_{2,4} , \\
 E_3 & = -2\mu^2\tau \int_{(-\infty,\ln R) \times N_+} V^2 \Phi' \partial_t\left( \sqrt{|\beta|}\right),\\ 
 E_4 &= -2\tau  \int_{(-\infty,\ln R) \times N_+} \theta_n^{1-2\alpha}\Phi' \partial_{\theta_i} \left( \theta_n^{1-2\alpha}V \right)\partial_{\theta_j} \left( \theta_n^{1-2\alpha}V \right) \partial_t \left( \sqrt{|\beta|} \beta^{ij} \right) .
 \end{align*}
 Note that the computation for $K_4$ is easier than the one on page 2322 of \cite{RulandTAMS2016}, since due to the Neumann condition in \eqref{eq:U.4} no boundary term appear. 
 
 Therefore, it follows combining  \eqref{Carleman:proof.1} with the above computations that
 \begin{align}
\left\| F \right\|_{L^2_{\beta}}^2 &\ge \left\| \mathcal{S}V \right\|_{L^2_{\beta}}^2+\left\| \mathcal{A}V \right\|_{L^2_{\beta}}^2
+4\tau^{3} \left\| \Phi'(\Phi'')^{\frac12}V\right\|_{L^2_{\beta}}^2 \nonumber \\ &\quad +4\tau \left\| (\Phi'')^{\frac12}\partial_tV\right\|_{L^2_{\beta}}^2  -\sum_{j=1}^4 \left| E_j\right| . \label{Carleman:proof.4}
\end{align} 

\medskip
 \noindent \emph{Step $5$: Control of the $L^2_{\beta}$-norm of  $\theta_n^{\frac{1-2\alpha}2}\left|\widetilde{\nabla}_{\mathbb S^n_+} \left(\theta_n^{\frac{2\alpha-1}2} V\right)\right|_{\beta}$.} We use the contribution of $\left\| \mathcal{S}V \right\|_{L^2_{\beta}}^2$. From the definition of $\mathcal{S}$ in \eqref{def:S}, we have
  \begin{align*}
-\tau &\left( \mathcal{S}V, \Phi''V \right)_{L^2_{\beta}} \\ & = - \tau\left( \partial_t^2V,  \Phi''V \right)_{L^2_{\beta}} -\tau^3 \left( \left(\Phi'\right)^2V, \Phi''V \right)_{L^2_{\beta}} \\
&+\mu^2\tau \left( V,  \Phi''V \right)_{L^2_{\beta}}-\tau \left( \Lambda_{\mathbb S^n_+} V,  \Phi''V \right)_{L^2_{\beta}} \\ 
&=:K_5+K_6+K_7+K_8 .
 \end{align*}
 First, we compute integrating by parts in $t$, 
 \begin{align*}
 K_5 &=\tau \int_{(-\infty,\ln R) \times N_+} (\partial_tV)^2  \Phi'' \sqrt{|\beta|}+E_5
 \end{align*}
 where 
 \begin{align*}
 E_5&=\tau \int_{(-\infty,\ln R) \times N_+} (\partial_tV) V \Phi'' \partial_t \left(\sqrt{|\beta|}\right) \\
& \quad  -\frac12\tau \int_{(-\infty,\ln R) \times N_+} V^2 \Phi''' \partial_t \left(\sqrt{|\beta|}\right) \\ & \quad -\frac12\tau \int_{(-\infty,\ln R) \times N_+} V^2 \Phi^{(4)} \sqrt{|\beta|}  \\
 &=:E_{5,1}+E_{5,2}+E_{5,3} . 
 \end{align*}
Next, we have integrating by parts in $\theta_i$ and using the Neumann condition in \eqref{eq:U.4} 
\begin{align*}
K_8&=\tau  \int_{(-\infty,\ln R) \times N_+} \theta_n^{1-2\alpha}\Phi'' \beta^{ij} \partial_{\theta_i} \left( \theta_n^{1-2\alpha}V \right)\partial_{\theta_j} \left( \theta_n^{1-2\alpha}V \right) \sqrt{|\beta|}  \\ 
&=\tau \int_{(-\infty,\ln R) \times N_+} \Phi'' \theta_n^{1-2\alpha}\left|\widetilde{\nabla}_{\mathbb S^n_+} \left(\theta_n^{\frac{2\alpha-1}2} V\right)\right|_{\beta}^2 \sqrt{|\beta|} .
\end{align*}
 Thus, we deduce gathering the above estimates that 
 \begin{align} 
 \tau &\left\| (\Phi'')^{\frac12}\partial_tV\right\|_{L^2_{\beta}}^2+ \tau\left\| (\Phi'')^{\frac12} \theta_n^{\frac{1-2\alpha}2}\left|\widetilde{\nabla}_{\mathbb S^n_+} \left(\theta_n^{\frac{2\alpha-1}2} V\right)\right|_{\beta}\right\|_{L^2_{\beta}}^2 \nonumber \\ 
 & \le \tau \left| \left( \mathcal{S}V, \Phi''V \right)_{L^2_{\beta}} \right|+\tau^3  \left\| \Phi'(\Phi'')^{\frac12}V\right\|_{L^2_{\beta}}^2+\left|E_5\right| \label{Carleman:proof.4a} \\ 
 & \le \left\| \mathcal{S}V \right\|_{L^2_{\beta}}^2+\frac14 \tau^2\left\| \Phi''V \right\|_{L^2_{\beta}}^2+\tau^3  \left\| \Phi'(\Phi'')^{\frac12}V\right\|_{L^2_{\beta}}^2+\left|E_5\right| . \nonumber
 \end{align}

Therefore, we conclude by combining \eqref{Carleman:proof.1} and \eqref{Carleman:proof.4a} and using \eqref{prop:Phi'}-\eqref{prop:Phi''} and $\tau \ge \tau_0 \ge 1$ that
 \begin{align}
\left\| F \right\|_{L^2_{\beta}}^2 &\ge \left\| \mathcal{A}V \right\|_{L^2_{\beta}}^2
+\frac32\tau^{3} \left\| (\Phi'')^{\frac12}V\right\|_{L^2_{\beta}}^2  +5\tau \left\| (\Phi'')^{\frac12}\partial_tV\right\|_{L^2_{\beta}}^2\nonumber \\ &\quad +\tau\left\| (\Phi'')^{\frac12} \theta_n^{\frac{1-2\alpha}2}\left|\widetilde{\nabla}_{\mathbb S^n_+} \left(\theta_n^{\frac{2\alpha-1}2} V\right)\right|_{\beta}\right\|_{L^2_{\beta}}^2  -\sum_{j=1}^5 \left| E_j\right| . \label{Carleman:proof.5}
\end{align} 

 \medskip
 \noindent \emph{Step $6$: Control of the lower order terms.} We estimate the error terms $\sum_{j=1}^5 \left| E_j \right|$. First, observe from \eqref{prop:Phi3-4} that, for $\tau \ge \tau_0$ large enough,
 \begin{equation*}
 \left|E_{2,4} \right|+ \left|E_{5,3} \right| \le C \tau  \left\| (\Phi'')^{\frac12}V\right\|_{L^2_{\beta}}^2 \le 2^{-2}\tau^3\left\| (\Phi'')^{\frac12}V\right\|_{L^2_{\beta}}^2 .
 \end{equation*}
  Next, from \eqref{bound_conformal_metric.1}-\eqref{bound_conformal_metric.2}, there exists $C=C(a)>1$ and $0<R_0=R_0(a)<1$ such that, for all $t \le \ln R_0$,
  \begin{equation*}
 \partial_t\left( \sqrt{|\beta|} \right) \le C e^{2t} \sqrt{|\beta|} \le 2^{-20} \Phi''(t) .
  \end{equation*}
 This bound combined to \eqref{prop:Phi'}-\eqref{prop:Phi3-4} yields for $t \le \ln R_0$ and $\tau \ge \tau_0$ large enough, 
 \begin{align*}
 \left|E_{1} \right|+ \left|E_{2,3} \right|+ \left|E_{3} \right|+\left|E_{5,2} \right| & \le 2^{-3}\tau^3\left\| (\Phi'')^{\frac12}V\right\|_{L^2_{\beta}}^2 , \\ 
  \left|E_{2,1} \right| & \le 2^{-3}\tau\left\| (\Phi'')^{\frac12}\partial_tV\right\|_{L^2_{\beta}}^2 , \\
   \left|E_{2,2} \right| +\left|E_{5,1} \right| & \le 2^{-4}\tau\left\| (\Phi'')^{\frac12}\partial_tV\right\|_{L^2_{\beta}}^2+2^{-4}\tau\left\| (\Phi'')^{\frac12}V\right\|_{L^2_{\beta}}^2 .
 \end{align*}
  Finally, by deriving the identity $\beta^{ij}\beta_{jk}=\delta_{ik}$, we find that 
  \begin{equation*}
  \partial_t\left( \beta^{ij} \right)=-\beta^{ij} \partial_t \left( \beta_{jk}\right) \beta^{kj} .
  \end{equation*}
  Thus, it follows from \eqref{bound_conformal_metric.1} that, for $t \le \ln R_0$ (where $R_0=R_0(a)$ is small enough), 
  \begin{equation*}
   \left|E_4 \right| \le \frac12 \tau\left\| (\Phi'')^{\frac12} \theta_n^{\frac{1-2\alpha}2}\left|\widetilde{\nabla}_{\mathbb S^n_+} \left(\theta_n^{\frac{2\alpha-1}2} V\right)\right|_{\beta}\right\|_{L^2_{\beta}}^2 .
  \end{equation*}
  
  Therefore, we conclude combining these estimates with \eqref{Carleman:proof.5} that 
  \begin{align}
\left\| F \right\|_{L^2_{\beta}}^2 &\ge \left\| \mathcal{A}V \right\|_{L^2_{\beta}}^2
+\tau^{3} \left\| (\Phi'')^{\frac12}V\right\|_{L^2_{\beta}}^2  +4\tau \left\| (\Phi'')^{\frac12}\partial_tV\right\|_{L^2_{\beta}}^2\nonumber \\ &\quad +\frac12\tau\left\| (\Phi'')^{\frac12} \theta_n^{\frac{1-2\alpha}2}\left|\widetilde{\nabla}_{\mathbb S^n_+} \left(\theta_n^{\frac{2\alpha-1}2} V\right)\right|_{\beta}\right\|_{L^2_{\beta}}^2  . \label{Carleman:proof.6}
\end{align} 
This provides the bounds for the first three terms on the right-hand side of \eqref{Carleman.3}.

  \medskip
 \noindent \emph{Step $7$: Estimate for the last term in \eqref{Carleman.3}}.
 The proof of this estimate uses of the contribution of $\|\mathcal{A}V\|_{L^2_{\beta}}$ in \eqref{Carleman:proof.6}. Recall that $U$ is supported in the half geodesic annulus $\mathcal{A}^+_{\delta,R}$ with  $\kappa_0\delta<R$ and that $1<\kappa<\kappa_0$. 
 
 On the one hand, we have from the fundamental theorem of calculus, for all $t \in (\ln(\delta),\ln(\kappa \delta))$ and  $\Theta \in \mathbb S^n_+$,
\[V(t,\Theta) \le (t-\ln (\delta))^{\frac12}\left( \int_{\ln(\delta)}^t(\partial_tV)^2(\tilde{t},\theta)d\tilde{t}\right)^{\frac12}\]
so that
\begin{equation} \label{Carleman.4}
\int_{\ln(\delta)}^{\ln(\kappa \delta)}V^2(t,\theta)e^{2t}dt \lesssim \delta^2 \int_{\ln(\delta)}^{\ln(\kappa \delta)}(\partial_tV)^2(t,\theta)dt .
\end{equation}
On the other hand, it follows from the definition of $\mathcal{A}$ and the triangle inequality that 
\begin{equation} \label{Carleman.5}
\|\mathcal{A}V\|_{L^2_{\beta}} \geq \tau \|\Phi' \partial_tV \|_{L^2_{\beta}}-\tau \|\Phi'' V \|_{L^2_{\beta}} .
\end{equation}
While the second term on the right-hand side of \eqref{Carleman.5} can be absorbed by the second term on the right-hand side of \eqref{Carleman:proof.6} by choosing $\tau$ large enough, we use that $\Phi' \sim 1$ and \eqref{Carleman.4} to bound by below the first term on the right-hand side of \eqref{Carleman.5}. We obtain 
\begin{equation*}
\tau \|\Phi' \partial_tV \|_{L^2_{\beta}((\ln\delta,\ln (c\delta))\times \mathbb S^n_+)}  \gtrsim \tau \delta^{-1}\left\| e^{t} V\right\|_{L^2_{\beta}((\ln \delta,\ln (c\delta))\times \mathbb S^n_+)} .
\end{equation*} 
These estimates together with \eqref{Carleman:proof.6} yield the bound for the last term on the right-hand side of \eqref{Carleman.3}.
\end{proof}

\subsection{Doubling estimate.} From the Carleman estimate, we prove that solutions of \eqref{eq:U.3} satisfy the doubling property. 

\begin{proposition} [Doubling property] \label{prop:doubling} Let $\alpha \in (0,1)$. There exists $R_1=R_1(a)$ with $0<R_1<R_0$ such that, for any $0<\widetilde{R}<R_1$, the following holds. Assume that $U \in H^1(\mathcal{B}^+_{\widetilde{R}},y^{1-2\alpha}d{\bf x})$ 
is a solution of \eqref{eq:U.2} in $\mathcal{B}^+_{\widetilde{R}}$ with  $U \equiv 0$ on $\mathcal{B}^0_{\widetilde{R}} $. Then, there exists $C=C(a,U)>1$ such that, for any $0<R<\widetilde{R}/2$, 
\begin{equation} \label{doubling.1}
\|y^{\frac{1-2\alpha}2}U\|_{L^2(\mathcal{B}_{2R}^+)} \le C \|y^{\frac{1-2\alpha}2}U\|_{L^2(\mathcal{B}_R^+)} .
\end{equation}
\end{proposition}

In the proof of Proposition \ref{prop:doubling}, we need the following elliptic gradient estimate, which follows from (2.3.2) in \cite{FKS}. 
\begin{lemma}[Cacciopoli estimate] \label{lemm:Cacciopoli}
Let $\alpha \in (0,1)$ . There exists $R_2=R_2(a)$ such that for $0<R<R_2$ and $0<\kappa<1$, the following holds. Assume that $U \in H^1(\mathcal{B}^+_{ R},y^{1-2\alpha}d{\bf x})$ is a solution of \eqref{eq:U.2} satisfying $U \equiv 0$ on $\mathcal{B}^0_{R}$. Then, there exists $C=C(a)>0$ such that 
\begin{equation} \label{Caccipolli:est.1}
\left\|  y^{\frac{1-2\alpha}2}\nabla_{{\bf x}}U\right\|_{L^2(\mathcal{B}^+_{\kappa R})} \le \left(\frac{C}{(1-\kappa)R}+\|c\|_{L^{\infty}} \right)\left\|  y^{\frac{1-2\alpha}2} U\right\|_{L^2(\mathcal{B}^+_{ R})} .
\end{equation}
\end{lemma} 

\begin{proof}
Let $\psi=\psi_{\kappa,R}$ be a smooth radial cut-off function such that 
\begin{equation} \label{def:psi}
0 \le \psi \le 1, \quad \text{supp} \, \psi \subset \overline{\mathcal{B}^+_{\frac{(\kappa+1)R}2}} \quad \text{and} \quad \psi \equiv 1 \ \text{on} \ \mathcal{B}^+_{\kappa R} .
\end{equation} 
By abuse of notation, we denote $\psi({\bf x})=\psi(r)$ with $r=1_{a^{-1}}({\bf x})$. Then, there exist $C=C(a)>0$ and $R_2=R_2(a)>0$ such that, for $0<R<R_2(a)$,
\begin{equation} \label{prop:psi}
|\nabla_{{\bf x}} \psi(r)| \le C \left((1-\kappa)R\right)^{-1}, \ \text{for} \ r \in (\kappa R,R) .
\end{equation} 
Observe that the equation in \eqref{eq:U.2} can be rewritten as 
\begin{equation} \label{eq:U.2.div}
\nabla_{{\bf x}} \cdot (y^{1-2\alpha}\widetilde{a}(x)\nabla_{{\bf x}}U)=0 ,
\end{equation}
where ${\bf x}=(x,y)$ and  
$\widetilde{a}(x)= 
\begin{pmatrix}  
 a(x) & 0 \\ 
 0 & 1
 \end{pmatrix} .
$ It follows by multiplying \eqref{eq:U.2.div} by $\psi^2 U$ and using the divergence theorem, (recall here that $U \equiv 0$ on $\mathcal{B}^+_{R} \cap \{y=0\}$), that
\begin{align*} 
\int_{\mathcal{B}^+_{R} }\psi^2 y^{1-2\alpha}\widetilde{a}(x)\nabla_{{\bf x}}U \cdot \nabla_{{\bf x}}U    d{\bf x}&=
-2\int_{\mathcal{B}^+_{R} }\psi U y^{1-2\alpha}\widetilde{a}(x)\nabla_{{\bf x}}U \cdot \nabla_{{\bf x}}\psi    d{\bf x} \\ & \quad 
-\int_{\mathcal{B}^+_{R} }\psi^2 y^{1-2\alpha}c(x)U^2 d{\bf x} .
\end{align*}

Thus, we deduce from the ellipticity assumption on $a=(a_{jk})$ and the Cauchy-Schwarz inequality,
\begin{align*} 
(1+\lambda)&\int_{\mathcal{B}^+_{R} } \psi ^2 y^{1-2\alpha} \left|\nabla_{{\bf x}}U \right|^2    d{\bf x}  
\\ &\le 2(1+\Lambda) \left(\int_{\mathcal{B}^+_{R} }\psi^2 y^{1-2\alpha}\left|\nabla_{{\bf x}}U\right|^2   d{\bf x}\right)^{\frac12}\left(\int_{\mathcal{B}^+_{R} }\ y^{1-2\alpha} U^2 \left|\nabla_{{\bf x}}\psi\right|^2   d{\bf x}\right)^{\frac12} \\ & \quad +\|c\|_{L^{\infty}}\int_{\mathcal{B}^+_{R} }\ y^{1-2\alpha} U^2    d{\bf x}
\end{align*}
This concludes the proof of \eqref{Caccipolli:est.1} by using the properties of $\psi$ in \eqref{def:psi}-\eqref{prop:psi}. 
\end{proof} 

\begin{proof}[Proof of Proposition \ref{prop:doubling}]
Let $0<\delta<1$ and $\widetilde{R}>0$ satisfy $0<3\delta<\widetilde{R}<\frac{R_0}3$ and let $\eta_{\delta,\widetilde{R}}=\eta$ be a smooth radial cut-off function such that 
\begin{equation} \label{def:eta}
0 \le \eta \le 1, \quad \text{supp} \, \eta \subset \overline{\mathcal{A}^+_{\frac{\delta}2,2\widetilde{R}}} \quad \text{and} \quad \eta \equiv 1 \ \text{on} \ \mathcal{A}^+_{\delta,\widetilde{R}} .
\end{equation} 
By abuse of notation, we denote $\eta({\bf x})=\eta(r)$ with $r=1_{a^{-1}}({\bf x})$. Then, there exist $C=C(a)>0$ and $R_3=R_3(a)>0$ such that, for $0<R<R_3(a)$,
\begin{equation} \label{prop:eta}
|\nabla_{{\bf x}} \eta(r)| \le C \delta^{-1}, \ \text{for} \ r \in (\frac{\delta}2,\delta) \quad \text{and} \quad |\nabla_{{\bf x}} \eta(r)| \le C \widetilde{R}^{-1}  \ \text{for} \ r \in (\widetilde{R},2\widetilde{R}) .
\end{equation} 
Hence, $\widetilde{U}=\eta U$ is a solution of \eqref{eq:U.3} with $\widetilde{U}=0$ on $\mathcal{B}_R^0$, and
\begin{align*} 
f&= \partial_y(y^{1-2\alpha}\partial_y(\eta U))+\nabla_{x} \cdot (y^{1-2\alpha}a(x)\nabla_{x}(\eta U))-y^{1-2\alpha}c(x)\eta U  \\ 
&=y^{1-2\alpha} \left( \partial_y \eta \partial_yU+ a(x)\nabla_x \eta \cdot \nabla_x U\right)\\
&+ \partial_y \left( y^{1-2\alpha}U \partial_y \eta \right)+ \nabla_x \cdot (y^{1-2\alpha}a(x)U \nabla_x \eta) . \label{def:f}
\end{align*} 
From the Carleman estimate \eqref{Carleman.1}, we have, for $\tau \ge \tau_0$,
\begin{align*}  
 \tau \delta^{-1} \left\| e^{\tau \phi} y^{\frac{1-2\alpha}2} U\right\|_{L^2(\mathcal{A}_{\delta,3\delta}^+)}  + \tau^{\frac32} &\left\| e^{\tau \phi} (1+\ln^2(r))^{-\frac12}r^{-1}y^{\frac{1-2\alpha}2}U\right\|_{L^2(\mathcal{A}_{\widetilde{R}/4,\widetilde{R}/2}^+)} \\ &
 \lesssim   \left\| e^{\tau \phi} r y^{\frac{2\alpha-1}2}f\right\|_{L^2(\mathcal{A}^+_{\delta/2,2\widetilde{R}})} ,
 \end{align*}
 so that it follows from the properties of $\eta$ in \eqref{def:eta}-\eqref{prop:eta},
 \begin{align*} 
 \tau \delta^{-1}& \left\| e^{\tau \phi} y^{\frac{1-2\alpha}2} U\right\|_{L^2(\mathcal{A}_{\delta,3\delta}^+)} +
  \tau^{\frac32} \left\| e^{\tau \phi} (1+\ln^2(r))^{-\frac12}r^{-1}y^{\frac{1-2\alpha}2}U\right\|_{L^2(\mathcal{A}_{\widetilde{R}/4,\widetilde{R}/2}^+)}
  \\ & \lesssim  \left( \left\| e^{\tau \phi}  y^{\frac{1-2\alpha}2}\nabla_{{\bf x}}U\right\|_{L^2(\mathcal{A}^+_{\delta/2,\delta})}
 +\delta^{-1}\left\| e^{\tau \phi}  y^{\frac{1-2\alpha}2}U\right\|_{L^2(\mathcal{A}^+_{\delta/2,\delta})} \right) \\ & 
\quad + \left( \left\| e^{\tau \phi}  y^{\frac{1-2\alpha}2}\nabla_{{\bf x}}U\right\|_{L^2(\mathcal{A}^+_{\widetilde{R},2\widetilde{R}})}
 +\widetilde{R}^{-1}\left\| e^{\tau \phi}  y^{\frac{1-2\alpha}2}U\right\|_{L^2(\mathcal{A}^+_{\widetilde{R},2\widetilde{R}})} \right) .
 \end{align*}
 Multiplying by $\delta$, recalling that $\phi$ is a decreasing function and $\tau \ge 1$ and then using the gradient elliptic estimate \eqref{Caccipolli:est.1}, we deduce that 
 \begin{align*}
 e^{\tau \phi(3\delta)} &\left\| y^{\frac{1-2\alpha}2} U\right\|_{L^2(\mathcal{A}_{\delta,3\delta}^+)} +\delta e^{\tau \phi(\widetilde{R}/2)}\widetilde{R}^{-1} (1+\ln^2(\widetilde{R}))^{-\frac12}\left\| y^{\frac{1-2\alpha}2}U\right\|_{L^2(\mathcal{A}_{\widetilde{R}/4,\widetilde{R}/2}^+)}
 \\ & \lesssim e^{\tau \phi(\delta/2)}\left( \delta \left\|  y^{\frac{1-2\alpha}2}\nabla_{{\bf x}}U\right\|_{L^2(\mathcal{A}^+_{\delta/2,\delta})}
 +\left\|   y^{\frac{1-2\alpha}2}U\right\|_{L^2(\mathcal{A}^+_{\delta/2,\delta})} \right) \\ 
 &\quad +e^{\tau \phi(R)}\left( \delta \left\|   y^{\frac{1-2\alpha}2}\nabla_{{\bf x}}U\right\|_{L^2(\mathcal{A}^+_{R,2R})}
 +\delta \widetilde{R}^{-1}\left\| y^{\frac{1-2\alpha}2}U\right\|_{L^2(\mathcal{A}^+_{\widetilde{R},2\widetilde{R}})} \right) \\ 
 & \quad \lesssim e^{\tau \phi(\delta/2)}\left\|  y^{\frac{1-2\alpha}2}U\right\|_{L^2(\mathcal{B}^+_{\frac{3\delta}2})}+\delta e^{\tau \phi(\widetilde{R})}\widetilde{R}^{-1}\left\| y^{\frac{1-2\alpha}2}U\right\|_{L^2(\mathcal{B}^+_{3\widetilde{R}})} .
 \end{align*}
 Now, we choose $\tau$ large enough, and independent of $\delta$, such that the last term on the right-hand side of the above inequality can be absorbed by the second term on the left-hand side. For example, $\tau$ can be chosen on the form
 \begin{equation*} 
 \tau \sim \ln \left(\left\| y^{\frac{1-2\alpha}2}U\right\|_{L^2(\mathcal{B}^+_{3\widetilde{R}})}\right)-\ln \left(\left\| y^{\frac{1-2\alpha}2}U\right\|_{L^2(\mathcal{A}_{\widetilde{R}/4,\widetilde{R}/2}^+)} \right) .
 \end{equation*}
 Thus, we conclude observing that $\phi(\delta/2)-\phi(3\delta) \sim 1$ that 
  \begin{align*}
 \left\| y^{\frac{1-2\alpha}2} U\right\|_{L^2(\mathcal{A}_{\delta,3\delta}^+)} \lesssim \left\|  y^{\frac{1-2\alpha}2}U\right\|_{L^2(\mathcal{B}^+_{\frac{3\delta}2})} ,
 \end{align*}
 where the implicit constant does not depend on $\delta$ but possibly on $U$ through the choice of $\tau$. 
 Therefore, we conclude the proof of \eqref{doubling.1} by adding $\left\|  y^{\frac{1-2\alpha}2}U\right\|_{L^2(\mathcal{B}^+_{\delta})}$ on both sides of the inequality and taking $R=\frac{3\delta}2$.
\end{proof}

\subsection{Proof of Proposition \ref{WUCP:ext}} We start by recalling the following observation about the metric $1 \times a^{-1}$ on $\mathbb R^{n+1}_+$, introduced in the sub-section Caleman estimats, which is a consequence of the ellipticity of $a$ and its $C^2$-smoothness (see hypothesis of Theorem \ref{A3}). For $R>0$, let $B_R^+$ be the euclidian half-ball defined in \eqref{euc:ball} and $\mathcal{B}_R^+$ be the geodesic half-ball defined in \eqref{geo:ball}. Then, there exist $R_3=R_3(a)>0$ and $0<\alpha_0<\alpha_1$ such that for $0<r<R_3$, we have 
\begin{equation} \label{ineq:ball}
B_{\alpha_0 r}^+ \subset \mathcal{B}_r^+ \subset B_{\alpha_1 r}^+ . 
\end{equation} 
Assume now that $\omega$ is a non-negative measure and that there exists $D>0$ satisfying $\omega ( \mathcal{B}_{2r}^+) \le D \omega ( \mathcal{B}_{r}^+)$, for $r>0$ small enough. Then, there exist $\widetilde{D}>0$, depending only on $\alpha_0$, $\alpha_1$ and $D$ such that $\omega ( B_{2r}^+) \le \widetilde{D} \omega ( B_{r}^+)$, for $r>0$ small enough. As a consequence of this fact and Proposition \ref{prop:doubling}, there exists $R_4=R_4(a)>0$ such that, for all $0<R<R_4$,
\begin{equation} \label{doubling.euclidian}
\|y^{\frac{1-2\alpha}2}U\|_{L^2(B_{2R}^+)} \le C \|y^{\frac{1-2\alpha}2}U\|_{L^2(B_R^+)} .
\end{equation}

Next, let $W_{\sigma}(x,y)=U(\sigma x,\sigma y)$, where $\sigma>0$ is small. For $r_0>0$ small and fixed and for $\sigma>0$ small, observe from \eqref{ineq:ball} and Lemma \ref{lemm:Cacciopoli} that
\begin{align*} 
 \sigma \left\| y^{\frac{1-2\alpha}2} \nabla_{x,y}U \right\|_{L^2(B^+_{\sigma r_0})} 
 & \le\sigma \left\| y^{\frac{1-2\alpha}2} \nabla_{x,y}U \right\|_{L^2(\mathcal{B}^+_{\frac{\sigma r_0}{\alpha_0}})} \\ 
& \le C_{\alpha_0,r_0} \left\| y^{\frac{1-2\alpha}2} U \right\|_{L^2(\mathcal{B}^+_{\frac{2\sigma r_0}{\alpha_0}})} 
\\ &  \le C_{\alpha_0,r_0,r_1} \left\| y^{\frac{1-2\alpha}2} U \right\|_{L^2(B^+_{\frac{2 \alpha_1 \sigma r_0}{\alpha_0}})} \\ & 
\le C_{\alpha_0,r_0,r_1}\left\| y^{\frac{1-2\alpha}2}U \right\|_{L^2(B^+_{\sigma r_0})},
\end{align*}
where we applied \eqref{doubling.euclidian} a finite number of times in the last inequality. Thus, we have proved that
\begin{align*} 
\left\| y^{\frac{1-2\alpha}2} \nabla_{x,y}W_{\sigma} \right\|_{L^2(B^+_{r_0})} & = \sigma^{1-\frac{1-2\alpha}2-\frac{n+1}2}\left\| y^{\frac{1-2\alpha}2} \nabla_{x,y}U \right\|_{L^2(B^+_{\sigma r_0})}  
\\ & \leq C_{\alpha_0,r_0,r_1}  \sigma^{-\frac{1-2\alpha}2-\frac{n+1}2}\left\| y^{\frac{1-2\alpha}2}U \right\|_{L^2(B^+_{\sigma r_0})}
\\ & = C_{\alpha_0,r_0,r_1}\left\| y^{\frac{1-2\alpha}2}W_{\sigma} \right\|_{L^2(B^+_{r_0})},
\end{align*}
for some positive constant $C_{\alpha_0,r_0,r_1}=C_{\alpha_0,r_0,r_1}(a,U)$.

For $\sigma>0$ small enough, we now define 
\begin{equation} \label{def:Usigma}
U_{\sigma}(x,y)= \frac{W_{\sigma}(x,y)}{ \left\|y^{\frac{1-2\alpha}2}W_{\sigma}\right\|_{L^2(B_{r_0}^+)}} , 
\end{equation}
where we assumed by contradiction that $\left\|y^{\frac{1-2\alpha}2}W_{\sigma}\right\|_{L^2(B_{r_0}^+)} \neq 0$, otherwise $U$ would be identically $0$ by using the Weak Unique Continuation Property for elliptic equation with $C^2$ coefficients in the domain $ B^+_{ r_0} \cap \{y>\epsilon\}$, (see \cite{AronJMPA1957,Cordes1956}).
Then, we have, for all $\sigma>0$ small enough, 
\begin{align} \label{bounds:Usigma}
\left\| y^{\frac{1-2\alpha}2} U_{\sigma} \right\|_{L^2(B^+_{r_0})}=1 \quad \text{and} \quad \left\| y^{\frac{1-2\alpha}2} \nabla_{x,y}U_{\sigma} \right\|_{L^2(B^+_{r_0})} \le C .
\end{align}

We show that $U_{\sigma}$ is a weak solution in $B_{r_0}^+$ of 
\begin{equation}
 \label{eq:U.10}
 \begin{aligned}
 \begin{cases}
 \partial_y(y^{1-2\alpha}\partial_yU)+\nabla_{x} \cdot (y^{1-2\alpha}a_{\sigma}(x)\nabla_{x}U)-y^{1-2\alpha}c_{\sigma}(x)U=0 & \text{in}  \quad B_{r_0}^+;\\
 \lim_{y\downarrow 0} \,y^{1-2\alpha} \,\partial_yU=0 & \text{on} \quad \Gamma_{r_0}^0 ,
 \end{cases}
 \end{aligned}
 \end{equation}
where $a_{\sigma}(x)=a(\sigma x)$, $c_{\sigma}(x)=\sigma^2 c(\sigma x)$. Indeed, we know that 
\[ \left\| y^{\frac{1-2\alpha}2} \nabla_{x,y}U_{\sigma} \right\|_{L^2(B^+_{r_0})} \le C . \] 
Now, let $\xi \in C^1(\overline{B_{r_0}^+})$ such that $\xi \equiv 0$ on $\Gamma_{r_0}^+$. Then, 
\begin{equation} \label{weak:Usigma}
\begin{split}
&\int_{B_{r_0}^+}y^{1-2\alpha} \left( \partial_yU_{\sigma} \partial_y\xi+a_{\sigma}(x)\nabla_xU_{\sigma} \cdot \nabla_x \xi+c_{\sigma}(x) U \xi\right)dxdy \\ &=
C_{\sigma,n}\int_{B_{\sigma r_0}^+}y^{1-2\alpha} \left(\partial_yU \partial_y\xi_{\sigma}+a(x)\nabla_xU \cdot \nabla_x \xi_{\sigma}+c(x)U \xi_{\sigma}\right)dxdy ,
\end{split}
\end{equation}
where $\xi_{\sigma}(x,y)=\xi(\sigma^{-1}x,\sigma^{-1}y)$. Note that $\xi_{\sigma} \in C^1(\overline{B_{\sigma r_0}^+})$ is such that $\xi_{\sigma} \equiv 0$ on $\Gamma_{\sigma r_0}^+$, and recall that $\sigma r_0<r_0<R$ where $R$ is an in Proposition \ref{WUCP:ext}. Let $\widetilde{\xi}_{\sigma}$ be defined by $\widetilde{\xi}_{\sigma}=\xi_{\sigma}$ on $B_{\sigma r_0}^+$ and $\widetilde{\xi}_{\sigma}=0$ on $B_R^+ \setminus B_{\sigma r_0}^+$. Even though $\widetilde{\xi}_{\sigma}$ need not belong to $C^{1}(\overline{B_R^+})$, it is certainly Lipschitz in $B_R^+$ and a simple regularization argument together with the fact that $U$ is a weak solution of \eqref{eq:U.2} in $B^+_R$, gives now that 
\begin{align*}
\int_{B_{R}^+}y^{1-2\alpha} \left( \partial_yU \partial_y\widetilde{\xi}_{\sigma}+a(x)\nabla_xU \cdot \nabla_x \widetilde{\xi}_{\sigma}+c(x)U \widetilde{\xi}_{\sigma} \right)dxdy =0 .
\end{align*}
Then, we conclude from \eqref{weak:Usigma} that $U_{\sigma}$ is a weak solution of \eqref{eq:U.10} in $B_{r_0}^+$.

Next, we claim that there exist $U_0 \in H^1(B^+_{r_0},y^{1-2\alpha}dxdy)$  and a sequence  $\sigma_k \to 0$ and such that 
\begin{equation} \label{Usigma:conv}
\begin{cases}
 U_{\sigma_k} \underset{k \to +\infty}{\longrightarrow} U_0 \ \text{weakly in} \, H^1(B^+_{r_0},y^{1-2\alpha}dxdy) ;  \\ 
 U_{\sigma_k} \underset{k \to +\infty}{\longrightarrow} U_0 \ \text{strongly in} \, L^2(B^+_{r_0},y^{1-2\alpha}dxdy) .
 \end{cases}
 \end{equation}
Indeed, the weak convergence in \eqref{Usigma:conv} follows directly from the bounds \eqref{bounds:Usigma} and the Banach-Alaoglu theorem. To prove the strong convergence in \eqref{Usigma:conv}, we first observe from the weighted Sobolev embedding Theorem 1.3 in \cite{FKS} (recall that $y^ {1-2\alpha}$ is an $A_2$-weight for $0<\alpha<1$) and the bounds in \eqref{bounds:Usigma} that $\left\| U_{\sigma} \right\|_{L^{2p}(B^+_{r_0},y^{1-2\alpha})} \le C$, for some $p>1$.  Thus, given $\epsilon>0$, there exists $0<r_1<r_0$ such that 
\[ \int_{B_{r_0}^+ \setminus B_{r_1}^+} U_{\sigma}^2 \, y^{1-2\alpha}  dxdy \le C \left( \int_{B_{r_0}^+ \setminus B_{r_1}^+} y^ {1-2\alpha} dxdy \right)^{\frac1{p'}} \le \epsilon .\]
On the other hand, recall that $U_{\sigma}(x,0)=0$ on $\Gamma^0_{r_0}$. Thus, we deduce from the $C^{0,\alpha}$ bound in Theorem 2.4.6  in \cite{FKS} that $\{U_{\sigma}\}$ is equicontinuous in $\overline{B^+_{r_1}}$, so that we conclude by using the Arzela-Ascoli theorem.

We can also assume without loss of generality that $a(0)=I$. Moreover, it follows from the doubling property \eqref{doubling.euclidian} and the $L^{\infty}$ estimate in \cite{FKS} and the $C^{0,\alpha}$ estimate in Theorem 2.4.6 in \cite{FKS} , that, for some $0<\beta<1$ and any $k$,
\begin{equation} \label{bounds:Usigma.2}
\left\| y^{\frac{1-2\alpha}2} U_{\sigma_k} \right\|_{L^2(B^+_{\frac{r_0}2})} \sim 1, \quad \left\|U_{\sigma_k} \right\|_{L^{\infty}(B^+_{\frac{r_0}2})}   \lesssim 1  \quad \text{and} \quad  \left\|U_{\sigma_k} \right\|_{C^{0,\beta}(\overline{B^+_{\frac{r_0}2}})} \lesssim 1 .
\end{equation} 
with all the implicit constants being uniform in $k$. Moreover, since $U  \equiv 0$ on $\Gamma_R^0$, we see that $U_{\sigma_k} \equiv 0$ on $\Gamma_{\frac{r_0}2}^0$. Thus $U_0$ inherits all these properties in the limit. 

Finally, we note that $U_0$ is a weak solution of 
\begin{equation}
 \label{eq:U.11}
 \begin{aligned}
 \begin{cases}
 \partial_y(y^{1-2\alpha}\partial_yU)+\nabla_{x} \cdot (y^{1-2\alpha}\nabla_{x}U)=0, & \text{in}  \quad B_{r_0}^+;\\
 \lim_{y\downarrow 0} \,y^{1-2\alpha} \,\partial_yU=0 & \text{on} \quad \Gamma_{r_0}^0 .
 \end{cases}
 \end{aligned}
 \end{equation}
 Indeed, we already saw that $U_{\sigma_k}$ is a weak solution to \eqref{eq:U.10}. Let $\xi \in C^1(\overline{B_{r_0}^+})$ such that $\xi \equiv 0$ on $\Gamma_{r_0}^+$. Then, it follows that 
 \begin{align*} 
0&
=\int_{B_{r_0}^+}y^{1-2\alpha} \left( \partial_yU_{\sigma_k} \partial_y\xi+a(0)\nabla_xU_{\sigma_k} \cdot \nabla_x \xi\right)dxdy \\ 
& \quad +\int_{B_{r_0}^+}y^{1-2\alpha} \left( \partial_yU_{\sigma_k} \partial_y\xi+\left(a(\sigma_k x)-a(0)\right)\nabla_xU_{\sigma_k} \cdot \nabla_x \xi\right)dxdy  \\ & \quad
+\int_{B_{r_0}^+}y^{1-2\alpha} \sigma_k^2c(\sigma_k x)U_{\sigma_k}  \xi dxdy 
\\ & =I_k+II_k+III_k .
\end{align*}
On the one hand, using the first convergence in \eqref{Usigma:conv}, we see that 
\begin{equation*} 
I_k \underset{k \to +\infty}{\longrightarrow} \int_{B_{r_0}^+}y^{1-2\alpha} \left( \partial_yU_{0} \partial_y\xi+a(0)\nabla_xU_{0} \cdot \nabla_x \xi\right)dxdy .
\end{equation*}
On the other hand, observe from the regularity assumption on $a$ that $a(\sigma_k \cdot) \to a(0)$ and $\sigma_k^2c(\sigma_k \cdot) \to 0$ as $k \to +\infty$ uniformly in $B_{r_0}^+$. Then, we deduce from the uniform boundedness of $U_{\sigma_k}$ in $H^1(B^+_{r_0},y^{1-2\alpha}dxdy)$ that $II_{k} \to 0$ and $III_k \to 0$ as $k \to +\infty$. Thus, recalling  $a(0)=I$, we conclude that $U_0$ is a weak solution of \eqref{eq:U.11} in the ball $B^+_{r_0}$ with $U_0 \equiv 0$ on $\Gamma_{\frac{r_0}2}^0$. 

An argument similar to the one used before for $U$ shows that $U_0$ is also a weak solution of \eqref{eq:U.11} in the ball $B^+_{\frac{r_0}2}$ with $U_0 \equiv 0$ on $\Gamma_{\frac{r_0}2}^0$. Hence, by Proposition 2.2 in \cite{Ru1}, $U_0 \equiv 0$ on $B^+_{\frac{r_0}2}$. This contradicts the fact that $\left\| y^{\frac{1-2\alpha}2} U_0 \right\|_{L^2(B^+_{\frac{r_0}2})} \sim 1$, which finishes the proof.

\vspace{5mm}
\noindent\underline{\bf Acknowledgements.} C.E.K.  was supported by the NSF grant DMS-2153794, D. P. was supported by a Trond Mohn Foundation grant. L. V. was supported in part by MICINN (Spain) projects Severo Ochoa CEX2021-001142, and PID2021-126813NB-I00 (ERDF A way of making Europe).



\begin{thebibliography}{99}

\bibitem{AronJMPA1957} N. Aronszajn, 
\emph{A unique continuation theorem for solutions of elliptic partial differential equations or inequalities of second order,}
 J. Math. Pures Appl. (9) \textbf{36} (1957), 235--249.


\bibitem{AKS} N. Aronszajn, A. Krzywicki, and J. Szarski, {\emph{A unique continuation theorem for exterior differential forms on Riemannian manifolds},
Ark. Mat. 4 (1962), 417--453.}



\bibitem{BGLV} J. Bellazzini, V. Georgiev, E. Lenzmann, and N. Visciglia, \emph{On Traveling Solitary Waves and Absence of Small Data Scattering for Nonlinear Half-Wave Equations}, Comm. Math. Phys. {\bf 372} (2019), no. 2, 713--732

\bibitem{Be}  T. B. Benjamin, \emph{Internal waves of permanent form in fluids of great depth}, J. Fluid Mech. {\bf 29} (1967), 559--592.
{\bf 31} (2013) 2--53.


\bibitem{BoSm} J. L. Bona and R. Smith, {\emph The initial-value problem for the Korteweg-de Vries equation},  Philos. Trans. Roy. Soc. London Ser. A, {\bf 278} (1975) 555--601.

\bibitem{BoRi}  J. P. Borgna and D. F. Rial, \emph{Existence of ground states for a one-dimensional relativistic Schr\"odinger equation},
J. Math. Phys. {\bf 53} (2012), 062301, 19 pages.


\bibitem{CaSir}X. Cabr\'e and Y. Sire,
\emph{Nonlinear equations for fractional Laplacians, I: Regularity, maximum principles, and Hamiltonian estimates,}
Ann. Inst. H. Poincar\'e, \textbf{31} (2014), 23--53.

\bibitem{CaSi} L. Caffarelli and L. Silvestre, \emph{An extension problem related to the fractional Laplacian},  Comm. P.D.E. {\bf 32} (2007), no. 7-9, 124--1260


\bibitem{CHHO} Y. Cho, H. Hajaiej, G. Hwang, and T. Ozawa, \emph{On the Cauchy Problem for Fractional Schr\"odinger Equation with Hartree Type Nonlinearity}, Funkcialaj Ekvacioj {\bf 56} (2013), 193--224.



\bibitem{Cordes1956} H. O. Cordes, 
\emph{\"Uber die eindeutige Bestimmtheit der Lösungen elliptischer Differentialgleichungen durch Anfangsvorgaben. (German),}
Nachr. Akad. Wiss. G\"ottingen. Math.-Phys. Kl. IIa. 1956 (1956), 239--258.



\bibitem{CoMoRa} G. Covi, K. M\"onkk\"onen, and J. Railo, \emph{Unique continuation property and Poincar\'e inequality for higher order laplacians with appliactions in inverse problems}, 2020 arXiv 2001.06210.



\bibitem{davis}E. B. Davies, {\em Spectral Theory and Differential Operators}, Cambridge University Press, 1995.




\bibitem{DoFef1988} H. Donnelly and C. Fefferman, 
\emph{Nodal sets of eigenfunctions on Riemannian manifolds,}
Invent. Math. 93 (1988), no. 1, 161--183.


\bibitem{ElSc} A. Elgart and B. Schlein, {\em Mean field dynamics of boson stars}, Comm. Pure Appl. Math. {\bf 60} (2007), no. 4, 500--545

\bibitem{evans} L. C. Evans, {\em Partial Differential Equations}, Graduate Stuides in Math. vol. 19, American Math. Soc. (1998)


\bibitem{FaFe}  M. M. Fall and V. Felli, \emph {Unique continuation property and local asymptotic of solutions to fractional elliptic equations},  Comm. P.D.E. {\bf 39} (2014), 354--397.


\bibitem{FaFe2}  M. M. Fall and V. Felli, \emph {Unique continuation properties for relativistic Schr\"odinger operators with a singular potential},  Discrete Contin. Dyn. Syst. {\bf 35} (2015), no. 12, 582--5867.

\bibitem{FKS} E. B. Fabes, C. E. Kenig and R. P. Serapioni, 
\emph{The local regularity solutions of degenerate elliptic equations,}
Comm. Partial Differential Equations, \textbf{7} (1982), no. 1, 77--116.

\bibitem{FKU} A. Feizmohammadi, K. Krupchyk, and G. Uhlmann
\emph{Calder\'on problem for fractional Schrödinger operators on closed Riemannian manifolds}, preprint arXiv:2407.16866



\bibitem{FeBeRoRu} A. Fern\' andez-Bertolin, L. Roncal and A. R\"uland, 
\emph{On (global) unique continuation properties of the fractional discrete Laplacian,}
 J. Funct. Anal., \textbf{286} (2024), no. 9, Paper No. 110375, 64 pp.


\bibitem{fitz}  P. Fitzpatrick, \emph{A note on functional calculus for unbounded self-adjoint operators}, J. Fixed Point Theory and Appl. {\bf 13} (2013), 633--640. 

 \bibitem{FrLe} J. Fr\"ohlich and E. Lenzmann, \emph{Blow up for nonlinear wave equations describing boson stars}, Comm. Pure Appl. Math. {\bf 60} (2007), no. 11, 1691--1705.
  
 
  
  \bibitem{GeLePoRa}  P. G\'erard, E. Lenzmann, O. Pocovnicu, and P. Rapha\"el, 
\emph{A two-soliton with transient turbulent regime for the cubic half-wave equation on the real line} ,
Ann. PDE, {\bf 4} (2018), no. 1, Art. 7, 166 pp
  
   
   \bibitem{GhSaUh} T. Ghosh, M. Salo, and G. Uhlmann, \emph{The Calder\'on problem for the fractional Schr\"odinger equation}, Analysis \& PDE. {\bf 13} (2020), 455--475.
  
   \bibitem{GW}  M. Gr\"uter and K-O. Widman, {\em The Green function for uniformly elliptic equations}, Manuscripta Math.{\bf 37} (1982), no. 3, 303--342.




\bibitem{HoPo}   C.  Hong and G. Ponce, {\em  On special properties of solutions to the Benjamin-Bona-Mahony equation}, preprint (2019) arXiv:2311.01578.


  
\bibitem{HoSi}   Y. Hong and Y. Sire, {\em  On fractional Schr\"odinger equations in Sobolev spaces},  Comm. Pure Appl. Anal. {\bf 14} (2015), no. 6, 226--2282.

\bibitem{Jost2011} J. Jost,
{\em Riemannian geometry and geometric analysis. Sixth edition}. Universitext. Springer, Heidelberg, 2011. xiv+611 .




\bibitem{Iz} V. Izakov, {\em Carleman type estimates in anisotropic case and applications}, J. Diff. Eqs. {\bf 105} (1993),  217--238.


 \bibitem{IoPu}  A. Ionescu and F. Pusateri, {\em Nonlinear fractional Schr\"odinger equations in one dimension},  J. Funct. Anal. {\bf 266} (2014), no. 1, 139--176.

  \bibitem{Ka-Po} T. Kato and G. Ponce, {\em Commutator estimates and the Euler and Navier-Stokes equations},
Comm. Pure Appl. Math. {\bf 41} (1988) 891--907.

\bibitem{KN} C. E. Kenig and W.-M. Ni, {\em   On the elliptic equation $Lu-k+Kexp[2u]=0 $}, Ann. Scuola Norm. Sup. Pisa Cl. Sci. (4) 12 (1985), no. 2, 191--224.

\bibitem{KPPV} C. E. Kenig, G. Ponce, D. Pilod, and L. Vega, \emph{On the unique continuation of solutions to non-local non-linear dispersive equations}, Comm. PDE, {\bf 45} (2020),  872--886.

\bibitem{KPV19} C. E. Kenig, G. Ponce, and L. Vega, \emph{Uniqueness Properties of Solutions to the Benjamin-Ono equation and related models},  J. Funct.  Anal. 78 (2020), no. 5, 108396, 14 pp.

\bibitem{Kl} S. Klainerman,  \emph{Global, small
amplitude solutions to nonlinear evolution equations},
Comm. Pure Appl. Math., {\bf 33} (1980) 241--263.






\bibitem{KlPo} S. Klainerman and G. Ponce,  \emph{Global, small
amplitude solutions to nonlinear evolution equations} ,
Comm. Pure Appl. Math., {\bf 36} (1983) 133 -- 141.


\bibitem{KoTa2001} H. Koch and D. Tataru, 
\emph{Carleman estimates and unique continuation for second-order elliptic equations with nonsmooth coefficients,}
Comm. Pure Appl. Math. 54 (2001), no. 3, 339--360.


\bibitem{KdV} D. J. Korteweg and G. de Vries
  \emph{On the change of form of long waves advancing in a
   rectangular canal, and on a new type of long stationary waves}, 
  Philos. Mag. 
 { \bf 39}
  (1895), 
  422--443.
  
\bibitem{KLR}  J.  Krieger, E. Lenzmann, and P. Rapha\"el, \emph{Nondispersive solutions to the $L^2$-critical half-wave equation},  Arch. Ration. Mech. Anal. {\bf 209} (2013) no. 1, 6--129.

\bibitem{La} N. Laskin, {\em Fractional Schr\"odinger equation}, Physical Rev. E. {\bf 66} (2002), no. 5, 056108, 7pp.


\bibitem{Le}  E. Lenzmann, {\em Well-posedness for semi-relativistic Hartree equations of critical type}, Math. Phys. Anal. Geom. {\bf 10} (2007), 43--62.


\bibitem{LD} D. Li, {\em On Kato-Ponce and fractional Leibniz},  Rev. Mat Iberam. {\bf 35} (2019), 
23--100.


\bibitem{LiPo} F. Linares, and G. Ponce, {\em Unique Continuation Properties for solutions to the Camassa-Holm equation and other non-local equations},  Proc. A.M.S.
{bf 148} (2020), 3871--3879.







\bibitem{LiPo1} F. Linares and G. Ponce, \emph{Introduction to nonlinear dispersive equations}, Second edition. Universitext. Springer, New York, 2015.








\bibitem{On} H. Ono,  \emph{Algebraic solitary waves in stratified fluids}, Journal  Physical Society of Japan, {\bf 39} (4) (1975), 108--1091.


\bibitem{Pe} J. Pedlosky, \emph{Geophysical Fluid Dynamics},  (New York Springer) (1987) 345--368, 653--670.


\bibitem{re-si-I} M. Reed and B. Simon, \emph{Methods of Modern Mathematical Physics, vol. 1}, Academic Press, 1972.


\bibitem{re-si-II} M. Reed and B. Simon, \emph{Methods of Modern Mathematical Physics, vol. 2}, Academic Press, 1972.

\bibitem{Ri} M. Riesz, \emph{Integrales de Riemann-Liouville et potentiels}, Acta Sci. Math. (Szeged) 9:1-1 (1938-40), 1--42.


\bibitem{rudin}  W. Rudin \emph{Functional Analysis},  McGraw-Hill Book Co. 1973.



\bibitem{Ru1} A. R\"uland, 
\emph{Unique continuation for fractional Schr\"odinger equations with rough potentials},
 Comm. Partial Differential Equations, \textbf{40} (2015), no. 1, 77--114.


 \bibitem{RulandTAMS2016} A. R\"uland, 
\emph{ On quantitative unique continuation properties of fractional Schr\"odinger equations: doubling, vanishing order and nodal domain estimates,}
 Trans. Amer. Math. Soc. 369 (2017), no. 4, 2311--2362.




\bibitem{SaSc} J-C. Saut and B. Scheurer, {\em Unique continuation for some evolution equations}, J. Diff. Eqs. {\bf 66} (1987), 118--139.

\bibitem{Sh} J. Shatah, {\em Global small solutions to nonlinear evolution equations}, J. Diff. Eqs. {\bf 46} (1982), 313--327.



\bibitem{StiTorCPDE2010} P. R. Stinga, J. L. Torrea, 
\emph{Extension problem and Harnack's inequality for some fractional operators},
 Comm. Partial Differential Equations, \textbf{35} (2010), no. 11, 2092--2122.
 


\bibitem{St} E. M. Stein, \emph{Harmonic Analysis : Real-Variable Method, Orthogonality, and Oscillatory Integrals},  Princeton University Press, Princeton, New Jersey (1993)


\bibitem{stri} R. S. Strichartz, 
\emph{Analysis of the Laplacian on the complete Riemannian manifold}, J. Funct. Aanl. {\bf 52} (1983) 48--79.


\bibitem{Yo}  K. Yosida \emph{Functional Analysis}, (third edition) Springer-Verlag  1971.


\bibitem{Yu} H. Yu, \emph{Unique continuation for fractional orders of elliptic equations},  Ann. PDE {\bf 3} (2017), no. 2, Art. 16.

\end{thebibliography}
\end{document}